\documentclass[12pt, reqno]{amsproc}
\usepackage[T2A]{fontenc}
\usepackage[utf8]{inputenc}
\usepackage[ukrainian, english]{babel}
\usepackage{graphicx}
\usepackage{amsmath,amsfonts,amssymb}
\usepackage{dutchcal}
\usepackage{mathtools}
\usepackage{geometry}
\usepackage{hyperref}
\hypersetup{hidelinks=true, hypertexnames=false}
\hypersetup{
   hypertexnames=false,
   colorlinks,
   linkcolor=RawSienna, 
   citecolor={blue!50!black},
   urlcolor={blue!80!black}
}

\allowdisplaybreaks

\usepackage{graphicx}
\usepackage{multicol,multirow}
\usepackage{amsmath,amssymb,amsfonts}
\usepackage{mathrsfs}
\usepackage{enumitem}
\usepackage[all]{xy}
\usepackage[svgnames, dvipsnames]{xcolor}

\usepackage{hyperref}

\newcommand\mycolor[1]{}

\setlist[enumerate]{itemsep=0.3ex, topsep=0.3ex, label={\rm(\arabic*)}}
\setlist[itemize]{itemsep=0.3ex, topsep=0.3ex, leftmargin=4ex}

\newtheorem{subtheorem}[subsubsection]{Theorem}
\newtheorem{sublemma}[subsubsection]{Lemma}

\newtheorem{subcorollary}[subsubsection]{Corollary}
\newtheorem{subremark}[subsubsection]{Remark}
\newtheorem{subexample}[subsubsection]{Example}
\newtheorem{subdefinition}[subsubsection]{Definition}

\newtheorem{claim}{Claim}

\makeatletter
\@addtoreset{subsection}{section}
\@addtoreset{equation}{section}
\@addtoreset{figure}{section}
\@addtoreset{table}{section}
\makeatother

\newcommand\term[1]{\textsf{\textit{#1}}}

\newcommand\monoArrow{\lhook\joinrel\rightarrow}

\newcommand\der[3][\empty]{\ifx\empty#1#2\else\left(#2\right)\fi'_{#3}}
\newcommand\dertwo[4][\empty]{\ifx\empty#1#2\else\left(#2\right)\fi''_{#3#4}}

\newcommand\id{\mathrm{id}}          
\newcommand\Int{\mathrm{Int}}        
\newcommand\supp{\mathrm{supp\,}}    
\newcommand\rank{\mathsf{rank}}      
\newcommand\restr[2]{#1\vert_{#2}}                   
\newcommand\restrb[2]{\restr{\left(#1\right)\!}{#2}}   
\newcommand\nabs[1]{\vert#1\vert}    

\newcommand\GL{\mathrm{GL}}

\newcommand\GLRn[1]{\GL(#1,\bR)}
\newcommand\Aut{\mathrm{Aut}}       
\newcommand\Diff{\mathcal{D}}       
\newcommand\Homeo{\mathcal{H}}      
\newcommand\End{\mathcal{E}}        
\newcommand\Emb{\mathrm{Emb}}       
\newcommand\Mat[2]{\mathsf{M}(#1,#2)}        

\newcommand\amatr[4]{\left(\!\begin{smallmatrix}#1\ &#2\\[0.5mm] #3\ &#4 \end{smallmatrix}\!\right)}
\newcommand\AMATR[4]{\left(\begin{matrix}#1\ &#2\\[0.5mm] #3\ &#4 \end{matrix}\right)}

\newcommand\calU{\mathcal{U}}
\newcommand\calV{\mathcal{V}}

\newcommand\calX{\mathcal{X}}

\newcommand\Aman{A}
\newcommand\Bman{B}

\newcommand\Lman{L}

\newcommand\Nman{N}

\newcommand\Rman{R}

\newcommand\Tman{T}
\newcommand\Uman{U}
\newcommand\Vman{V}
\newcommand\Wman{W}
\newcommand\Xman{X}
\newcommand\Yman{Y}

\newcommand\Jmatr{J}

\newcommand\Pmatr{P}
\newcommand\Qmatr{Q}
\newcommand\Rmatr{R}
\newcommand\Smatr{S}

\newcommand\Xmatr{X}

\newcommand\bR{\mathbb{R}}

\newcommand\XX{\mathcal{X}}
\newcommand\YY{\mathcal{Y}}

\newcommand\Amap{A}

\newcommand\SECT[1]{\S\ref{#1}}
\newcommand\NEAT{{neat}}
\newcommand\GOOD{{good}}
\newcommand\MGENERIC[2]{{$(#1,#2)$-generic}}
\newcommand\FOLNEAT{{neat}}

\newcommand\TOP{{3}}  

\newcommand\acol[1]{{#1}}
\newcommand\bcol[1]{{#1}}

\newcommand\AManif{\acol{A}}
\newcommand\ATotMan{\acol{E}}
\newcommand\ASingMan{\acol{X}}
\newcommand\Apr{\acol{p}}
\newcommand\Afib[1]{\ATotMan_{#1}}
\newcommand\AFol{\acol{\mathcal{F}}}
\newcommand\AFolPr{\acol{\alpha}}
\newcommand\AopBs{\acol{\Vman}}   
\newcommand\aleaf{\acol{\mathcal{a}}}

\newcommand\BManif{\bcol{B}}
\newcommand\BTotMan{\bcol{F}}
\newcommand\BSingMan{\bcol{Y}}
\newcommand\Bpr{\bcol{q}}
\newcommand\Bfib[1]{\BTotMan_{#1}}
\newcommand\BFol{\bcol{\mathcal{G}}}
\newcommand\BFolPr{\bcol{\beta}}

\newcommand\bleaf{\bcol{\mathcal{b}}}

\newcommand\CManif{\bcol{C}}
\newcommand\DManif{\bcol{D}}

\newcommand\apnt{\acol{x}}  
\newcommand\bpnt{\bcol{y}}  
\newcommand\avct{\acol{u}}  

\newcommand\axdim{\acol{\mathsf{b}}}
\newcommand\ardim{\acol{\mathsf{m}}}
\newcommand\bxdim{\bcol{\mathsf{c}}}
\newcommand\brdim{\bcol{\mathsf{n}}}

\newcommand\ALocBase{\acol{\bR^{\axdim}_{+}}}
\newcommand\ALocFiber{\acol{\bR^{\ardim}}}
\newcommand\BLocBase{\bcol{\bR^{\bxdim}_{+}}}
\newcommand\BLocFiber{\bcol{\bR^{\brdim}}}

\newcommand\ALocBaseFiber{\ALocBase\times\ALocFiber} 

\newcommand\AScalNbh{\acol{\Vman}}
\newcommand\BScalNbh{\bcol{\Wman}}

\newcommand\ahomot[1]{\acol{\theta_{#1}}}
\newcommand\bhomot[1]{\bcol{\eta_{#1}}}

\newcommand\tang[2][\empty]{\mathsf{T}\ifx\empty#1\relax\else_{\!#1}\fi#2}
\newcommand\vtang[2][\empty]{\mathsf{V}\ifx\empty#1\relax\else_{\!#1}\fi#2}  
\newcommand\tbPrj[1]{\pi_{#1}}  

\newcommand\AVect{\vtang[\ASingMan]{\ATotMan}}
\newcommand\BVect{\vtang[\BSingMan]{\BTotMan}}

\newcommand\glHom[1][\empty]{\lambda^{#1}}
\newcommand\glHomt[1][\empty]{\lambda^{#1}_{\pt}}

\newcommand\hprm{{\sigma}} 
\newcommand\bhprm{\bar{\hprm}}

\newcommand\hpFunc{\mathbf{d}}

\newcommand\hgeps{{\eps}}

\newcommand\boundQ{{q}}
\newcommand\boundR{{\delta}}

\newcommand\constC{{C}}

\newcommand\multByConst[2]{\ifdim#1pt=1pt \else#1\fi#2}

\newcommand\aConst{0.2}
\newcommand\bConst{0.8}
\newcommand\cConst{1}
\newcommand\aSqConst{0.04}
\newcommand\bSqConst{0.64}

\newcommand\genNbhx[1]{\Tman_{#1}}
\newcommand\aNbhx[1]{\Tman_{\multByConst{\aConst}#1}}
\newcommand\bNbhx[1]{\Tman_{\multByConst{\bConst}#1}}
\newcommand\cNbhx[1]{\Tman_{\multByConst{\cConst}#1}}
\newcommand\aNbh{\aNbhx{\hprm}}
\newcommand\bNbh{\bNbhx{\hprm}}
\newcommand\cNbh{\cNbhx{\hprm}}

\newcommand\XLIN{\mathrm{lin}}
\newcommand\XVRT{\mathsf{v}}
\newcommand\XFIX{\mathrm{fix}}
\newcommand\XNBH{\mathrm{nb}}
\newcommand\XLOC{\mathrm{loc}}

\newcommand\LetterGeneral[4][\empty]{#4_{#2}^{#3}\ifx\empty #1\relax\else(#1)\fi}

\newcommand\LetterLin[2][\empty]{\LetterGeneral[#1]{}{\XLIN}{#2}}

\newcommand\LetterLoc[2][\empty]{\LetterGeneral[#1]{\XLOC}{}{#2}}

\newcommand\LetterFix[2][\empty]{\LetterGeneral[#1]{\XFIX}{}{#2}}
\newcommand\LetterFixLin[2][\empty]{\LetterGeneral[#1]{\XFIX}{\XLIN}{#2}}

\newcommand\LetterNbh[2][\empty]{\LetterGeneral[#1]{\XNBH}{}{#2}}

\newcommand\LetterVrt[2][\empty]{\LetterGeneral[#1]{\XVRT}{}{#2}}

\newcommand\LetterRestr[4][\empty]{\LetterGeneral[#1]{#2,#3}{}{#4}}
\newcommand\LetterRestrLin[4][\empty]{\LetterGeneral[#1]{#2,#3}{\XLIN}{#4}}

\newcommand\LetterVDer[3][\empty]{\LetterGeneral[#1]{#2,\mathsf{v}}{}{#3}}
\newcommand\LetterVDerLin[3][\empty]{\LetterGeneral[#1]{#2,\mathsf{v}}{\XLIN}{#3}}

\newcommand\LetterNeat[4][\empty]{\LetterGeneral[#1]{(#2,#3)}{}{#4}}
\newcommand\LetterNeatLin[4][\empty]{\LetterGeneral[#1]{(#2,#3)}{\XLIN}{#4}}

\newcommand\LetterFol[3][\empty]{\LetterGeneral[#1]{#2,\mathrm{fol}}{}{#3}}

\newcommand\LetterSelfFol[2][\empty]{\LetterGeneral[#1]{\mathrm{fol}}{}{#2}}
\newcommand\LetterVFol[2][\empty]{\LetterGeneral[#1]{\mathsf{vfol}}{}{#2}}

\newcommand\mAB{\mathcal{A}}
\newcommand\mABlin{\LetterLin{\mAB}}
\newcommand\mABvrt{\LetterVrt{\mAB}}
\newcommand\mABnt[2]{\LetterNeat{#1}{#2}{\mAB}}
\newcommand\mABntlin[2]{\LetterNeatLin{#1}{#2}{\mAB}}
\newcommand\mABh[1]{\LetterRestr{#1}{\ASingMan}{\mAB}}
\newcommand\mABhlin[1]{\LetterRestrLin{#1}{\ASingMan}{\mAB}}
\newcommand\mABv[1]{\LetterVDer{#1}{\mAB}}
\newcommand\mABvlin[1]{\LetterVDerLin{#1}{\mAB}}

\newcommand\hmap{g}     
\newcommand\vmap{\xi}   

\newcommand\mNF{\mathcal{L}}
\newcommand\mNFvrt{\LetterVrt{\mNF}}
\newcommand\mNFlin{\LetterLin{\mNF}}
\newcommand\mNFnt[2]{\LetterNeat{#1}{#2}{\mNF}}
\newcommand\mNFntlin[2]{\LetterNeatLin{#1}{#2}{\mNF}}
\newcommand\mNFh[1]{\LetterRestr{#1}{\ASingMan}{\mNF}}

\newcommand\mNFv[1]{\LetterVDer{#1}{\mNF}}

\newcommand\mNB{\mathcal{N}}

\newcommand\mNBlin{\LetterLin{\mNB}}
\newcommand\mNBnt[2]{\LetterNeat{#1}{#2}{\mNB}}

\newcommand\folmap{\phi}
\newcommand\mABfol[1]{\LetterFol{#1}{\mAB}}

\newcommand\LEF{\mathcal{Hom}(\ATotMan,\BTotMan)}

\newcommand\ImmAB{\mathcal{I}\!m^{\mathrm{nt}}}
\newcommand\SubAB{\mathcal{S}\!b}
\newcommand\dAB{\Diff}
\newcommand\dABloc{\LetterLoc{\dAB}}

\newcommand\DiffAXBY{\Diff_{\ASingMan,\BSingMan}}
\newcommand\DiffAXBYlin{\Diff^{\XLIN}_{\ASingMan,\BSingMan}}

\newcommand\EMNnbh[1]{\mABvlin{#1}}

\newcommand\ANbhSingMan[1][\ASingMan]{\Vman_{#1}}
\newcommand\BNbhSingMan[1][\BSingMan]{\Wman_{#1}}

\newcommand\EndVFol[3][\empty]{\LetterVFol[#1]{\End}}

\newcommand\eAX{\End}
\newcommand\eAXfix{\LetterFix{\eAX}}
\newcommand\eAXvrt{\LetterVrt{\eAX}}
\newcommand\eAXnt[2]{\LetterNeat{#1}{#2}{\eAX}}
\newcommand\eAXvid{\LetterVDer{\id}{\eAX}}
\newcommand\eAXL{\LetterLin{\eAX}}
\newcommand\eAXLfix{\LetterFixLin{\eAX}}
\newcommand\eAXnb{\LetterNbh{\eAX}}
\newcommand\eAXLnt[2]{\LetterNeatLin{#1}{#2}{\eAX}}

\newcommand\eAXfol{\LetterSelfFol{\eAX}}
\newcommand\eAXvfol{\EndVFol{\AManif}{\ASingMan}}

\newcommand\dAX{\Diff}
\newcommand\dAXfix{\LetterFix{\dAX}}
\newcommand\dAXvrt{\LetterVrt{\dAX}}

\newcommand\dAXvid{\LetterVDer{\id}{\dAX}}
\newcommand\dAXL{\LetterLin{\dAX}}
\newcommand\dAXLfix{\LetterFixLin{\dAX}}
\newcommand\dAXnb{\LetterNbh{\dAX}}

\newcommand\dAXfol{{\LetterSelfFol{\dAX}}}

\newcommand\phifunc{{\theta}}
\newcommand\Phimap[1][\empty]{{\Theta_{#1}}}

\newcommand\Scomp{\Lman}

\newcommand\STopol[2]{\mathsf{S}^{#1}(#2)} 
\newcommand\WTopol[2]{\mathsf{W}^{#1}(#2)} 

\newcommand\eps{\varepsilon}                 
\newcommand\trans[1]{#1^{t}}

\newcommand\Cr[1]{\mathcal{C}^{#1}}
\newcommand\Cinfty{\mathcal{C}^{\infty}}
\newcommand\Crm[3]{\Cr{#1}(#2,#3)}
\newcommand\Ci[2]{\mathcal{C}^{\infty}(#1,#2)}               
\newcommand\Wtop{\mathsf{W}}
\newcommand\Stop{\mathsf{S}}
\newcommand\Wtr[1]{\Wtop^{#1}}
\newcommand\Str[1]{\Stop^{#1}}

\newcommand\func{f}

\newcommand\dif{h}
\newcommand\gdif{g}

\newcommand\ps{s}

\newcommand\pt{t}

\newcommand\px{x}
\newcommand\py{y}

\newcommand\pu{u}
\newcommand\pv{v}
\newcommand\pw{w}

\newcommand\pvi[1]{\pv_{#1}}
\newcommand\pxi[1]{\px_{#1}}

\newcommand\pyi[1]{\py_{#1}}

\newcommand\pr{\mathrm{pr}}

\newcommand\rr{r}

\newcommand\Ahom[1][\empty]{{\Lambda}\ifx#1\empty\relax\else^{#1}\fi}    
\newcommand\Chom[1][\empty]{{H}\ifx#1\empty\relax\else^{#1}\fi}    
\newcommand\Dhom[1][\empty]{{\Delta}\ifx#1\empty\relax\else^{#1}\fi}    

\newcommand\Chombt[1][\empty]{\Chom[#1]_{\hprm,\pt}}
\newcommand\Dhombt[1][\empty]{\Dhom[#1]_{\hprm,\pt}}

\newcommand\Ghom{G}   

\newcommand\STL{{star-shaped}}

\newcommand\VERTINJECTIVE[1][\ASingMan]{{$\mathsf{V}$-injective}}
\newcommand\FLPMap[1][$\AFol$] {{#1-preserving}}
\newcommand\FLPBij[1][$\AFol$] {{#1-automorphism}}

\newcommand\ftfol[1][\folmap]{$#1$-foliated}

\newcommand\VBAut[2][\empty]{\GL(#2\ifx\empty #1\relax\else,#1\fi)}

\newcommand\DiffLP{\Diff}  
\newcommand\DiffLPInv[3][\empty]{\DiffLP_{\mathrm{inv}}(#2,#3\ifx\empty#1\relax\else,#1\fi)}

\newcommand\DiffLPFix[3][\empty]{\DiffLP_{\mathrm{fix}}(#2,#3\ifx\empty#1\relax\else,#1\fi)}

\newcommand\DiffLPNb[3][\empty]{\DiffLP_{\mathrm{nb}}(#2,#3\ifx\empty#1\relax\else,#1\fi)}

\newcommand\EndLP{\End}  
\newcommand\EndLPInv[3][\empty]{\EndLP_{\mathrm{inv}}(#2,#3\ifx\empty#1\relax\else,#1\fi)}

\newcommand\EndLPFix[3][\empty]{\EndLP_{\mathrm{fix}}(#2,#3\ifx\empty#1\relax\else,#1\fi)}

\newcommand\EndLPNb[3][\empty]{\EndLP_{\mathrm{nb}}(#2,#3\ifx\empty#1\relax\else,#1\fi)}

\newcommand\Nbh[4]{\mathcal{N}^{#2}_{#3,#4}(#1)}  

\newcommand\BComp{K}

\newcommand\aafunc{\acol{a}}
\newcommand\bbfunc{\acol{b}}

\newcommand\haafunc{\hat{\aafunc}}
\newcommand\hbbfunc{\hat{\bbfunc}}

\newcommand\nrm[1]{\|#1\|}
\newcommand\fnrm[1][\empty]{\nu\ifx\empty#1\relax\else(#1)\fi}
\newcommand\xnrm[3]{\|#1\|_{#2,#3}}

\newcommand\EChart{\mathsf{\phi}}

\newcommand\dvdif[2][\empty]{{\mycolor{Green}\alpha\ifx\empty#1}}

\newcommand\Circle{S^1}

\newcommand\ScalNbh{\Nman}

\newcommand\CSp[3][\empty]{\mathcal{E}^{\ASingMan\BSingMan}_{#1}(#2,#3)}

\newcommand\CMaxRankSp[3][\empty]{\CSp[mr\ifx\empty#1\relax\else,#1\fi]{#2}{#3}}

\newcommand\vrank[2][\empty]{\mathsf{vr}_{#1}(#2)} 
\newcommand\hrank[2][\empty]{\mathsf{hr}_{#1}(#2)} 
\newcommand\trank[2][\empty]{\mathsf{rank}_{#1}(#2)}  

\newcommand\JM[2][\empty]{J_{#2}\ifx#1\empty\relax\else(#1)\fi}  

\newcommand\EInv[1]{\Apr^{-1}(#1)}

\newcommand\matrixBlock[4][\empty]{%
    {#2}\ifx#3\empty\relax\else^{#3}\fi
    \ifx#1\empty_{#4}\else_{#1\ifx#4\empty\relax\else,#4\fi}\fi}

\newcommand\mP[1][\empty]{\Pmatr\ifx#1\empty\relax\else_{#1}\fi}
\newcommand\mQ[1][\empty]{\Qmatr\ifx#1\empty\relax\else_{#1}\fi}
\newcommand\mR[1][\empty]{\Rmatr\ifx#1\empty\relax\else_{#1}\fi}
\newcommand\mS[1][\empty]{\Smatr\ifx#1\empty\relax\else_{#1}\fi}
\newcommand\mX[1][\empty]{\Xmatr\ifx#1\empty\relax\else_{#1}\fi}

\newcommand\uJmatr{\tilde{J}}
\newcommand\uPmatr{\tilde{P}}
\newcommand\uQmatr{\tilde{Q}}
\newcommand\uRmatr{\tilde{R}}
\newcommand\uSmatr{\tilde{S}}

\newcommand\matrTwoThree[6]{%
    \left(\!\begin{smallmatrix}
        #1\ &#2 \ &#3 \\[0.5mm] 
        #4\ &#5 \ &#6 
    \end{smallmatrix}\!\right)}

\newcommand\matrThreeTwo[6]{%
    \left(\!\begin{smallmatrix}
        #1\ &#2 \\[0.5mm] 
        #3\ &#4 \\[0.5mm] 
        #5\ &#6 
    \end{smallmatrix}\!\right)}

\newcommand\rnk[2]{\rank_{#1}\ifx#2\empty\relax\else(#2)\fi}

\title{Foliated and compactly supported isotopies of tubular neighborhoods} 

\author{Oleksandra Khokhliuk}
\email{khokhliyk@gmail.com}
\address{Liceum ``Educator'', Yelyzavety Chavdar Str., 11a, Kyiv, 02140, Ukraine}

\author{Sergiy Maksymenko}
\email{maks@imath.kyiv.ua}
\address{Algebra and Topology Department, Institute of Mathematics NAS of Ukraine, Tereshchenkivska Str., 3, Kyiv, 01024, Ukraine}

\keywords{Foliation, diffeomorphism, homotopy type, vector bundle}
\subjclass[2020]{57R30, 
           57T20 
}

\begin{document}
\begin{abstract}
Let $\mathcal{F}$ be a foliation with a ``singular'' submanifold $\Sigma$ on a smooth manifold $A$ and $p:E \to \Sigma$ be a tubular neighborhood of $\Sigma$ in $A$.
Under certain ``scalability'' assumptions on $\mathcal{F}$ near $\Sigma$, we prove that every leaf preserving diffeomorphism $h$ of $A$ is isotopic via a leaf preserving isotopy to a diffeomorphism which coincides with some vector bundle morphism of $E$ near $\Sigma$.
This result is simultaneously a foliated, compactly supported and parametrized variant of a well known statement that every diffeomorphism $h$ of $\mathbb{R}^n$ fixing the origin is isotopic to the linear isomorphism induced by its Jacobi matrix at $0$.
Particular cases of such a ``linearization theorem'' allowed recently to compute homotopy types of leaf preserving diffeomorphism groups of polar Morse-Bott foliations on lens spaces and the twisted bundle $S^2\tilde{\times}S^1$.

In general, that ``linearization theorem'' can be used to reduce the study of the homotopy type of the group of leaf preserving diffeomorphisms of $\mathcal{F}$ to an investigation of the following three groups (being independent in a certain sense): (a)~diffeomorphism group of $\Sigma$, (b)~the group of leaf preserving vector bundle isomorphisms of $E$, and (c)~the group of leaf preserving diffeomorphisms of $A$ supported in $A\setminus \Sigma$, i.e.\ out of the ``singularity'' $\Sigma$.
\end{abstract}

\maketitle
\section{Introduction}\label{sect:introduction}
Let $\AManif$ be a smooth ($\Cinfty$) manifold and $\AFol$ be some partition of $\AManif$ whose elements will be called \term{leaves}.
For instance, $\AFol$ can be a usual foliation, or a foliation with ``singular leaves'' (having smaller dimension than a typical leaf).
This paper is motivated by the study of the homotopy type of the group $\dAXfol$ of \FLPMap\ diffeomorphisms of $\AManif$, that is, diffeomorphisms that map each leaf of $\AFol$ to itself.
We are particularly interested in Morse-Bott foliations, which are fundamental to various areas of mathematics, including symplectic and Poisson geometry.

Even for non-singular foliations, the above problem of understanding the homotopy type of $\dAXfol$ is highly complex, with explicit computations limited to a few specific examples.
Most of the existing results focus on establishing the perfectness (triviality of first homology) of such groups in the non-singular setting, extending the foundational work of Herman~\cite{Herman:CR:1971}, Thurston~\cite{Thurston:BAMS:1974}, Mather~\cite{Mather:Top:1971, Mather:BAMS:1974}, Epstein~\cite{Epstein:CompMath:1970} on the simplicity of certain diffeomorphism groups.
See, for instance, \cite{Fukui:JMKU:1980, HallerTeichmann:AGAG:2003, AbeFukui:CEJM:2005, Rybicki:AMB:2019} for further developments and references.

On the other hand, for singular foliations there is much less information.
It is worth mentioning the works of Fukui and Ushiki~\cite{FukuiUshiki:JMKU:1975} and Fukui~\cite{Fukui:JJM:1976} computing homotopy types of $\dAXfol$ for non-singular foliations on compact $3$-manifolds with finitely many generalized Reeb components.
Also, in~\cite{Maksymenko:TA:2003, Maksymenko:OsakaJM:2011} the second author computed the homotopy types of $\dAXfol$ for singular foliations by orbits of vector fields linearizable near their zeros.
This made it possible to describe the homotopy types of diffeomorphism groups preserving smooth functions on compact surfaces having ``homogeneous'' critical points which includes all Morse functions, \cite{Maksymenko:AGAG:2006, Maksymenko:UMZ:ENG:2012}.

Let $\ASingMan$ be a submanifold of $\AManif$.
Fix any \term{tubular} neighborhood $\Apr\colon\ATotMan\to\ASingMan$ of $\ASingMan$, so $\ATotMan$ is an open neighborhood of $\ASingMan$ in $\AManif$ and $\Apr$ is a smooth retraction with a fixed structure of a vector bundle.
Suppose also that $\ASingMan$ is a leaf (i.e.\ an element) of $\AFol$.
The main result of the present paper claims that under certain ``scalability'' assumption on $\AFol$ near $\ASingMan$, \term{each diffeomorphism $\dif\in\dAXfol$ is isotopic in $\dAXfol$ to a diffeomorphism $\dif_0$ being ``linear on fibers of $\ATotMan$'' near $\ASingMan$}.
Even more, that isotopy can be made continuous in $\dif$ and this gives a deformation of the group $\dAXfol$ into its subgroup consisting of diffeomorphisms ``linearized'' near $\ASingMan$, see Theorems~\ref{th:linearization:diff_pres_submanif} and~\ref{th:foliated_linearization:formulation} for precise formulations.
We will refer to a statement of this kind as a \term{linearization theorem}.

In fact we will obtain more general results about linearization of maps between pairs of manifolds, see the discussion in \SECT{sect:comparison_with_current_paper} and Theorems~\ref{th:linearization_EMN_and_subspaces}, \ref{th:linearization_DiffMN}, \ref{th:linearization_fol}.

Some applications of linearization theorems were given in~\cite{KhokhliukMaksymenko:JHRS:2023, Maksymenko:ProcIMNANUS:2023, Maksymenko:JHRS:2024} for computing the homotopy types of the diffeomorphism group of Morse-Bott foliations with only extreme singular circles on the solid torus $D^2\times\Circle$, solid Klein bottle, all lens spaces, and twisted $S^2$-bundle over $\Circle$.
They were established by geometric methods using simpler particular cases of the linearization theorem, see~\cite[Theorem~3.12]{KhokhliukMaksymenko:JHRS:2023}.
Further, in~\cite{Maksymenko:VBAuto:2024}, they were also reproved via ``homotopy theory'' methods, assuming that the linearization theorems from the present paper hold.

Owing to its length, the present article does not contain applications.
They will be given in subsequent papers.

\paragraph{\bf Structure of the paper}
In \SECT{sect:known_linearization_results} we will review various known linearization results and in \SECT{sect:comparison_with_current_paper} compare them with what is done in our paper.
\SECT{sect:manif__reg_nbh} contains the definition of \term{good} submanifolds and their tubular neighborhoods, \SECT{sect:Whitney_topologies} recalls weak and strong Whitney topologies, and \SECT{sect:vert_derivative} discusses the notion of vertical derivative.

In \SECT{sect:linearization_theorems} we formulate our main linearization Theorem~\ref{th:linearization_EMN_and_subspaces} for maps of pairs and its foliated counterpart~\ref{th:linearization_fol}.
We also deduce from them Theorem~\ref{th:linearization_DiffMN} on linearization of diffeomorphisms.

In \SECT{sect:self_maps_of_a_pair} we reformulate the results from \SECT{sect:linearization_theorems} for self-maps of a manifold $\AManif$ preserving a submanifold $\ASingMan$.

\SECT{sect:global_linearization} is devoted to the proof of Theorem~\ref{th:linearization_EMN_and_subspaces} (\SECT{sect:proof_th:linearization_EMN_and_subspaces}) and its foliated variant Theorem~\ref{th:linearization_fol} (\SECT{sect:proof:th:linearization_fol}).
These proofs are done modulo one technical Lemma~\ref{lm:estim_PQRS_var} which is then proved in \SECT{sect:proof:lm:estim_PQRS_var}.

\subsection{Known linearization results}\label{sect:known_linearization_results}
It is well known that every $\Cr{k}$-diffeomorphism $\dif\colon(\bR^n,0)\to(\bR^n,0)$, $k\geq1$,  is isotopic to a linear isomorphism $\bR^n\to\bR^n$ induced by the Jacobi matrix $J$ of $\dif$ at $0$, via the $\Cr{k-1}$ isotopy $\glHom[\dif]\colon[0;1]\times\bR^n\to\bR^n$ given by the following explicit formula:
\begin{equation}\label{equ:h_isot_jacobi_matrix}
    \glHom[\dif](\pt,\px) =
    \begin{cases}
        \tfrac{1}{\pt}\dif(\pt \px), & \pt\in(0;1],\\
        J\px,                        & \pt=0.
    \end{cases}
\end{equation}
Thus $\glHom_{\pt}$, for $\pt\in(0;1]$, is a conjugation of $\dif$ with homotheties $\px\mapsto\pt\px$, and due to the Hadamard lemma, the limit $\glHom_{0}$ as $\pt\to0$ exists and the whole map $\glHom[\dif]$ is $\Cr{k-1}$.
One can regard $\glHom[\dif]$ as a ``linearization'' process of $\dif$.
Moreover, the above formula gives in fact a strong deformation retraction of the group $\Diff^{\infty}(\bR^n,0)$ of $\Cinfty$ diffeomorphisms of $\bR^n$ fixing the origin onto the group $\GLRn{n}$ with respect to the weak Whitney topology $\Wtr{\infty}$.

This idea of linearization of some map by conjugating it with homothety-like maps (often called \emph{dilation} or \emph{zooming}) appears in mathematics in various forms and contexts.
\begin{enumerate}[wide, label={\rm\arabic*.}, nosep]
\newcommand\Ag{\AManif}
\newcommand\Bg{\BManif}
\newcommand\Aalg{\mathfrak{a}}
\newcommand\Balg{\mathfrak{b}}
\newcommand\astep{r}
\newcommand\hmth[2]{\theta_{#1,#2}}
\newcommand\delat[2]{\delta_{#1,#2}}
\newcommand\pans[2]{D_{#1}#2}
\newcommand\glh[3][\dif]{\glHom[#1]_{#2,#3}}

\item
In a certain sense~\eqref{equ:h_isot_jacobi_matrix} it is a counterpart of the famous Alexander trick implying that the group $\Homeo(D^n,\partial D^n)$ of homeomorphisms of $n$-disk fixed on its boundary is contractible, \cite{Alexander:PNAS:1923}.

\item
In~\cite{Kister:AM:1964} Kister, using similar conjugations, proved that there is a deformation of the space of topological embeddings $\Emb((\bR^n,0),(\bR^n,0))$ fixing the origin onto the group $\Homeo(\bR^{n}, 0)$ of self homeomorphisms of $\bR^{n}$.

\item One can regard $\bR^{n}$ as a tubular neighborhood of a submanifold $\{0\}$.
Then a parametric variant of that formula is an important theorem that any two tubular neighborhoods of a submanifold $\ASingMan$ in a manifold $\AManif$ are isotopic, see, e.g.~\cite[Chapter~4, Theorem~5.3]{Hirsch:DiffTop}.

\item
Let $(\AManif,\rho)$ be a connected Riemannian manifold.
Then a point $p\in \AManif$ is called \term{pole} if the exponential map $\exp_p\colon T_{p}\AManif \to \AManif$ is a diffeomorphism, see e.g.\ \cite{Ambrose:AnM:1956,GreeneWu:LNM:1979,Kasue:OMJ:1981}.
In particular, if $\AManif$ is a complete, simply connected manifold of nonpositive sectional curvature (called \term{Cartan-Hadamard}), then every point $p\in \AManif$ is a pole.
For $\pt>0$ define the following ``dilation'' map $\hmth{\AManif}{\pt}\colon(\AManif,p)\to(\AManif,p)$, $\hmth{\AManif}{\pt}(x) = \exp_{p} \bigl(\pt \exp^{-1}_{p}(x)\bigr)$.
Now, given a pair of manifolds with poles $(\AManif,p)$ and $(\BManif,q)$ and a smooth map $\dif\colon(\AManif,p) \to (\BManif,q)$ one can define a ``linearization'' homotopy
\[
    \glHom[\dif]\colon[0;1]\times\AManif\to\BManif,
    \qquad
    \glHom[\dif]_{\pt}(x) =
        \begin{cases}
            \hmth{\BManif}{1/\pt} \circ \dif \circ \hmth{\AManif}{\pt},  & \pt\in(0;1] \\
            \exp_{q}\circ\tang[p]{\dif}\circ\exp_{p}^{-1}, & \pt=0.
        \end{cases}
\]
If $\dif$ is a diffeomorphism, then $\glHom[\dif]$ is an isotopy.

\item
Discrete variants of these kinds of transformations with $\pt\to0$ and $\pt\to\infty$ appear in~\cite{Biswas:GFA:2012}.

\item
Recall that a Lie group $\Ag$ is called \term{exponential} or \term{of type (E)}, if its exponential map $\exp\colon \tang[e]{\Ag} \to \Ag$ is a diffeomorphism, \cite{Dixmier:BSMF:1957}.
Then one can define similar deformations $\glHom[\dif]$ for maps $\dif\colon (\Ag,e) \to (\Bg,e)$ between exponential groups sending unit to unit.

\item
In particular, suppose $\Ag$ is a \term{stratified group}, \cite{Folland:ArkM:1975},  called also \term{Carnot}, \cite{Pansu:AM:1989}.
By definition it is a simply connected nilpotent Lie group such that its Lie algebra $\tang[e] \Ag = \Aalg$ splits into a direct sum of subspaces $\Aalg=V_1\oplus\cdots\oplus V_{\astep}$ satisfying $V_{i+1} = [V_1,V_i]$ for $i=1,\ldots,\astep-1$ and $[V_1,V_{\astep}]=0$.
Then $\Ag$ is exponential and the number $\astep$ is called the \term{step} of $\Ag$.
For such groups instead of homotheties for every $\pt>0$ one defines the maps
$\hmth{\Ag}{\pt}\colon \Aalg \to \Aalg$, $\hmth{\Ag}{\pt}(v_1,\ldots,v_{\astep}) = \sum_{i=1}^{\astep} \pt^{i} v_i$ and their conjugations $\delat{\Ag}{\pt} = \exp_{\Ag} \circ\, \hmth{\Ag}{\pt} \circ \exp_{\Ag}^{-1} \colon \Ag \to \Ag$, called \term{dilations}.

Now let $\dif\colon \Ag\to \Bg$ be a continuous map between Carnot groups, $\Balg$ be the Lie algebra of $\Bg$, $\px\in\Ag$ be a point, and $\pt>0$.
Then one has the following mapping
\[
    \glh{\px}{\pt}\colon \Ag \to \Bg,
    \qquad
    \glh{\px}{\pt}(y) = \delat{\Bg}{1/\pt}(\dif(\px)^{-1}\dif(\px\delat{\Ag}{\pt}(y))),
\]
which can be also seen from the following commutative diagram:
\[
\xymatrix@C=2.7em{
      \Aalg \ar[d]_-{\hmth{\Ag}{\pt}} \ar[r]^-{\exp_{\Ag}} &
      \Ag \ar[rrrrr]^-{\glh{\px}{\pt}} \ar[d]^-{\delat{\Ag}{\pt}} &
      &&&&
      \Bg   &
      \Balg  \ar[l]_-{\exp_{\Bg}} \\
      \Aalg  \ar[r]^-{\exp_{\Ag}} &
      \Ag \ar[rr]^-{s \mapsto \px s} &&
      \Ag \ar[r]^-{\dif} &
      \Bg  \ar[rr]^-{s \mapsto (\dif(\px))^{-1} s} &&
      \Bg \ar[u]_-{\delat{\Bg}{1/\pt}}  &
      \Balg \ar[l]_-{\exp_{\Bg}} \ar[u]_-{\hmth{\Bg}{1/\pt}}
}
\]

Suppose that there exists a \term{homomorphism} $\pans{\px}{\dif}\colon\Ag\to\Bg$ such that the map
\[
    \glHom[\dif]_{\px}\colon\Ag\times[0;+\infty) \to \Bg,
    \qquad
    \glHom_{\px}(y,\pt) =
    \begin{cases}
        \glh{\px}{\pt}(y), & \pt>0, \\
        \pans{\px}{\dif}(y), & \pt=0,
    \end{cases}
\]
is continuous.
Then $\pans{x}{\dif}$ is called the \term{Pansu derivative} of $\dif$ at $x$.
It is easy to see that in this case the homomorphism $\pans{x}{\dif}$ is also \term{homogeneous} in the sense that $\pans{x}{\dif}\circ\delat{\Ag}{\pt} = \delat{\Bg}{\pt}\circ\pans{x}{\dif}$ for all $\pt>0$.

\newcommand\HomAut{\mathrm{Aut}'}
Given a Carnot group $\Ag$, let
\[
    \HomAut(\Ag) =
    \{
        \varphi\in\Aut(\Ag) \mid
        \varphi\circ\delat{\Ag}{\pt}=\delat{\Ag}{\pt}\circ\varphi\ \text{ for all }\pt>0
    \}
\]
be the group of all \term{homogeneous} automorphisms of $\Ag$, i.e.\ commuting with each dilation, and $\Diff_{p}(\Ag,e)$ be the set of $\Cinfty$ maps $\dif\colon(\Ag,e) \to (\Ag,e)$ fixing the unit and being Pansu differentiable at $e$.
In~\cite{Warhurst:BAuMS:2008} the set $\Diff_{p}(\Ag,e)$ is identified with the set of the so-called \term{contact} diffeomorphisms of $\Ag$.
Endow $\Diff_{p}(\Ag,e)$ with the weak Whitney topology $\Wtr{\infty}$.
Then the results of~\cite{Warhurst:BAuMS:2008} can be interpreted so that the map $\glHom\colon \Diff_{p}(\Ag,e)\times[0;1] \to \Diff_{p}(\Ag,e)$, $\glHom(\dif,\pt)(y) = \glHom[\dif]_{e}(y,\pt)$, is a strong deformation retraction of $\Diff_{p}(\Ag,e)$ onto $\HomAut(\Ag)$.
\end{enumerate}

\begin{subremark}\rm
Note that in each of the above cases the homotopy $\glHom[\dif]\colon\AManif\times[0;1]\to\BManif$ deforming a map $\dif=\glHom[\dif]_{1}$ into its ``linearized'' variant $\glHom[\dif]_{0}$, has the following two properties.
\begin{enumerate}[nosep, label={(P\arabic*)}]
\item\label{enum:glhom_prop:global}
\emph{The support of $\glHom$ is usually all of $\AManif$}, since $\glHom[\dif]_{1}$ can differ substantially from $\glHom[\dif]_{0}$ everywhere;
\item\label{enum:glhom_prop:diff}
\emph{If $\glHom[\dif]_{1}$ is a diffeomorphism, then $\glHom[\dif]_{\pt}$, $\pt\in[0;1]$, is a diffeomorphism as well}.
Indeed, $\glHom[\dif]_{\pt}$ for $\pt\in(0;1]$ is just a conjugation of $\dif$ by diffeomorphisms, while $\glHom[\dif]_{0}$ is a diffeomorphism due to specific properties of $\dif$.
\end{enumerate}
\end{subremark}

\subsection{Comparison with the results of the paper}
\label{sect:comparison_with_current_paper}
In the present paper we consider the space of smooth maps
\[
    \mAB = \Ci{(\AManif,\ASingMan)}{(\BManif,\BSingMan)}
\]
between arbitrary pairs of manifolds.
Then there are no ``canonical global linearized'' variants of maps $\dif\in\mAB$.
However, we can fix tubular neighborhoods $\Apr\colon\ATotMan\to\ASingMan$ and $\Bpr\colon\BTotMan\to\BSingMan$ of $\ASingMan$ and $\BSingMan$ and then by a \term{linearization} of $\dif$ mean a homotopy of $\dif$ to a map which coincides with some vector bundle morphism $\vmap\colon\ATotMan\to\BTotMan$ near $\ASingMan$.

To construct such a linearization we first fix some neighborhood $\ScalNbh\subset\ATotMan\cap\AManif$ of $\ASingMan$ satisfying $\dif(\ScalNbh) \subset \BTotMan$.
Then one can define a linearization homotopy $\glHom[\dif]\colon[0;1]\times\ScalNbh\to\BTotMan$ similar to~\eqref{equ:h_isot_jacobi_matrix}, see Eq.~\eqref{equ:homotopy_Ahom_}, which linearizes $\dif$ \term{on all of} $\ScalNbh$, and has both properties~\ref{enum:glhom_prop:global} and~\ref{enum:glhom_prop:diff}.

The problem now is that in general $\glHom[\dif]$ can not be extended to a homotopy of all $[0;1]\times\AManif\to\BManif$, so one needs to ``release'' somehow the property~\ref{enum:glhom_prop:global} and linearize $\dif$ only near $\ASingMan$.
This can be done by the following ``slowing down'' procedure making $\glHom[\dif]$ fixed on a neighborhood of $\partial\ScalNbh$.
Namely, fix any pair of neighborhoods $\Vman\subset\Uman$ of $\ASingMan$ in $\ScalNbh$ and a $\Cinfty$ function $\gamma\colon[0;1]\times\AManif\to[0;1]$ such that $\gamma(\pt,\px)=\pt$ for $(\pt,\px)\in[0;1]\times\Vman$ and $\gamma(\pt,\px)=1$ for $(\pt,\px)\in[0;1]\times(\ScalNbh\setminus\Uman)$.
Then the following homotopy, see Eq.~\eqref{equ:Chom_formula},
\begin{equation}
    \Chom[\dif]\colon[0;1]\times\AManif\to\BManif,
    \qquad
    \Chom[\dif](\pt,\px) =
    \begin{cases}
        \glHom[\dif](\gamma(\pt,\px),\px), & \px\in\ScalNbh, \\
        \dif(\px), & \px\in\AManif\setminus\Uman,
    \end{cases}
\end{equation}
is the desired linearization of $\dif$ near $\ASingMan$ supported in $\ScalNbh$.

However, this ``slowing down'' may destroy the second property~\ref{enum:glhom_prop:diff}, i.e.\ if $\dif$ is a diffeomorphism, then the map $\Chom[\dif]_{\pt}$ is not necessarily a diffeomorphism for some $\pt$: namely it may lose non-degeneracy on $\Uman\setminus\Vman$, a place of ``passage'' from $\restr{\Chom[\dif]_{\pt}}{\Vman} = \restr{\glHom[\dif]_{\pt}}{\Vman}$ to $\restr{\Chom[\dif]_{\pt}}{\AManif\setminus\Uman} = \restr{\dif}{\AManif\setminus\Uman}$.
Note that the openness of the set of diffeomorphisms $\Diff(\AManif,\BManif)$ does not help here, since (as we noted above) $\glHom[\dif]_{\pt}$ can be sufficiently ``far'' from $\dif$ on $\Vman$.

The main technical result of the paper, Lemma~\ref{lm:estim_PQRS_var}, shows that if the neighborhood $\Uman$ is chosen sufficiently thin and also \term{depends on $\dif$}, (consequently $\gamma=\gamma_{\dif}$ also depends on $\dif$), then the homotopy
\[
    \Ghom\colon\mAB\times[0;1]\to\mAB,
    \qquad
    \Ghom(\dif,\pt)(\px) =
     \begin{cases}
        \glHom[\dif](\gamma_{\dif}(\pt,\px),\px), & \px\in\ScalNbh, \\
        \dif(\px), & \px\in\AManif\setminus\ScalNbh,
    \end{cases}
\]
is a desired deformation of $\mAB$ into a subspace of ``linearized'' maps which also preserves diffeomorphisms (in fact ranks of Jacobi matrices), see Corollaries~\ref{cor:linearization_EMNab}, \ref{cor:def_L_Llin_not_pres_Lntab} and~\ref{cor:def_L_Llin_pres_Lntab}.

Another important observation is that if $\AManif$ and $\BManif$ are equipped with partitions $\AFol$ and $\BFol$ having some ``scalability'' properties, and $\dif\in\mAB$ sends leaves of $\AFol$ into leaves of $\BFol$, then so does $\Ghom(\dif,\pt)$ for all $\pt\in[0;1]$, see Definitions~\ref{def:G-isol_nbg}, \ref{def:folmap-scalable}, and Theorem~\ref{th:linearization_fol}.

Before formulating main results let us fix certain notations and definitions.
The reader may skip the rest of this section and use it for definitions and references.

\subsection{Manifolds, \GOOD\ submanifolds, and their tubular neighborhoods}
\label{sect:manif__reg_nbh}
A common strategy for studying a manifold $\AManif$ involves decomposing it into simpler pieces by cutting along codimension-one submanifolds.
In this context, connected components and even submanifolds of $\partial\AManif$ of smaller dimension play an important role.
For such submanifolds, their typical neighborhoods resemble ``halves'' of standard  neighborhoods  (like collars of boundary components).
Therefore, in what follows, we will need to adapt standard definitions of tubular neighborhoods to such ``halves''.
This also leads us to work only with submanifolds that have a ``good'' behavior with respect to $\partial\AManif$.

By $\ALocBase = \{ (\pxi{1},\ldots,\pxi{\axdim}) \mid \pxi{1}\geq0\}$, $\axdim\geq0$, we will denote a closed half-space in $\bR^{\axdim}$.
A topological space $\AManif$ is \term{locally Euclidean} if every point $x\in\AManif$ has an open neighborhood homeomorphic to some open subset of $\ALocBase$ for some $\axdim$ which may depend on $x$.

A \term{manifold} is a Hausdorff, second countable and locally Euclidean topological space equipped with a smooth ($\Cinfty$) structure; however, it may be non-compact, may have a boundary, and its connected components may have distinct dimensions.
Also, unless the contrary is stated, all spaces of smooth maps between manifolds are always equipped with the corresponding \term{weak} $\Cinfty$ Whitney topologies (which we will denote by $\Wtr{\infty}$ in contrast to the \term{strong} topology $\Str{\infty}$, see \SECT{sect:Whitney_topologies}).

A subset $\ASingMan \subset \AManif$ will be called a \term{submanifold}%
\footnote{Recall that usually, a subset $\ASingMan\subset\AManif$ is a \term{submanifold near its point $\px \in \ASingMan\cap \Int{\AManif}$}, if there exists a local chart $\AManif\supset \Uman \xrightarrow{\psi} \bR^{n}$ such that $\psi(\Uman\cap\ASingMan)$ is a coordinate linear subspace $\bR^{k}\times 0^{n-k}$ of $\bR^{n}$ for some $k$.
However, if $\px\in\ASingMan\cap\partial\AManif$ or we need to define a \term{boundary point of a submanifold}, the latter approach with ``normal forms'' does not always work.
The presented definition via $\Cinfty$ embedding $\phi$ allow one to consider, e.g.\ a smooth curve in $\AManif$ tangent to $\partial\AManif$ as a submanifold of $\AManif$.
}
if it is an image of a smooth embedding $\phi\colon\widehat{\ASingMan}\to\AManif$ of some manifold $\widehat{\ASingMan}$.
In other words, $\phi$ is a $\Cinfty$ homeomorphism of $\widehat{\ASingMan}$ onto $\ASingMan=\phi(\widehat{\ASingMan})$ (with respect to the induced topology from $\AManif$) such that the corresponding tangent map $\tang{\phi}\colon\tang{\widehat{\ASingMan}}\to\tang{\AManif}$ is injective.
Then $\phi$ pushes forward a smooth structure on $\ASingMan$.

Let $\Xman\subset\AManif$ be a subset.
Then a neighborhood $\ScalNbh\subset\AManif$ of $\Xman$ will be called \term{smooth}, if $\ScalNbh$ is a submanifold of $\AManif$.
For instance, every open neighborhood of $\Xman$ is smooth.

Let $\Apr\colon\ATotMan\to\ASingMan$ be a vector bundle over a manifold $\ASingMan$.
Then we will always assume that it is \term{smooth}, i.e.\ $\ATotMan$ is equipped with a smooth structure under which $\Apr$ is $\Cinfty$.
A fiber over a point $\apnt\in\ASingMan$ will be denoted by $\Afib{\apnt}:=\Apr^{-1}(\apnt)$.
By the construction, $\Afib{\apnt}$ has a structure of a real vector space of some dimension $\ardim_{\apnt}$, which in fact is the same for all points $\py$ from the connected component of $\apnt$ in $\ASingMan$.
We allow $\ardim_{\apnt}$ to be distinct for different connected components of $\ASingMan$.
We also identify $\ASingMan$ with its image under the zero section $\zeta\colon\ASingMan\to\ATotMan$.

For each $\tau>0$ the following vector bundle isomorphism $\ahomot{\tau}\colon\ATotMan\to\ATotMan$, $\ahomot{\tau}(\avct) = \tau\avct$, will be called the \term{$\tau$-homothety}.
For a subset $\ScalNbh\subset\ATotMan$ and $\tau\geq0$ it will also be convenient to denote $\tau\ScalNbh:=\ahomot{\tau}(\ScalNbh)$.
Say that $\ScalNbh$ is \term{\STL} if $\ahomot{\tau}(\ScalNbh)\subset\ScalNbh$ for all $\tau\in[0;1]$.

Note that for every open subset $\AopBs \subset \ASingMan$, the restriction map
\[ \restr{\Apr}{\EInv{\AopBs}}\colon\EInv{\AopBs} \to \AopBs \]
has an induced vector bundle structure.
Then a \term{trivialization} over $\AopBs$ is a smooth open embedding $\EChart\colon\EInv{\AopBs}\monoArrow \ALocBaseFiber$ being a vector bundle morphism into a trivial vector bundle $\ALocBaseFiber\to\ALocBase$.
Let us mention that if $\pw\in\EInv{\AopBs}$ and $\EChart(\pw) = (\apnt,\avct)\in\ALocBaseFiber$, then for every $\tau\in\bR$ we have that $\EChart(\tau\pw) = (\apnt, \tau\avct)$.

\newcommand\propPart[1]{#1^{(\mathsf{nt})}}
\newcommand\bdPart[1]{#1^{(\partial)}}
\newcommand\ESingManProp{\propPart{\ASingMan}}
\newcommand\ESingManBd{\bdPart{\ASingMan}}
\newcommand\ATotManProp{\propPart{\ATotMan}}
\newcommand\ATotManBd{\bdPart{\ATotMan}}
\newcommand\EprProp{\propPart{\Apr}}
\newcommand\EprBd{\bdPart{\Apr}}
\begin{subdefinition}\label{def:good_submanifold}
A submanifold $\ASingMan$ of $\AManif$ will be called \term{\GOOD}, if for every connected component $\ASingMan'$ of $\ASingMan$ either
\begin{enumerate}[nosep, label={\rm(\alph*)}, leftmargin=*]
\item\label{enum:def:good_submanifold:proper}
$\ASingMan'$ is a \term{\NEAT} submanifold of $\AManif$, i.e.\ $\partial\ASingMan' = \ASingMan' \cap \partial\AManif$ and this intersection is transversal, \cite[Ch.~4, \S4]{Hirsch:DiffTop}, or

\item\label{enum:def:good_submanifold:bd}
 $\ASingMan' \subset \partial\AManif$ and $\partial\ASingMan'=\varnothing$.
\end{enumerate}
\rm In this case we will denote by $\ESingManProp$ and $\ESingManBd$ the unions of components of $\ASingMan$ satisfying conditions~\ref{enum:def:good_submanifold:proper} and~\ref{enum:def:good_submanifold:bd} respectively, and call them \term{\NEAT} and \term{boundary} parts of $\ASingMan$.
Thus $\ASingMan = \ESingManProp \sqcup \ESingManBd$.
\end{subdefinition}

\newcommand\TwoMman{2\AManif}
For a \GOOD\ submanifold $\ASingMan$ of $\AManif$, we will now adapt the standard notion of a tubular neighborhood.
It will be convenient to regard it as a subset of the \term{double} of $\AManif$:
\[
    \TwoMman =
        \bigl(\AManif\times\{0,1\}\bigr)
            \,/\,
        \{ (\px,0)\sim(\px,1) \mid \px\in\partial\AManif\},
\]
obtained by gluing two copies of $\AManif$ by the identity diffeomorphism of $\partial\AManif$, see Figure~\ref{fig:tubular_neighborhood}.
It is well known that $\TwoMman$ has a ``canonical'' smooth structure such that the inclusions $\AManif \times \{i\} \subset \TwoMman$, $i=0,1$, are smooth embeddings.
We will identify $\AManif$ with $\AManif\times\{0\}$ in $\TwoMman$.

\begin{subdefinition}[Tubular neighborhood]
\label{def:tubular_nbh}
Let $\ASingMan$ be a \GOOD\ submanifold of $\AManif$.
Consider three cases.
\begin{enumerate}[wide, label={\rm\arabic*)}]
\item\label{enum:tub_nbh:prop}
If $\ASingMan$ is a \NEAT\ submanifold of $\AManif$, then a \term{tubular} neighborhood of $\ASingMan$ is a smooth vector bundle $\Apr\colon\ATotMan\to\ASingMan$, where
\begin{enumerate}[nosep, label={\rm(\roman*)}]
\item
$\ATotMan$ is an open neighborhood of $\ASingMan$ in $\AManif$, and $\Apr$ is fixed on $\ASingMan$ (i.e.\ $\Apr$ is a smooth retraction);
\item
$\Apr^{-1}(\partial\ASingMan) = \partial\ATotMan = \ATotMan\cap\partial\AManif$ is a subbundle of $\ATotMan$ and the latter intersection is transversal.
\end{enumerate}

\item\label{enum:tub_nbh:bd}
Suppose $\ASingMan$ is a submanifold of $\partial\AManif \subset\AManif\subset \TwoMman$ and $\partial\ASingMan=\varnothing$.
Then a \term{tubular} neighborhood of $\ASingMan$ is a smooth vector bundle $\Apr\colon\ATotMan\to\ASingMan$, where
\begin{enumerate}[nosep, label={\rm(\roman*)}]
\item
$\ATotMan$ is an open neighborhood of $\ASingMan$ in $\TwoMman$ and $\Apr$ is fixed on $\ASingMan$;
\item
$\ATotMan\cap\partial\AManif$ is a subbundle of $\ATotMan$, being thus a tubular neighborhood of $\ASingMan$ in $\partial\AManif$ in the sense~\ref{enum:tub_nbh:prop}.
\end{enumerate}

\item\label{enum:tub_nbh:general}
In general, $\ASingMan = \ESingManProp\sqcup\ESingManBd$ is a disjoint union of its \NEAT\ and boundary parts.
Let $\AManif\supset\ATotManProp  \xrightarrow{~\EprProp~} \ESingManProp$ and $\TwoMman \supset \ATotManBd \xrightarrow{~\EprBd~} \ESingManBd$ be any tubular neighborhoods of $\ESingManProp$ and $\ESingManBd$ in the above sense such that $\ATotManProp \cap \ATotManBd=\varnothing$.
Put $\ATotMan:=\ATotManProp \cup \ATotManBd$.
Then the induced map $\Apr:=\EprProp\sqcup\EprBd\colon\ATotMan\to\ASingMan$ will be called a \term{tubular} neighborhood of $\ASingMan$, see Figure~\ref{fig:tubular_neighborhood} illustrating all these cases.
\end{enumerate}
\end{subdefinition}

\begin{figure}[htbp!]
\centerline{\includegraphics[height=2.5cm]{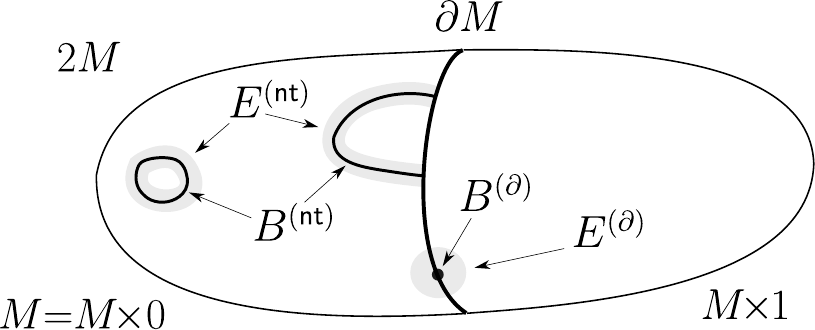}}
\caption{Good submanifold and its tubular neighborhood}\label{fig:tubular_neighborhood}
\end{figure}

Thus, a tubular neighborhood of $\ASingMan$ is always a vector bundle $\Apr\colon\ATotMan\to\ASingMan$, but over the boundary part $\ESingManBd$ it is defined on a neighborhood $\ATotManBd$ of $\ESingManBd$ in $\TwoMman$.
Also, $\ATotMan\cap\AManif = \ATotManProp \sqcup (\ATotManBd\cap\AManif)$ is always an open neighborhood of $\ASingMan$ in $\AManif$.
Notice that $\ATotMan\cap\AManif$ is invariant with respect to the multiplication of its elements by \term{non-negative} reals, and in particular it is \STL.
Such property will be very important for our constructions, see e.g.~\eqref{equ:homotopy_Ahom_}.

\subsection{Whitney topologies}
\label{sect:Whitney_topologies}
For $\px=(\pxi{1},\ldots,\pxi{\axdim})\in\bR^{\axdim}$ we will denote by $\nrm{\px}=\sqrt{\sum_{i=1}^{\axdim}\pxi{i}^2}$ the usual Euclidean norm.

Let $\Uman \subset \ALocBase$ be an open subset and $\rr\geq0$ be a finite non-negative integer.
For every compact subset $\BComp\subset\Uman$ and a $\Cr{\rr}$ map $\dif=(\dif_1,\ldots,\dif_{\bxdim})\colon\Uman\to\bR^{\bxdim}$ define the \term{$(\rr,\BComp)$-seminorm} of $\dif$ on $\BComp$ by the following formulas:
\begin{equation}\label{equ:r_K_norm}
    \xnrm{\dif}{\rr}{\BComp}
        = \sum_{\substack{\alpha=(\alpha_1,\ldots,\alpha_{\axdim}), \\ \nabs{\alpha}:=\alpha_1+\cdots+\alpha_{\axdim} \leq \rr}}
          \ \sum_{i=1}^{\bxdim} \
          \sup_{\px\in\BComp}
            \Bigl\vert
                \frac{\partial^{\nabs{\alpha}}\dif_i}{\partial\pxi{1}^{\alpha_1}\cdots \partial\pxi{\axdim}^{\alpha_{\axdim}}}(\px)
            \Bigr\vert.
\end{equation}

Evidently,
\begin{equation}\label{equ:r_rplus1_norms}
    \xnrm{\dif}{\rr}{\BComp} \leq \xnrm{\dif}{\rr+1}{\BComp}
\end{equation}

Now let $\dif\colon\AManif\to\BManif$ be a $\Cr{\rr}$ map between manifolds, and
\begin{align*}
    &\AManif\supset\Uman \xrightarrow{\phi} \bR^{\axdim}, &
    &\BManif\supset\Vman \xrightarrow{\psi} \bR^{\bxdim},
\end{align*}
two coordinate charts such that $\dif(\Uman) \subset \Vman$, so we have the corresponding \term{coordinate representation} $\psi\circ\dif\circ\phi^{-1}\colon\phi(\Uman) \to \psi(\Vman)$ of $\dif$.
Given a compact subset $\BComp\subset\Uman$ and $\eps>0$ define the following subset of $\Ci{\AManif}{\BManif}$:
\begin{equation}\label{equ:weak_nbh}
\begin{aligned}
    \Nbh{\dif}{\rr}{\BComp}{\eps,\phi,\psi} :=
    \bigl\{
        &\gdif\in\Crm{\rr}{\AManif}{\BManif} \mid \\
        &\gdif(\BComp)\subset\Vman \ \text{and} \
        \xnrm{\psi\circ\gdif\circ\phi^{-1} - \psi\circ\dif\circ\phi^{-1}}{\rr}{\phi(\BComp)} < \eps
    \bigr\}.
\end{aligned}
\end{equation}
Then the family $\alpha = \{ \Nbh{\dif}{\rr}{\BComp}{\eps,\phi,\psi} \}$ of all such sets generates some topology on $\Crm{\rr}{\AManif}{\BManif}$ called the \term{weak $\Cr{\rr}$ Whitney topology}.
We will denote it by $\Wtr{\rr}$.
Note that the base of $\Wtr{\rr}$ consists of \term{finite} intersections $\mathop{\cap}\limits_{i=1}^{n}\Nbh{\dif_i}{\rr}{\BComp_i}{\eps_i,\phi_i,\psi_i}$, while $\Wtr{\rr}$ consists of arbitrary unions of those finite intersections.

Let also $\beta$ be the family of all \term{at most countable} intersections of sets~\eqref{equ:weak_nbh}:
\[
    \mathop{\cap}\limits_{i\in\Lambda}\Nbh{\dif_i}{\rr}{\BComp_i}{\eps_i,\phi_i,\psi_i},
\]
in which $\{\BComp_i\}_{i\in\Lambda}$ is a \term{locally finite} family of compact subsets.
Then $\beta$ generates on $\Crm{\rr}{\AManif}{\BManif}$ a finer topology called the \term{strong $\Cr{\rr}$ Whitney topology}.
We will denote it by $\Str{\rr}$.

Clearly $\Wtr{0}$ coincides with the compact open topology, $\Wtr{\rr}  \subset \Str{\rr}$, $\Wtr{\rr} \subset \Wtr{\rr+1}$, and $\Str{\rr} \subset \Str{\rr+1}$ for all $\rr\geq0$.
Then for $\rr=\infty$ one can also define the \term{weak} and \term{strong} Whitney topologies as the corresponding unions:
\begin{align}\label{equ:ws_infty_def}
\Wtr{\infty} &= \mathop{\cup}\limits_{\rr=0}^{\infty}\Wtr{\rr}, &
\Str{\infty} &= \mathop{\cup}\limits_{\rr=0}^{\infty}\Str{\rr}.
\end{align}

Let $\calX\subset\Ci{\AManif}{\BManif}$ be a subset and $\rr=0,1\ldots\infty$.
Then it will be convenient to denote by $\WTopol{\rr}{\calX}$ and $\STopol{\rr}{\calX}$ that set $\calX$ endowed with the topologies $\Wtr{\rr}$ and $\Str{\rr}$ respectively on $\Ci{\AManif}{\BManif}$.
We also say that $\calX$ is \term{$\Wtr{\rr}$-open} (resp.\ \term{$\Str{\rr}$-open}) if it is open in the topology $\Wtr{\rr}$ (resp.\ $\Str{\rr}$).

Let $\CManif,\DManif$ be some other manifolds and
\[
    \Ci{\AManif}{\BManif}
    \ \supset \
        \XX \xrightarrow{~\dif~} \YY
    \ \subset \
    \Ci{\CManif}{\DManif}
\]
be a map between some subsets $\XX$ and $\YY$.
Then $\dif$ will be called \term{$\Wtr{r,s}$-continuous} (resp.\ \term{$\Wtr{r,s}$-homotopy equivalence}) if it is so when $\XX$ is endowed with the induced topology $\Wtr{r}$ and $\YY$ with the induced topology $\Wtr{s}$.
If $r=s$, then we will sometimes abbreviate $\Wtr{r,s}$ to $\Wtr{r}$.
Similar terms will be used for continuity of maps with respect to strong topologies.
The following lemma is straightforward:
\begin{sublemma}\label{lm:infty_continuous}
Let $f\colon\XX\to\YY$ be a map.
Suppose that for each $s<\infty$ there exists $r=r(s)<\infty$ such that $f$ is $\Wtr{r(s),s}$-continuous (resp.\ $\Str{r(s),s}$-continuous).
Then $f$ is $\Wtr{\infty,\infty}$-continuous (resp.\ $\Str{\infty,\infty}$-continuous).
\end{sublemma}

\subsection{Vertical derivative}
\label{sect:vert_derivative}
Let $\Apr\colon\ATotMan\to\ASingMan$ be a vector bundle over some manifold $\ASingMan$ and $\tbPrj{\ATotMan}\colon\tang{\ATotMan}\to\ATotMan$ and $\tbPrj{\ASingMan}\colon\tang{\ASingMan}\to\ASingMan$ be the corresponding tangent bundles.
As mentioned above, we regard $\ASingMan$ as the submanifold of zeros in the fibers of $\ATotMan$, i.e.\ the image under the zero section of $\Apr$.
Since $\ASingMan$ transversely intersects each fiber of $\ATotMan$ at a unique (zero) point, it follows that for each $\px\in\ASingMan$ its tangent space $\tang[\px]{\ATotMan}$ splits into the direct sum $\tang[\px]{\ASingMan} \oplus \tang[\px]{\Afib{\px}}$ of tangent spaces to $\ASingMan$ and to the fiber $\Afib{\px}$.
The above pointwise direct sum gives rise to a Whitney sum of $\tang{\ASingMan}$ with the bundle of vertical vectors along $\ASingMan$.
Namely, let
\[
    \tang[\ASingMan]{\ATotMan} := \mathop{\sqcup}\limits_{\px\in\ASingMan}\!\!\tang[\px]{\ATotMan}
                                = \tbPrj{\ATotMan}^{-1}(\ASingMan)
\]
be the pull back of the tangent bundle $\tbPrj{\ATotMan}$ via the zero-section inclusion $\zeta\colon\ASingMan\subset\ATotMan$,
and
\[
    \AVect := \mathop{\sqcup}\limits_{\px\in\ASingMan}\!\!\tang[\px]{\Afib{\px}}
    \ \subset \
    \tang[\ASingMan]{\ATotMan}
\]
be the kernel of the tangent map $\tang{\Apr}\colon\tang[\ASingMan]{\ATotMan}\to\tang{\ASingMan}$.
Then we get a short exact sequence of vector bundles
\[
    0 \to
    \AVect  \to
    \tang[\ASingMan]{\ATotMan} \xrightarrow{\tang{\Apr}}
    \tang{\ASingMan} \to
    0
\]
for which the tangent map inclusion $\tang{\zeta}\colon\tang{\ASingMan} \subset \tang[\ASingMan]{\ATotMan}$ is a section of $\tang{\Apr}$.
Hence this sequence splits and we get the following
\begin{sublemma}\label{lm:TES_VES__TESE}
The natural inclusion
\begin{equation}\label{equ:TES_VES__TESE}
    \eta\colon\tang{\ASingMan}\oplus\AVect \to \tang[\ASingMan]{\ATotMan},
    \qquad
    \eta(\pu,\pv) = \pu +\pv,
\end{equation}
is an isomorphism of vector bundles over $\ASingMan$.
\qed
\end{sublemma}

\begin{subremark}\rm
Let $\vtang{\ATotMan} \subset \tang{\ATotMan}$ be subbundle of \emph{all} tangent vectors to \emph{all} the fibers of $\ATotMan$, so it is the kernel of the tangent map $\tang{\Apr}\colon\tang{\ATotMan}\to\tang{\ASingMan}$.
Then $\tang{\Apr}$ has a lot of sections, each giving an isomorphism $\vtang{\ATotMan} \oplus \tang{\ASingMan} \cong \tang{\ATotMan}$.
However, there is no canonical choice of such a section.
On the other hand, such a ``canonical'' section exists \emph{``over $\ASingMan$''}, and this gives an isomorphism from Lemma~\ref{lm:TES_VES__TESE}.
\end{subremark}

Another well-known and straightforward result is that $\ATotMan$ is canonically isomorphic with $\AVect$.
\begin{sublemma}[{\rm e.g.\cite[\S8.12]{Michor:DG:AMS:2008}}]\label{lm:E__psi_pr2}
For each $\pv\in\ATotMan$ define the following curve
\[
    \gamma_{\pv}\colon\bR\to\Afib{\Apr(\pv)} \subset \ATotMan,
    \qquad
    \gamma_{\pv}(\pt)=\pt\pv,
\]
in the fiber $\Afib{\Apr(\pv)}$.
Let also
\[
    \psi(\pv) = \restr{\tfrac{d\gamma_{\pv}}{dt}}{\pt=0} \in \tang[\Apr(\pv)]\Afib{\Apr(\pv)}
\]
be the tangent vector to $\gamma_{\pv}$ at $\gamma_{\pv}(0) = \Apr(\pv) \in \ASingMan$.
Then the correspondence
\begin{equation}\label{equ:E_VB}
    \psi\colon\ATotMan \to \AVect,
    \qquad
    \pv \mapsto \psi(\pv),
\end{equation}
is a vector bundle isomorphism.
\qed
\end{sublemma}

Thus Lemmas~\ref{lm:TES_VES__TESE} and~\ref{lm:E__psi_pr2} show that there are isomorphisms
\[ \tang[\ASingMan]{\ATotMan} \cong \tang{\ASingMan}\oplus\AVect \cong \tang{\ASingMan} \oplus \ATotMan. \]
The following lemma gives the formulas for the projection $\tang[\ASingMan]{\ATotMan} \to \ATotMan$.
\begin{sublemma}\label{lm:E__pr2}
The projection
\[
    \pr_{2}\colon
        \tang[\ASingMan]{\ATotMan}
            \xrightarrow[~\eqref{equ:TES_VES__TESE}~]{\eta^{-1}}
        \tang{\ASingMan}\oplus \AVect
            \xrightarrow{~(\pu,\pv)\,\mapsto\,\pv~}
        \AVect
            \xrightarrow[~\eqref{equ:E_VB}~]{\psi^{-1}}
        \ATotMan
\]
is given by the following rule.
Let $\px\in\ASingMan$, $\pv \in \tang[\px]{\ATotMan} \subset \tang[\ASingMan]{\ATotMan}$ be a tangent vector at $\px$, and $\gamma\colon\bR\to\ATotMan$ be a $\Cr{1}$ curve representing $\pv$, so $\gamma(0)=\px$ and $\pv = \restr{\tfrac{d\gamma}{dt}}{\pt=0}$.
Then
\[
    \mathrm{pr}_{2}(\pv) = \lim\limits_{\pt\to0} \tfrac{1}{\pt}\gamma(\pt).
\]
\end{sublemma}
\begin{proof}
Since the problem is local, one can assume that $\Apr\colon\ATotMan=\ASingMan\times\ALocFiber\to\ASingMan$ is a trivial vector bundle over an open subset $\ASingMan \subset\ALocBase$.
This yields a canonical identification $\tang{\ATotMan} = \ASingMan\times\ALocFiber \times \bR^{\axdim} \times \ALocFiber$, so that the maps $\psi$ and $\pr_{2}$ are given by the following formulas:
\begin{align*}
\psi\colon&
    \ATotMan=\ASingMan\times\ALocFiber
        \to
    \AVect = \ASingMan\times 0 \times 0 \times \ALocFiber, &
& \psi(\py,\pv)=(\py,0,0,\pv), \\
\pr_{2}\colon&
    \tang[\ASingMan]{\ATotMan} = \ASingMan\times 0 \times \bR^{\axdim} \times \ALocFiber
        \to
    \ATotMan = \ASingMan\times\ALocFiber, & \
& \pr_{2}(\py,0,\pu,\pv) = (\py,\pv).
\end{align*}

Let $\gamma=(\alpha,\beta)\colon\bR\to\ASingMan\times\ALocFiber$ be the coordinate maps of $\gamma$ and
\[
    \tang{\gamma}\colon \tang{\bR} = \bR\times\bR \to \tang{\ATotMan},
    \qquad
    \tang{\gamma}(\pt,\pv) = \bigl(\alpha(\pt),\beta(\pt), \alpha'(\pt)\pv, \beta'(\pt)\pv\bigr),
\]
be the corresponding tangent map.
By assumption $\px=\gamma(0)\in\ASingMan\times0$, whence $\beta(0)=0$, and thus $\beta(\pt) = \pt \beta'(0) + \pt^2\hat{\beta}(\pt)$ for some continuous map $\hat{\beta}\colon\bR\to\ALocFiber$.
Moreover,
\[
    \pv = \tang{\gamma}(0,1) = (\alpha(0),0, \alpha'(0), \beta'(0)).
\]
Therefore, $\pr_{2}(\pv) = (\alpha(0), \beta'(0))$.

On the other hand, notice that the multiplication in $\ATotMan$ by scalars is given by $\pt\cdot(\py,\pv) = (\py,\pt\pv)$.
Hence, $\tfrac{1}{\pt}\gamma(\pt) = \bigl(\alpha(\pt), \tfrac{1}{\pt}\beta(\pt)\bigr) = \bigl(\alpha(\pt), \beta'(0) + \pt\hat{\beta}(\pt)\bigr)$ for $\pt\not=0$, and therefore
\[
    \pr_{2}(\pv) = (\alpha(0), \beta'(0)) = \lim\limits_{\pt\to0}(\alpha(\pt), \tfrac{1}{\pt}\beta(\pt)) =
    \lim\limits_{\pt\to0}\tfrac{1}{\pt}\gamma(\pt).
    \qedhere
\]
\end{proof}

\begin{sublemma}\label{lm:vertical_derivative}
Let $\Apr\colon\ATotMan\to\ASingMan$ and $\Bpr\colon\BTotMan\to\BSingMan$ be two vector bundles, $\ScalNbh \subset \ATotMan$ be a smooth neighborhood of $\ASingMan$, and $\mNF:=\Ci{(\ScalNbh,\ASingMan)}{(\BTotMan,\BSingMan)}$.
Then for each $\dif\in\mNF$ there exists a unique vector bundle morphism $\vtang{\dif}\colon\ATotMan\to\BTotMan$, called the \term{vertical derivative} of $\dif$, making commutative the following diagram:
\begin{equation}\label{equ:def_tfib_diagram}
\newcommand{\mapsfrom}{\mathrel{\reflectbox{\ensuremath{\mapsto}}}}
\begin{gathered}
\xymatrix@C=3em{
 \ATotMan \ar[r]^-{\psi}_-{\cong} \ar[d]_-{\vtang{\dif}} &
 \ \AVect \ \ar@{^(->}[rr]^-{i\colon \pv \ \mapsto \ (0, \pv)}  &&
 \tang{\ASingMan} \oplus \AVect
  \ar[r]^-{\eta}_-{\cong} & 
 \tang[\ASingMan]{\ATotMan} \ar[d]^-{\tang{\dif}} \\
 \BTotMan \ar[r]^-{\psi}_-{\cong} &
 \ \BVect \ &&
 \tang{\BSingMan} \oplus \BVect
 \ar@{->>}[ll]_{\pv\,\mapsfrom\,(\py,\pv) \colon \pr} \ar[r]^-{\eta}_-{\cong} &
 \tang[\BSingMan]{\BTotMan}
}
\end{gathered}
\end{equation}
In fact, $\vtang{\dif}$ can be defined as follows:
\begin{equation}\label{equ:Tfibh:formula}
    \vtang{\dif}(\apnt) := \lim_{\pt\to0}\tfrac{1}{\pt}\dif(\pt\apnt), \ \apnt\in\ATotMan.
\end{equation}
Moreover, the following identities hold:
\begin{gather}
    \label{equ:Vh_on_as__h_on_as}
    \restrb{\vtang{\dif}}{\ASingMan} = \restr{\dif}{\ASingMan}, \\
    \label{equ:VVh_Vh}
    \vtang{(\vtang{\dif})}=\vtang{\dif}.
\end{gather}
\begin{enumerate}[label={\rm(\arabic*)}]
\item\label{enum:lm:vertical_derivative:uniqueness}
In addition, let $\vmap\colon\ATotMan\to\BTotMan$ be a vector bundle morphism, and $\dif\in\mNF$ be such that $\dif=\vmap$ near $\ASingMan$.
Then $\vmap = \vtang{\dif}$.

\item\label{enum:lm:vertical_derivative:diff_yields_iso}
Also, suppose that there exist open neighborhoods $\Uman \subset \ATotMan$ and $\Vman\subset\BTotMan$ of $\ASingMan$ and $\BSingMan$ respectively, such that $\dif$ induces a diffeomorphism of a pair $(\Uman,\ASingMan)$ onto the pair $(\Vman,\BSingMan)$.
Then $\vtang{\dif}\colon\ATotMan\to\BTotMan$ is a vector bundle isomorphism.
\end{enumerate}
\end{sublemma}
\begin{proof}
Since $\dif(\ASingMan)\subset\BSingMan$, we also have that $\tang{\dif}\bigl(\tang[\ASingMan]{\ATotMan}\bigr) \subset \tang[\ASingMan]{\BTotMan}$.
Hence the diagram~\eqref{equ:def_tfib_diagram} is well defined and uniquely determines $\vtang{\dif}$ as the composition:
\[
    \vtang{\dif} =
    \bigl(\psi^{-1} \circ \pr \circ \eta^{-1}\bigr)
    \circ
    \tang{\dif}
    \circ
    \bigl(\eta\circ i \circ \psi\bigr).
\]

Let us verify the formula~\eqref{equ:Tfibh:formula}.
Let $\pw\in\ATotMan$.
Since $\ScalNbh$ is a neighborhood of $\ASingMan$, there exists $\eps>0$ such that $\pt\pw\in\ScalNbh$ for all $\pt\in[0;\eps]$.
Hence, the expression $\tfrac{1}{\pt}\dif(\pt\pw)$ is well-defined for all sufficiently small $\pt$.

Note that in the diagram~\eqref{equ:def_tfib_diagram}, the composition $\eta\circ i\colon\AVect \to \tang[\ASingMan]{\ATotMan}$ is the identity inclusion map.
Hence the composition of upper horizontal arrows $\eta\circ i \circ \psi\colon\ATotMan\to\tang[\ASingMan]{\ATotMan}$ coincides with $\psi$.

Therefore, by Lemma~\ref{lm:E__psi_pr2}, $\eta\circ i \circ \psi(\pw)$ is the vertical tangent vector at $\Apr(\pw)$ represented by the curve $\pt\mapsto \pt \pw$.
Whence the tangent vector $\tang{\dif}\circ\bigl(\eta\circ i \circ \psi(\pw)\bigr)$ is represented by the curve $\pt\mapsto\dif(\pt\pw)$.

On the other hand, the composition of the lower horizontal arrows in~\eqref{equ:def_tfib_diagram} is the map $\pr_{2}$ from Lemma~\ref{lm:E__pr2}.
Hence $\vtang{\dif}(\pw) = \lim\limits_{\pt\to0}\tfrac{1}{\pt}\dif(\pt\pw)$.

\eqref{equ:Vh_on_as__h_on_as}
This identity follows from commutativity of diagram~\eqref{equ:def_tfib_diagram}, and can also be deduced from~\eqref{equ:Tfibh:formula} as follows.
Let $\apnt\in\ASingMan$ and $\bpnt=\dif(\apnt)\in\BSingMan$.
Then $\apnt$ and $\bpnt$ are zeros in the corresponding fibers $\Afib{\apnt}$ and $\Bfib{\bpnt}$, respectively, i.e.\ $\pt\apnt=\apnt$ and $\pt\bpnt=\bpnt$ for all $\pt\in\bR$.
Hence
\[
    \tfrac{1}{\pt}\dif(\pt\apnt) = \tfrac{1}{\pt}\dif(\apnt) = \dif(\apnt),
\]
for all $\pt\in\bR\setminus\{0\}$, and therefore $\vtang{\dif}(\apnt)=\dif(\apnt)$.

\eqref{equ:VVh_Vh}
Since $\vtang{\dif}$ is linear on fibers, we have that $\vtang{\dif}(\pt\apnt)=\pt\vtang{\dif}(\apnt)$ for all $\apnt\in\ATotMan$ and $\pt\in\bR$, whence
\[
    \vtang{(\vtang{\dif})}(\apnt)
        = \lim\limits_{\pt\to0} \tfrac{1}{\pt} \vtang{\dif}(\pt\apnt)
        = \lim\limits_{\pt\to0} \tfrac{\pt}{\pt} \vtang{\dif}(\apnt)
        = \lim\limits_{\pt\to0} \vtang{\dif}(\apnt)
        =  \vtang{\dif}(\apnt).
\]

\ref{enum:lm:vertical_derivative:uniqueness}
Let $\vmap\colon\ATotMan\to\BTotMan$ be a vector bundle morphism, and $\dif\in\mNF$ be such that $\dif=\vmap$ on some neighborhood $\Uman\subset\ATotMan$ of $\ASingMan$.
Decreasing $\Uman$ one may assume that it is \STL.
Then for each $\apnt\in\Uman$ we have that
\[
    \vtang{\dif}(\apnt) =
    \lim\limits_{\pt\to0} \tfrac{1}{\pt} \dif(\pt\apnt) =
    \lim\limits_{\pt\to0} \tfrac{1}{\pt} \vmap(\pt\apnt) =
    \lim\limits_{\pt\to0} \tfrac{1}{\pt} \pt \vmap(\apnt) =
    \lim\limits_{\pt\to0} \vmap(\apnt) =
    \vmap(\apnt).
\]
Thus $\vmap = \vtang{\dif}$ on $\Uman$, and therefore they coincide on all of $\ATotMan$.

\ref{enum:lm:vertical_derivative:diff_yields_iso}
Finally, suppose that there exist open neighborhoods $\Uman \subset \ATotMan$ and $\Vman\subset\BTotMan$ of $\ASingMan$ and $\BSingMan$ respectively, such that $\dif$ induces a diffeomorphism of a pair $(\Uman,\ASingMan)$ onto the pair $(\Vman,\BSingMan)$.
In particular, in the Diagram~\eqref{equ:def_tfib_diagram}, $\tang{\dif}$ is a vector bundle isomorphism which isomorphically sends $\tang{\ASingMan}$ onto $\tang{\BSingMan}$.
Hence $\vtang{\dif}$ is a composition of the following vector bundle isomorphisms:
\[
    \ATotMan                     \xrightarrow[\cong]{~\psi~}
    \vtang[\ASingMan]{\ATotMan}  \xrightarrow[\cong]{~~}
    \frac{\tang{\ASingMan}\oplus\vtang[\ASingMan]{\ATotMan}}{\tang{\ASingMan}}
                                 \xrightarrow[\cong]{~\pr\,\circ\,\eta^{-1}\,\circ\,\tang{\dif}\,\circ\,\eta~}
    \vtang[\BSingMan]{\BTotMan}  \xrightarrow[\cong]{~\psi^{-1}~}
    \BTotMan.
    \qedhere
\]
\end{proof}

\begin{subremark}\rm
The map $\vtang{\dif}$ is often called the \term{normal derivative} of $\dif$, e.g.\ \cite[\S2.2]{BursztynLimaMeinrenken:JRAM:2019}.
We will prefer the term \term{vertical} since it indicates a dependency on the vector bundle structures on $\ATotMan$ and $\BTotMan$.
Also, note that the quotient vector bundle morphism $\nu\colon\tang[\BSingMan]{\BTotMan} \to \tang[\BSingMan]{\BTotMan}/\tang{\BSingMan}$ isomorphically maps $\BVect$ onto $\tang[\BSingMan]{\BTotMan}/\tang{\BSingMan}$, and therefore one can define $\vtang{\dif}$ as the following composition
$\AVect \subset \tang[\ASingMan]{\ATotMan} \xrightarrow{\tang{\dif}} \tang[\BSingMan]{\BTotMan} \xrightarrow{\nu} \tang[\BSingMan]{\BTotMan}/\tang{\BSingMan}$.
This definition is used in~\cite[\S5]{GuilleminSternberg:MS:1977}, and that composition is called the \emph{intrinsic derivative} of $\dif$.
\end{subremark}

\section{Main results: linearization theorems}
\label{sect:linearization_theorems}

\subsection{Preliminaries}
Let $\AManif$ and $\BManif$ be two connected manifolds, $\ASingMan\subset\AManif$ and $\BSingMan\subset\BManif$ be connected \GOOD\ submanifolds (see Definition~\ref{def:good_submanifold}), and $\mAB = \Ci{(\AManif,\ASingMan)}{(\BManif,\BSingMan)}$ be the space of $\Cinfty$ maps of pairs.
Fix some tubular neighborhoods $\Apr\colon\ATotMan\to\ASingMan$ and $\Bpr\colon\BTotMan\to\BSingMan$ of $\ASingMan$ and $\BSingMan$ and denote
\begin{align*}
    \axdim &:= \dim\ASingMan,        &
    \ardim &:= \dim\AManif - \axdim, &
    \bxdim &:= \dim\BSingMan,        &
    \brdim &:= \dim\BManif - \bxdim,
\end{align*}
so $\ardim$ and $\brdim$ are the dimensions of fibers of $\Apr$ and $\Bpr$ respectively.

Before formulating our main results Theorems~\ref{th:linearization_EMN_and_subspaces}, \ref{th:linearization_DiffMN}, and~\ref{th:linearization_fol} we need the notions of horizontal and vertical ranks of a map $\dif\in\mAB$ and also need to define several subspaces of $\mAB$.

\textbf{Linearized, vertical, and \MGENERIC{a}{b}\ maps.}
Let $\ScalNbh\subset\AManif$ be a smooth neighborhood of $\ASingMan$, and $\dif\colon\ScalNbh\to\BManif$ be a $\Cinfty$ map such that $\dif(\ASingMan)\subset\BSingMan$.
In particular, there exists a smaller neighborhood $\ScalNbh'\subset\ATotMan\cap\AManif$ of $\ASingMan$ such that $\dif(\ScalNbh')\subset\BTotMan$, whence, see Lemma~\ref{lm:vertical_derivative}, one can define the \term{vertical derivative} of $\dif$ being the vector bundle morphism $\vtang{\dif}\colon\ATotMan\to\BTotMan$.

Let $\apnt\in\AManif$ be a point and $\bpnt = \dif(\apnt) \in \BManif$.
Then we will denote by $\trank[\apnt]{\dif}$ the rank of the tangent map $\tang[\apnt]{\dif}\colon\tang[\apnt]{\AManif}\to\tang[\dif(\apnt)]{\BManif}$.

Further, suppose $\apnt\in\ASingMan$.
Then $\bpnt = \dif(\apnt)\in \BSingMan$, and in this case we will denote by $\hrank[\apnt]{\dif}$ the rank of the tangent map $\tang[\apnt]{(\restr{\dif}{\ASingMan})}\colon\tang[\apnt]{\ASingMan}\to\tang[\bpnt]{\BSingMan}$ and call it \term{horizontal rank} of $\dif$ at $\apnt$.

Moreover, let $\vtang[\apnt]{\dif} := \restrb{\vtang{\dif}}{\Afib{\apnt}}\colon \Afib{\apnt} \to \Bfib{\bpnt}$ be the linear map between the corresponding fibers induced by the vertical derivative of $\dif$.
Then the rank of $\vtang[\apnt]{\dif}$ will be called the \term{vertical rank} of $\dif$ at $\apnt$ and denoted by $\vrank[\apnt]{\dif}$.

We say that a map $\dif\colon(\ScalNbh,\ASingMan)\to(\BManif,\BSingMan)$ is
\begin{itemize}[topsep=1ex]
\item \term{linearized} if $\dif=\vtang{\dif}$ on some neighborhood of $\ASingMan$;
\item \term{vertical} if $\tang{\dif}(\AVect) \subset \BVect$, i.e.\  $\dif$ sends vertical vectors over $\ASingMan$ to vertical vectors over $\BSingMan$;
\item \term{\MGENERIC{a}{b}} for $a,b\geq0$ if $\hrank[\apnt]{\dif}\geq a$ and $\vrank[\apnt]{\dif}\geq b$ for every $\apnt\in\ASingMan$, and $\trank[\apnt']{\dif}\geq a+b$ for every $\apnt'\in\ScalNbh$.
\end{itemize}
Emphasize that the properties to be linearized or vertical depend on a particular choice of tubular neighborhoods $\Apr$ and $\Bpr$, while the property to be \MGENERIC{a}{b}\ does not depend on those structures.
The following simple statement will be used several times.
\begin{sublemma}\label{lm:hvranks}
Let $\gdif,\dif\colon(\ScalNbh,\ASingMan)\to(\BManif,\BSingMan)$ be two $\Cinfty$ maps of pairs such that $\vtang{\gdif} = \vtang{\dif}$.
Then $\hrank[\apnt]{\gdif} = \hrank[\apnt]{\dif}$ and $\vrank[\apnt]{\gdif} = \vrank[\apnt]{\dif}$ for all $\apnt\in\ASingMan$.
\end{sublemma}
\begin{proof}
The assumption $\vtang{\gdif} = \vtang{\dif}$ implies that
$\restr{\gdif}{\ASingMan} \stackrel{\eqref{equ:Vh_on_as__h_on_as}}{=}
    \restr{\vtang{\gdif}}{\ASingMan} =
    \restr{\vtang{\dif}}{\ASingMan} \stackrel{\eqref{equ:Vh_on_as__h_on_as}}{=}
    \restr{\dif}{\ASingMan}$.
Hence for each $\apnt\in\ASingMan$ we have that
\[
    \hrank[\apnt]{\gdif} :=
    \trank{\tang[\apnt]{(\restr{\gdif}{\ASingMan})}} =
    \trank{\tang[\apnt]{(\restr{\dif}{\ASingMan})}} =:
    \hrank[\apnt]{\dif}.
\]
Similarly,
$\vrank[\apnt]{\gdif} :=
    \trank{\vtang[\apnt]{\gdif}} =
    \trank{\vtang[\apnt]{\dif}} =:
    \vrank[\apnt]{\dif}$.
\end{proof}
Denote by $\mABlin$, $\mABvrt$, and $\mABnt{a}{b}$, the subspaces of $\mAB$ consisting of linearized, vertical, and \MGENERIC{a}{b}\ maps respectively.
Evidently, $\mABnt{a}{b}$ is $\Str{r}$-open in $\mAB$ for all $r\geq1$.
Further, given a map $\hmap\in\Ci{\ASingMan}{\BSingMan}$ and a vector bundle morphism $\vmap\colon\ATotMan\to\BTotMan$ put
\begin{align*}
    \mABh{\hmap} &:= \{ \dif\in\mAB \mid \restr{\dif}{\ASingMan} = \hmap \}, &
    \mABv{\vmap} &:= \{ \dif\in\mAB \mid \vtang{\dif} = \vmap \}.
\end{align*}
Evidently, $\mABlin \subset \mABvrt$ and $\mABv{\vmap} \subset \mABh{\restr{\vmap}{\ASingMan}}$.

\newcommand\fsubset{\subset\!\!\!\!\cdot\,\,}
\begingroup
\newcommand\Subsp{\Lman}
\textbf{Deformation into a subspace.}
Let $\Xman$ be a topological space, $\Subsp\subset\Xman$ be a subset, and $\Ghom\colon\Xman\times[0;1]\to\Xman$ be a homotopy.
Then
\begin{itemize}[leftmargin=*]
\item $\Ghom$ is a \term{deformation of $\Xman$} if $\Ghom_1=\id_{\Xman}$;
\item $\Ghom$ \term{preserves $\Subsp$} if $\Ghom(\Subsp\times[0;1])\subset\Subsp$;
\item $\Ghom$ is a \term{deformation of $\Xman$ into $\Subsp$} if $\Ghom_1=\id_{\Xman}$, $\Ghom(\Subsp\times[0;1])\subset\Subsp$, and $\Ghom_0(\Xman)\subset\Subsp$.
It is well known and easy that in this case the inclusion $\Subsp\subset\Xman$ is a \term{homotopy equivalence} and $\Ghom_0$ is its homotopy inverse.
\end{itemize}

Let $\alpha=\{\Aman_i\}_{i\in I}$ and $\beta=\{\Bman_i\}_{i\in I}$ be two families of subsets of $\Xman$ indexed by the same set $I$.
Then we will write $\alpha\fsubset\beta$, whenever $\Aman_i\subset\Bman_i$ for all $i\in I$.
Also by $\alpha\cap\Subsp = \{\Aman_i\cap\Subsp\}_{i\in I}$ we will denote the family of subsets of $\Xman$ consisting of intersections of elements of $\alpha$ with $\Subsp$.
Then $\alpha\cap\Subsp \ \fsubset \ \alpha$.

Assume that $\alpha\fsubset\beta$ and let $\Ghom\colon\Xman\times[0;1]\to\Xman$ be a deformation of $\Xman$ which \term{preserves} $\alpha$ and $\beta$ (i.e.\ each of their elements).
Say that $\Ghom$ is a \term{deformation of $\beta$ into $\alpha$} if $\Ghom_0(\Bman_i)\subset\Aman_i$ for all $i\in I$.
In this case we will also say that the inclusion $\alpha\fsubset\beta$ is a \term{homotopy equivalence of families}.

In particular, if $\Ghom$ is a deformation of $\Xman$ into $\Subsp$ which preserves $\alpha$, then $\Ghom$ is a deformation of $\alpha$ into $\alpha\cap\Subsp$, so the inclusion $\alpha\cap\Subsp \ \fsubset \ \alpha$ is a homotopy equivalence of families.
\endgroup

\subsection{Linearization of maps of pairs}
\label{sect:main_res:maps_of_pairs}
Let $\AManif$ and $\BManif$ be two connected manifolds, $\ASingMan\subset\AManif$ and $\BSingMan\subset\BManif$ be connected \GOOD\ submanifolds, equipped with tubular neighborhoods $\Apr\colon\ATotMan\to\ASingMan$ and $\Bpr\colon\BTotMan\to\BSingMan$ of $\ASingMan$ and $\BSingMan$, and $\mAB = \Ci{(\AManif,\ASingMan)}{(\BManif,\BSingMan)}$ be the space of $\Cinfty$ maps of pairs.
We will endow with the topology $\Wtr{\infty}$.
All its subspaces will be endowed with the corresponding induced topologies.
Let also $\LEF$ be the space of vector bundle morphisms $\ATotMan\to\BTotMan$.

\newcommand\calXlin{\calX_{lin}}
\newcommand\calXvrt{\calX_{vrt}}
\begin{subtheorem}\label{th:linearization_EMN_and_subspaces}
Suppose $\ASingMan$ is compact and let $a,b\geq0$.
Let also $\ANbhSingMan\subset\ATotMan\cap\AManif$ and $\BNbhSingMan\subset\BTotMan\cap\BManif$ be some neighborhoods of $\ASingMan$ and $\BSingMan$ respectively.
Assume that $\ANbhSingMan$ is compact.
Then there exists a deformation $\Ghom\colon\mABnt{a}{b}\times[0;1]\to\mABnt{a}{b}$ of $\mABnt{a}{b}$ into the subspace $\mABntlin{a}{b}$ of linearized maps such that
\begin{enumerate}[label={\rm(\alph*)}, topsep=1ex]
\item\label{enum:th:lin_main:V}
$\vtang{\Ghom(\dif,\pt)} = \vtang{\dif}$ for all $(\dif,\pt)\in\mABnt{a}{b}\times[0;1]$;
\item\label{enum:th:lin_main:support}
for each $\dif\in\mABnt{a}{b}$ there exists a neighborhood $\ANbhSingMan[\dif] \subset \ANbhSingMan$ of $\ASingMan$ (depending on $\dif$) such that for all $\pt\in[0;1]$
\begin{align}\label{equ:G_prop_map_U}
    &\Ghom(\dif,\pt)(\ANbhSingMan[\dif]) \subset \BNbhSingMan, &
    &\restr{\Ghom(\dif,\pt)}{\AManif\setminus\ANbhSingMan[\dif]} = \restr{\dif}{\AManif\setminus\ANbhSingMan[\dif]};
\end{align}
\item\label{enum:th:lin_main:strong_top}
$\Ghom$ is continuous if we endow $\mABnt{a}{b}$ with the topology $\Str{\infty}$.

\item\label{enum:th:lin_main:alpha}
Consider the following family of subsets of $\mABnt{a}{b}$:
\begin{equation}\label{equ:th:lin_main:alpha_family}
    \alpha = \bigl\{ \mABvrt, \ \mABh{\hmap}, \ \mABv{\vmap}, \ \mABnt{a}{b}\bigr\} \ \cap \ \mABnt{a}{b},
\end{equation}
where $\hmap\in\Ci{\ASingMan}{\BSingMan}$ and $\vmap\in\LEF$.
Then $\Ghom$ preserves $\alpha$ and therefore the inclusion of families
\[
    \alpha \cap \mABlin \ \fsubset \ \alpha
\]
is simultaneously a $\Wtr{\infty}$- and $\Str{\infty}$-homotopy equivalence.
\end{enumerate}
\end{subtheorem}

\begin{subremark}\rm
a)~If $\ardim=0$, i.e.\ $\dim\AManif=\dim\ASingMan$, then $\ATotMan = \ASingMan$, and $\mAB = \mABlin$.
Therefore we will always assume further that $\ardim\geq1$.

b)~Evidently, if $a'\geq a$ and $b'\geq b$, then $\mABnt{a'}{b'}\subset\mABnt{a}{b}$.
However the homotopy $\Ghom$ constructed for the pair $(a,b)$ does not necessarily preserve the subset $\mABnt{a'}{b'}$.
Nevertheless, we still have some other homotopy $\Ghom'$ of $\mABnt{a'}{b'}$ into $\mABntlin{a'}{b'}$.
\end{subremark}

\paragraph{\bf Linearization of neat immersions, submersions, and local diffeomorphisms}
Consider now important particular cases of the spaces $\mABnt{a}{b}$, which actually may be empty.
\begin{enumerate}[wide, label={\rm(\alph*)}]
\item
Note that $\mAB = \mABnt{0}{0}$, whence by Theorem~\ref{th:linearization_EMN_and_subspaces} the inclusions
\begin{align*}
&\mABlin \subset \mABvrt \subset \mAB, &
&\mABhlin{\hmap} \subset \mABh{\hmap}, &
&\EMNnbh{\vmap}  \subset \mABv{\vmap}
\end{align*}
are homotopy equivalences.

\item
Suppose $\axdim \leq \bxdim$ and $\ardim \leq \brdim$, and let $\dif\in\mAB$.
Then the following conditions are equivalent:
\begin{enumerate}[label={\rm(\roman*)}]
\item $\dif\in\mABnt{\axdim}{\ardim}$;
\item $\trank[\apnt]{\dif} = \axdim + \ardim$ for each $\apnt\in\AManif$, and for each $\apnt\in\ASingMan$ the maps
\begin{align}\label{equ:tang_vert_maps_x}
    &\tang[\apnt]{\dif}\colon\tang[\apnt]{\ASingMan}\to\tang[\dif(\apnt)]{\BSingMan} &
    &\vtang[\apnt]{\dif}\colon\Afib{\apnt}\to\Bfib{\dif(\apnt)}
\end{align}
are injective.
\end{enumerate}

In particular, $\dif$ and $\restr{\dif}{\ASingMan}$ are immersions, though, in general, $\mABnt{\axdim}{\ardim}$ does not coincide with the subspace of all immersions $\dif\colon\AManif\to\BManif$ for which $\restr{\dif}{\ASingMan}$ is an immersion as well.
We will call elements of $\mABnt{\axdim}{\ardim}$ \term{neat immersions} and denote that space by $\ImmAB:=\mABnt{\axdim}{\ardim}$.

\begin{subexample}\rm
Let $\Rman$ be any manifold.
Consider two trivial vector bundles
\begin{align*}
    &\Apr\colon\AManif=\Rman^{2}\times\Rman \to \Rman^{2}=\ASingMan, &
    &\Bpr\colon\BManif=\Rman^{3}\times\Rman^{2} \to \Rman^{3}=\BSingMan,
\end{align*}
and the following map of pairs $\dif\colon (\AManif,\ASingMan) \to (\BManif,\BSingMan)$:
\begin{align*}
    \dif(\pxi{1},\pxi{2};\pvi{1}) = (\pxi{1},\pxi{2},\pvi{1};0,0).
\end{align*}
Then $\dif$ and $\restr{\dif}{\ASingMan}$ are immersions, however $\dif$ is not a neat immersion.
\end{subexample}

\item
Similarly, let $\axdim \geq \bxdim$ and $\ardim \geq \brdim$ and $\dif\in\mAB$.
Then one easily checks that the following conditions are equivalent:
\begin{enumerate}[label={\rm(\roman*)}]
\item $\dif\in\mABnt{\bxdim}{\brdim}$;
\item $\trank[\apnt]{\dif} = \bxdim + \brdim$ for each $\apnt\in\AManif$, and for each $\apnt\in\ASingMan$ both maps~\eqref{equ:tang_vert_maps_x} are surjective;
\item $\dif$ is a \term{submersion of pairs}, i.e.\ $\dif$ and $\restr{\dif}{\ASingMan}$ are submersions.
\end{enumerate}
Thus in contrast to immersions, $\mABnt{\bxdim}{\brdim}$ consists of submersions of pairs.
We will denote that space by $\SubAB:=\mABnt{\bxdim}{\brdim}$.

\item
If $\axdim = \bxdim$ and $\ardim = \brdim$, then $\ImmAB := \mABnt{\axdim}{\ardim} = \mABnt{\bxdim}{\brdim} =: \SubAB$.
This space consists of \term{local diffeomorphisms} and we will denote it by $\dABloc$.
\end{enumerate}

Since $\ImmAB$, $\SubAB$, and $\dABloc$ are the spaces $\mABnt{a}{b}$ for some $a,b\geq0$ we get the following particular case of Theorem~\ref{th:linearization_EMN_and_subspaces}:
\begin{subcorollary}\label{cor:linearization_EMNab}
Let $\calX$ be one of the spaces $\ImmAB$, $\SubAB$, $\dABloc$, so $\calX = \mABnt{a}{b}$ for some $a,b\geq0$.
Then the inclusion $\calX \cap \mABlin \subset \calX$ is a homotopy equivalence.
\end{subcorollary}

\paragraph{\bf Linearization of diffeomorphisms}
A map $\dif\in\mAB$ is a \term{diffeomorphism of pairs $(\AManif,\ASingMan)$ and $(\BManif,\BSingMan)$} if $\dif\colon\AManif\to\BManif$ and the restriction $\restr{\dif}{\ASingMan}\colon\ASingMan\to\BSingMan$ are diffeomorphisms.
Let
\begin{itemize}
\item $\dAB=\Diff(\AManif,\BManif)$ be the set of all diffeomorphisms $\AManif\to\BManif$;
\item $\DiffAXBY=\Diff\bigl((\AManif,\ASingMan),(\BManif,\BSingMan)\bigr)$ be the subset of $\mAB$ consisting of all diffeomorphisms of those pairs;
\item $\DiffAXBYlin = \DiffAXBY\cap\mABlin$ be the set of linearized diffeomorphisms of pairs.
\end{itemize}

\begin{sublemma}\label{lm:diffs_of_pairs}
Suppose $\ASingMan$ is compact and connected and $\DiffAXBY\ne\varnothing$.
Then $\DiffAXBY = \dAB \cap \mAB$.
In other words, if pairs $(\AManif,\ASingMan)$ and $(\BManif,\BSingMan)$ are diffeomorphic, then for every diffeomorphism $\dif\colon\AManif\to\BManif$ such that $\dif(\ASingMan)\subset\BSingMan$, the restriction $\restr{\dif}{\ASingMan}\colon\ASingMan\to\BSingMan$ is a diffeomorphism as well.
\end{sublemma}
\begin{proof}
1) Evidently, $\DiffAXBY \subset \dAB \cap \mAB$, i.e.\ every $\dif\in\DiffAXBY$ is a diffeomorphism between $\AManif$ and $\BManif$ which also sends $\ASingMan$ into $\BSingMan$.

2) Conversely, let $\dif\in\dAB \cap \mAB$, so $\dif\colon\AManif\to\BManif$ is a diffeomorphism such that $\dif(\ASingMan)\subset\BSingMan$.
Since $\ASingMan$ is compact, $\dif(\ASingMan)$ is a closed subset of $\BSingMan$.
On the other hand, since $\DiffAXBY\ne\varnothing$, we obtain that $\ASingMan$ and $\BSingMan$ are diffeomorphic, whence the restriction $\restr{\dif}{\ASingMan}\colon\ASingMan\to\BSingMan$ is an injective local diffeomorphism.
In particular, $\restr{\dif}{\ASingMan}$ is an open map, so $\dif(\ASingMan)$ is also open in $\BSingMan$.
But $\BSingMan$ is connected, whence $\dif(\ASingMan)=\BSingMan$.
Thus $\restr{\dif}{\ASingMan}$ is a bijective local diffeomorphism, and therefore it is a diffeomorphism.
\end{proof}

The following \term{linearization theorem for diffeomorphisms} shows that for $\axdim = \bxdim$ and $\ardim = \brdim$ the diffeomorphisms are also preserved by the homotopy $\Ghom$ from Theorem~\ref{th:linearization_EMN_and_subspaces}.
However, this requires additional arguments and mild assumptions.

\begin{subtheorem}\label{th:linearization_DiffMN}
Suppose $\ASingMan$ is compact, $\axdim = \bxdim$, $\ardim = \brdim\geq1$, and $\DiffAXBY\ne\varnothing$.
Let also $\ANbhSingMan \subset \ATotMan\cap\AManif$ be a compact neighborhood of $\ASingMan$, $\BNbhSingMan \subset \BTotMan\cap\BManif$ be a neighborhood of $\BSingMan$ which does not coincide with all of $\BManif$, and $\Ghom$ be the corresponding deformation of $\mABnt{\axdim}{\ardim}$ into $\mABntlin{\axdim}{\ardim}$ from Theorem~\ref{th:linearization_EMN_and_subspaces}.
Then $\Ghom$ preserves $\DiffAXBY$, so the inclusion $\DiffAXBYlin\subset\DiffAXBY$ is a homotopy equivalence.

More generally, consider the following family of subsets of $\DiffAXBY$:
\[
    \alpha = \bigl\{
        \mABvrt, \ \mABh{\hmap}, \ \mABv{\vmap}, \ \mABnt{a}{b} = \dABloc
    \bigr\} \ \cap \ \DiffAXBY,
\]
where $\hmap\in\Ci{\ASingMan}{\BSingMan}$ and $\vmap\in\LEF$.
Then $\Ghom$ preserves $\alpha$, whence the inclusion of families $\ \alpha \cap \mABlin\,\fsubset\,\alpha\ $ is simultaneously a $\Wtr{\infty}$- and $\Str{\infty}$-homotopy equivalence of families.
\end{subtheorem}
\begin{proof}
Let $\dif\in\DiffAXBY \subset \dABloc$ and $\dif_{\pt}:=\Ghom(\dif,\pt)$, $\pt\in[0;1]$.
Since $\dABloc$ is $\Ghom$-invariant, $\dif_{\pt}\in\dABloc$ due to Corollary~\ref{cor:linearization_EMNab}, so it is already a local diffeomorphism, and thus an open map.

We need to show that $\dif_{\pt}$ is in fact a bijection.
The latter is implied by the following statement.
\begin{claim}[{\rm\cite[Lemma~3]{Ho:ProcAMS:1975}}]
\label{lm:loc_homeo_is_homeo}
Let $\dif\colon\AManif\to\BManif$ be a surjective proper local homeomorphism between Hausdorff spaces.
If $\BManif$ is path connected, then $\dif$ is a covering map.
In particular, if $\AManif$ is connected manifold and there exists a point $\py\in\BManif$ such that $\dif^{-1}(\py)$ is a single point, then $\dif$ is a homeomorphism.
\end{claim}

Indeed, by Theorem~\ref{th:linearization_EMN_and_subspaces}, there exists a neighborhood $\ANbhSingMan[\dif]\subset\ANbhSingMan$ of $\ASingMan$ such that for all $\pt\in[0;1]$
\begin{align}
    \label{equ:htV_in_W}
    \dif_{\pt}(\ANbhSingMan[\dif]) &:=
    \Ghom(\dif,\pt)(\ANbhSingMan[\dif]) \subset \BNbhSingMan, \\
    \label{equ:ht_AV__H_AV}
    \restr{\dif_{\pt}}{\AManif\setminus\ANbhSingMan[\dif]} &:=
    \restr{\Ghom(\dif,\pt)}{\AManif\setminus\ANbhSingMan[\dif]} = \restr{\dif}{\AManif\setminus\ANbhSingMan[\dif]}.
\end{align}
Let $\bpnt\in\BManif\setminus\BNbhSingMan$.
Then by~\eqref{equ:htV_in_W}, $\dif_{\pt}^{-1}(\bpnt) \subset \AManif\setminus\ANbhSingMan[\dif]$.
But since $\dif$ is a diffeomorphism, the restriction $\restr{\dif}{\AManif\setminus\ANbhSingMan[\dif]} \stackrel{\eqref{equ:ht_AV__H_AV}}{=} \restr{\dif_{\pt}}{\AManif\setminus\ANbhSingMan[\dif]}$ is injective.
Hence $\dif_{\pt}^{-1}(\bpnt)=\dif^{-1}(\bpnt)$ is a single point.

Moreover, since $\dif_{\pt}$ coincides with $\dif$ out of a compact set $\overline{\ANbhSingMan[\dif]} \subset \overline{\ANbhSingMan}$ and $\dif$ is a proper map, we see that $\dif_{\pt}$ is also proper.

Let us show that $\dif_{\pt}(\AManif) = \BManif$.
Indeed, since $\dif_{\pt}$ is a local diffeomorphism, $\dif_{\pt}(\AManif)$ is an open subset of $\BManif$.
On the other hand, as $\dif_{\pt}$ is proper, $\dif_{\pt}(\AManif)$ is also closed in $\BManif$.
Since $\BManif$ is connected, we obtain that $\dif_{\pt}(\AManif) = \BManif$.

Hence by Claim~\ref{lm:loc_homeo_is_homeo}, $\dif_{\pt}$ is a bijective local diffeomorphism, and thus a diffeomorphism.
\end{proof}

\subsection{Linearization of foliated maps}
\label{sect:lineatizaion_fol_mpas_of_pairs}

\newcommand\ASFol{\ASingMan_{\AFol}}
\newcommand\BSFol{\BSingMan_{\BFol}}

Let $\AFol$ be some partition of a set $\AManif$ and $\AFolPr\colon\AManif\to\AFol$ be the natural quotient map%
\footnote{Note that $\AFol$ is a \term{set} of certain subsets of $\AManif$, so $\AFolPr$ is a map between sets.} associating to each $\apnt\in\AManif$ that unique element $\AFolPr(\apnt)$ of $\AFol$ containing $\apnt$.
We will refer to elements of $\AFol$ as \term{leaves}.
A subset $\ASingMan\subset\AManif$ is called \term{$\AFol$-saturated} if it is a union of leaves of $\AFol$, i.e.\ $\ASingMan = \AFolPr^{-1}(\AFolPr(\ASingMan))$.
We also denote by $\ASFol:=\AFolPr(\ASingMan)$ the set of leaves intersecting $\ASingMan$.

Let $\BFol$ be a partition of some other set $\BManif$ with the corresponding quotient map $\BFolPr\colon\BManif\to\BFol$.
Let also $\ASingMan\subset\AManif$ and $\BSingMan\subset\BManif$ be some (possibly empty) $\AFol$- and $\BFol$-saturated subsets, respectively, and $\folmap\colon(\AFol, \ASFol) \to (\BFol, \BSFol)$ be some map of pairs.
Then for each $\aleaf\in\AFol$ it will be convenient to say that the leaves $\aleaf$ and $\folmap(\aleaf)\in\BFol$ are \term{$\folmap$-related}.
Say that $\folmap$ is \term{\FOLNEAT} (as a map of pairs), if $\ASFol = \folmap^{-1}(\BSFol)$.

A map $\dif\colon\AManif\to\BManif$ will be called \term{\ftfol}, if $\BFolPr\circ\dif = \folmap \circ \AFolPr$, i.e.\ the following commutative diagram holds:
\begin{equation}\label{equ:phi-foliated-map}
\begin{gathered}
\xymatrix@R=2em{
    \ASingMan  \ar@/^11pt/[rrr]^{\dif} \ar@{^(->}[r] \ar[d]^-{\AFol} & \AManif \ar[r]^-{\dif}  \ar[d]_-{\AFolPr}  & \BManif \ar[d]^-{\BFolPr} & \BSingMan \ar@{_(->}[l]  \ar[d]^-{\BFol} \\
    \ASFol \ar@{^(->}[r] \ar@/_8pt/[rrr]_{\folmap} & \AFol   \ar[r]^-{\folmap}  & \BFol & \BSFol \ar@{_(->}[l]
}
\end{gathered}
\end{equation}

In other words, $\dif$ sends each leaf $\aleaf\in\AFol$ into its $\folmap$-related leaf $\folmap(\aleaf)\in\BFol$, i.e.\ $\dif(\aleaf) \subset \folmap(\aleaf)$.
The following lemma is trivial:
\begin{sublemma}
Let $\dif\colon\AManif\to\BManif$ be a \ftfol\ map.
Then $\dif(\ASingMan) \subset \BSingMan$.
If $\folmap$ is \FOLNEAT, then $\ASingMan = \dif^{-1}(\BSingMan)$.
\end{sublemma}

Assume now that $\AManif$ and $\BManif$ are two manifolds equipped with some partitions $\AFol$ and $\BFol$ respectively, and let $\AFolPr\colon\AManif\to\AFol$ and $\BFolPr\colon\BManif\to\BFol$ be the corresponding quotient maps.
Let also $\ASingMan\subset\AManif$ be a \GOOD\ $\AFol$-saturated submanifold, and $\Apr\colon\ATotMan\to\ASingMan$ be some tubular neighborhood of $\ASingMan$.
Similarly, let $\BSingMan\subset\BManif$ be a \GOOD\ $\BFol$-saturated submanifold with a tubular neighborhood $\Bpr\colon\BTotMan\to\BSingMan$.
As above denote $\mAB = \Ci{(\AManif,\ASingMan)}{(\BManif,\BSingMan)}$.
\begin{subdefinition}\label{def:neat-vert}
A map $\dif\in\mAB$ will be called \term{\VERTINJECTIVE} if either of the following equivalent statements hold:
\begin{enumerate}[topsep=1ex]
\item for each $\apnt\in\ASingMan$ the vertical map $\vtang[\apnt]{\dif}\colon\Afib{\apnt}\to\Bfib{\dif(\apnt)}$ is injective;
\item $\ASingMan = (\vtang{\dif})^{-1}(\BSingMan)$.
\end{enumerate}
\end{subdefinition}

\begin{subdefinition}\label{def:G-isol_nbg}
A neighborhood $\BScalNbh\subset\BTotMan\cap\BManif$ of $\BSingMan$ will be called \term{$\BFol$-isolating} if $(\overline{\bleaf}\setminus\bleaf) \cap \BScalNbh \subset \BSingMan$ for every leaf $\bleaf\in\BFol$.
\end{subdefinition}

\begin{subdefinition}\label{def:folmap-scalable}
Let $\AScalNbh \subset \ATotMan\cap\AManif$ and $\BScalNbh\subset\BTotMan\cap\BManif$ be some neighborhoods of $\ASingMan$ and $\BSingMan$, respectively, and $\folmap\colon(\AFol,\ASFol)\to(\BFol,\BSFol)$ be a map of pairs.
Say that $(\AScalNbh,\BScalNbh)$ is \term{$\folmap$-scalable} (with respect to the tubular neighborhoods $\Apr$ and $\Bpr$) if for any $\px\in\AScalNbh$, $\py\in\BScalNbh$ and $\tau\geq1$ satisfying
\begin{align*}
    &\folmap(\AFolPr(\px)) = \BFolPr(\py), &
    &\tau\px\in\AScalNbh, &
    &\tau\py\in\BScalNbh,
\end{align*}
we have that $\folmap(\AFolPr(\tau\px)) = \BFolPr(\tau\py)$.
\end{subdefinition}
In other words, $\folmap$-scalability of $(\AScalNbh,\BScalNbh)$ means that if $\px,\tau\px\in\AScalNbh$ and $\py,\tau\py\in\BScalNbh$ for some $\tau\geq1$ and $\px,\py$ belong to $\folmap$-related leaves, then $\tau\px$ and $\tau\py$ also belong to $\folmap$-related leaves.

\begin{sublemma}\label{lm:decr_scal_isol_nbhs}
Let $\AScalNbh' \subset \AScalNbh \subset \ATotMan\cap\AManif$ and $\BScalNbh'\subset\BScalNbh\subset\BTotMan\cap\BManif$ be pairs of neighborhoods of $\ASingMan$ and $\BSingMan$ respectively.
If $\BScalNbh$ is a $\BFol$-isolating neighborhood for $\BSingMan$, then $\BScalNbh'$ is $\BFol$-isolating for $\BSingMan$ as well.
If $(\AScalNbh,\BScalNbh)$ is $\folmap$-scalable, then $(\AScalNbh',\BScalNbh')$ is also $\folmap$-scalable.
\end{sublemma}
\begin{proof}
Suppose $\BScalNbh$ is $\BFol$-isolating for $\BSingMan$.
Then for every leaf $\bleaf\in\BFol$ we have that
\[
    (\overline{\bleaf}\setminus\bleaf) \cap \BScalNbh'
        \subset
    (\overline{\bleaf}\setminus\bleaf) \cap \BScalNbh
        \subset
    \BSingMan,
\]
so $\BScalNbh'$ is also $\BFol$-isolating for $\BSingMan$.

Suppose $(\AScalNbh,\BScalNbh)$ is $\folmap$-scalable.
Let $\px\in\AScalNbh'$ and $\py\in\BScalNbh'$ be two points such that $\folmap(\AFolPr(\px)) = \BFolPr(\py)$, $\tau\px\in\AScalNbh'\subset\AScalNbh$, and $\tau\py\in\BScalNbh'\subset\BScalNbh$.
Since $(\AScalNbh,\BScalNbh)$ is $\folmap$-scalable, $\folmap(\AFolPr(\tau\px)) = \BFolPr(\tau\py)$.
Hence $(\AScalNbh',\BScalNbh')$ is $\folmap$-scalable as well.
\end{proof}

Let $\folmap\colon(\AFol,\ASFol)\to(\BFol,\BSFol)$ be a \FOLNEAT\ map of pairs.
Denote by
\begin{align*}
    \mABfol{\folmap} &= \{
        \dif\in\mAB \mid
        \BFolPr\circ\dif = \folmap \circ \AFolPr, \
        \ASingMan = (\vtang{\dif})^{-1}(\BSingMan)
    \}
\end{align*}
the subset of $\mAB$ consisting of all \ftfol\ maps being \VERTINJECTIVE.

Our main result is the following \term{foliated variant of the linearization theorem} whose proof will be given in \SECT{sect:proof:th:linearization_fol}.
\begin{subtheorem}\label{th:linearization_fol}
Let $a,b\geq0$ and $\folmap\colon(\AFol,\ASFol)\to(\BFol,\BSFol)$ be a \FOLNEAT\ map of pairs.
Suppose that $\ASingMan$ is compact, there exist a compact neighborhood $\AScalNbh\subset\ATotMan\cap\AManif$ of $\ASingMan$ and a neighborhood $\BScalNbh\subset\BTotMan\cap\BManif$ of $\BSingMan$ such that $(\AScalNbh,\BScalNbh)$ is $\folmap$-scalable and $\BScalNbh$ is $\BFol$-isolating.
Let also $\Ghom$ be the deformation of $\mABnt{a}{b}$ into $\mABntlin{a}{b}$ from Theorem~\ref{th:linearization_EMN_and_subspaces} constructed for $\ANbhSingMan = \AScalNbh$ and $\BNbhSingMan = \BScalNbh$.
Then $\Ghom$ preserves the following family of subsets of $\mABfol{\folmap}$:
\[
    \alpha = \bigl\{
        \mABvrt, \ \mABh{\hmap}, \ \mABv{\vmap}
    \bigr\} \ \cap \ \mABnt{a}{b} \ \cap \ \mABfol{\folmap}
\]
where $\hmap\in\Ci{\ASingMan}{\BSingMan}$ and $\vmap\in\LEF$.
Hence the inclusion of families
\[
    \alpha\cap\mABlin \ \fsubset \ \alpha
\]
is simultaneously a $\Wtr{\infty}$- and $\Str{\infty}$-homotopy equivalence of families.
\end{subtheorem}

\section{Linearization for self-maps of a pair of manifolds}
\label{sect:self_maps_of_a_pair}
In this section we will reformulate some of the statements from the previous section for the situation when $(\BManif,\Bpr\colon\BTotMan\to\BSingMan)$ coincides with $(\AManif,\Apr\colon\ATotMan\to\ASingMan)$ emphasizing certain particular cases presented in a preliminary preprint version of the present paper, and already used in other papers, e.g.~\cite{Maksymenko:ProcIMNANUS:2023, KhokhliukMaksymenko:JHRS:2023, Maksymenko:JHRS:2024}, \cite[Theorems~2.2.5 \& 3.1.1]{Maksymenko:VBAuto:2024}.

\subsection{Self-maps of a pair}
\label{sect:self_maps_of_a_pair:non-fol}
Let $\AManif$ be a manifold, $\ASingMan$ be a \GOOD\ submanifold, and $\Apr\colon\ATotMan\to\ASingMan$ be a tubular neighborhood of $\ASingMan$ in $\AManif$.
Denote by
\[
    \eAX := \mAB = \mABnt{0}{0} = \Ci{(\AManif,\ASingMan)}{(\AManif,\ASingMan)}
\]
the set of $\Cinfty$ self-maps of $\AManif$ preserving $\ASingMan$, i.e.\ \term{endomorphisms} of the pair $(\AManif,\ASingMan)$.
We endow $\eAX$ with the corresponding topology $\Wtr{\infty}$.

Note that for each $\dif\in\eAX$ its vertical derivative with respect to the same tubular neighborhood $\Apr$ of $\ASingMan$ is a vector bundle morphism $\vtang{\dif}\colon\ATotMan\to\ATotMan$.
Let also
\begin{itemize}[parsep=0.5ex, itemsep=0.5ex]
\item
$\eAXnt{a}{b} := \mABnt{a}{b}$, $a,b\geq0$, be the subset of $\mAB$ consisting of \MGENERIC{a}{b}\ maps $\dif$, i.e.\ $\hrank[\apnt]{\dif} \geq a$, $\vrank[\apnt]{\dif} \geq b$ for each $\apnt\in\ASingMan$, and $\trank[\apnt]{\dif} \geq a+b$ for each $\apnt\in\AManif$;

\item
$\eAXvrt := \mABvrt = \{ \dif\in\mAB \mid \tang[\ASingMan]{\dif}(\AVect) \subset \AVect\}$ be the subset of $\mAB$ leaving the set $\AVect$ of \term{vertical} vectors over $\ASingMan$ invariant;

\item
$\eAXfix := \mABh{\id_{\ASingMan}} = \{ \dif\in\mAB \mid \restr{\dif}{\ASingMan} = \id_{\ASingMan} \}$ be the subset of $\eAX$ consisting of maps pointwise \term{fixed} on $\ASingMan$;

\item
$\eAXvid = \mABv{\id_{\ATotMan}} = \{ \dif\in\mAB \mid \vtang{\dif} = \id_{\ATotMan} \}$ be the subset consisting of maps whose vertical derivative is the identity map $\id_{\ATotMan}$;

\item $\eAXL := \mABlin = \{\dif\in\eAX \mid \dif = \vtang{\dif} \ \text{near} \ \ASingMan\}$ be the subspace of linearized maps.
\end{itemize}
Define also the ``linearized'' variants of these spaces by intersecting them with $\eAXL$.
Thus let
\begin{itemize}
\item $\eAXLnt{a}{b} := \mABntlin{a}{b} = \eAXnt{a}{b} \cap \mABlin$, $a,b\geq0$;
\item $\eAXLfix := \mABh{\id_{\ASingMan}} \cap \mABlin = \eAXfix \cap \eAXL$;
\item $\eAXnb  := \mABv{\id_{\ATotMan}} \cap \mABlin = \eAXvid \cap \eAXL$, be the subspace consisting of maps which coincide with $\id_{\ATotMan}$ near $\ASingMan$, i.e.\ \term{pointwise fixed on some neighborhood} of $\ASingMan$.
\end{itemize}
Clearly, $\eAXL\subset\eAXvrt$.
Also note that $\eAXnt{\axdim}{\ardim} = \mABnt{\axdim}{\ardim} = \dABloc$ is the submonoid of $\eAX$ consisting of local diffeomorphisms.
Let $\Diff(\AManif)$ be the group of all diffeomorphisms of $\AManif$, and $\dAX$ be its subgroup consisting of diffeomorphisms of the pair $(\AManif,\ASingMan)$.
Then $\dAX \subset \eAXnt{\axdim}{\ardim}$.
Moreover, if $\ASingMan$ is compact, then by Lemma~\ref{lm:diffs_of_pairs},
\[
    \dAX = \eAX \cap \Diff(\AManif) = \eAXnt{\axdim}{\ardim} \cap \Diff(\AManif).
\]
Let also
\begin{align*}
    \dAXvrt  &= \eAXvrt \cap \dAX, &
    \dAXfix  &:= \eAXfix \cap \dAX, &
    \dAXvid  &:= \eAXvid \cap \dAX,  \\
    \dAXL    &:= \eAXL \cap \dAX, &
    \dAXLfix &:= \eAXLfix \cap \dAX, &
    \dAXnb   &:= \eAXnb \cap \dAX.
\end{align*}
be the corresponding subgroups of the above monoids consisting of diffeomorphisms.

The following theorem includes~\cite[Theorem~2.2.5]{Maksymenko:VBAuto:2024} stated without a proof:
\begin{subtheorem}\label{th:linearization:diff_pres_submanif}
Let $a,b\geq0$, $\ASingMan$ be a compact connected \GOOD\ submanifold of a connected manifold $\AManif$, $\Apr\colon\ATotMan\to\ASingMan$ be any tubular neighborhood of $\ASingMan$, and $\ScalNbh\subset\ATotMan\cap\AManif$ be a neighborhood of $\ASingMan$.
Then there exists a deformation $\Ghom\colon\eAXnt{a}{b}\times[0;1]\to\eAXnt{a}{b}$ of $\eAXnt{a}{b}$ into $\eAXLnt{a}{b}$ such that
\begin{enumerate}[label={\rm(\alph*)}]
\item\label{enum:th:linearization:diff_pres_submanif:subsp}
$\Ghom$ preserves the intersections of $\eAXvrt$, $\eAXfix$, and $\eAXvid$ with $\eAXnt{a}{b}$;
\item\label{enum:th:linearization:diff_pres_submanif:support}
$\restr{\Ghom(\dif,\pt)}{\AManif\setminus\ScalNbh} = \restr{\dif}{\AManif\setminus\ScalNbh}$ for all $(\dif,\pt)\in \eAXnt{a}{b}\times[0;1]$.
\end{enumerate}
In particular, the following inclusion of families:
\[
    \bigl\{ \eAXnb, \ \eAXLfix, \ \eAXL,   \ \eAXL \bigr\}
    \cap \eAXnt{a}{b}
    \ \fsubset \
    \bigl\{ \eAXvid, \ \eAXfix, \ \eAXvrt, \ \eAX  \bigr\}
    \cap \eAXnt{a}{b}
\]
is simultaneously a $\Wtr{\infty}$- and $\Str{\infty}$-homotopy equivalence of families.

Moreover, if $a=\axdim$ and $b=\ardim\geq1$, then one can choose $\Ghom$ so as to preserve $\dAB$.
Hence, the inclusions of families
\begin{align}\label{equ:incl_Diff}
    \bigl\{\dAXnb, \ \dAXLfix, \ \dAXL,  \ \dAXL\bigr\}
    \ \fsubset \
    \bigl\{\dAXvid, \ \dAXfix, \ \dAXvrt, \ \dAX\bigr\}
\end{align}
is simultaneously a $\Wtr{\infty}$- and $\Str{\infty}$-homotopy equivalence of families.
\end{subtheorem}
\begin{proof}
Let $\Ghom\colon\mABnt{a}{b}\times[0;1]\to\mABnt{a}{b}$ be the deformation from Theorem~\ref{th:linearization_EMN_and_subspaces} of $\mABnt{a}{b}$ into $\mABntlin{a}{b}$ constructed for $\ANbhSingMan = \BNbhSingMan = \ScalNbh$.
Then $\Ghom$ satisfies~\ref{enum:th:linearization:diff_pres_submanif:support} and also preserves the spaces $\eAXvrt := \mABvrt$, $\eAXfix := \mABh{\id_{\ASingMan}}$ and $\eAXvid := \mABv{\id_{\ATotMan}}$, which proves~\ref{enum:th:linearization:diff_pres_submanif:subsp} as well.

Suppose $a=\axdim$ and $b=\ardim \geq 1$.
Then one can choose $\Vman_{\ASingMan}=\Wman_{\ASingMan} \subset \ScalNbh$ to have compact closure and not coinciding with $\AManif$.
Now by Theorem~\ref{th:linearization_DiffMN} the corresponding homotopy $\Ghom$ will also preserve the group $\dAB$.
\end{proof}

\subsection{Leaf preserving maps}
Let $\AFol$ be some partition of a set $\AManif$.
Then a map $\dif\colon\AManif\to\AManif$ will be called
\begin{enumerate}[nosep, leftmargin=*]
\item \term{\FLPMap} if $\dif(\aleaf)\subset\aleaf$ for each leaf $\aleaf\in\AFol$;
\item a \term{\FLPBij} if $\dif$ is a bijection and $\dif(\aleaf)=\aleaf$ for each $\aleaf\in\AFol$;
\end{enumerate}
\begin{sublemma}\label{lm:leaf_pres_bijections}
Let $\dif\colon\AManif\to\AManif$ be a bijection.
Then $\dif$ is a \term{\FLPMap} iff it is a \term{\FLPBij}.
\end{sublemma}
\begin{proof}
\textit{Necessity.}
Clearly, a \FLPBij\ is \FLPMap.

\textit{Sufficiency.}
Suppose that the bijection $\dif\colon\AManif\to\AManif$ is a \FLPMap, i.e.\ $\dif(\aleaf) \subset \aleaf$ for all $\aleaf\in\AFol$.
We should prove that then $\dif(\aleaf) = \aleaf$, i.e.\ it is also surjective at each leaf $\aleaf\in\AFol$.

Indeed, let $\apnt\in\aleaf$ be any point, $\apnt' = \dif^{-1}(\apnt)$, and $\aleaf'\in\AFol$ be the leaf containing $\apnt'$.
Then $\dif(\aleaf') \subset \aleaf'$, and therefore $\apnt\in\aleaf\cap\aleaf'\ne\varnothing$.
Since $\AFol$ is a partition of $\AManif$, we have that $\aleaf=\aleaf'$.
In particular, $\apnt = \dif(\apnt')$, where $\apnt'\in\aleaf$.
Thus the restriction $\restr{\dif}{\aleaf}\colon\aleaf\to\aleaf$ is surjective.
\end{proof}

Now let $\AManif$ be a manifold, $\AFol$ be a partition of $\AManif$, and $\AFolPr\colon\AManif\to\AFol$ be the corresponding quotient map.
Let also $\ASingMan \subset \AManif$ be a \GOOD\ $\AFol$-saturated submanifold, $\Apr\colon\ATotMan\to\ASingMan$ be a tubular neighborhood of $\ASingMan$, and $\ASFol = \AFolPr(\ASingMan) \subset \AFol$ be the set of all leaves contained in $\ASingMan$.
Note that then the identity map $\id_{\AFol}$ is a map of pairs $\id_{\AFol}\colon(\AFol,\ASFol)\to(\AFol,\ASFol)$ being also \FOLNEAT, i.e.\ $\id_{\AFol}^{-1}(\ASFol)=\ASFol$, see \SECT{sect:lineatizaion_fol_mpas_of_pairs}.

\begin{sublemma}\label{lm:F-tau-leaf-pres-maps}
Let $\dif\colon(\AManif,\ASingMan)\to(\AManif,\ASingMan)$ be a \FLPMap\ diffeomorphism.
Then $\dif$ is a \FLPBij\ being \VERTINJECTIVE, i.e.\ $(\vtang{\dif})^{-1}(\ASingMan) = \ASingMan$, see Definition~\ref{def:neat-vert}.
\end{sublemma}
\begin{proof}
Since $\dif$ is a \FLPMap\ bijection, we have by Lemma~\ref{lm:leaf_pres_bijections} that $\dif$ is also a \FLPBij.
Moreover, as $\dif$ is a diffeomorphism, it follows from Lemma~\ref{lm:vertical_derivative}\ref{enum:lm:vertical_derivative:diff_yields_iso} that $\vtang{\dif}\colon\ATotMan\to\ATotMan$ is a vector bundle isomorphism, and therefore $(\vtang{\dif})^{-1}(\ASingMan) = \ASingMan$.
\end{proof}

Denote
\begin{align*}
    \eAXfol  &:= \{ \dif\in\Ci{\AManif}{\AManif} \mid \dif(\aleaf)\subset\aleaf, \ \aleaf\in\AFol \}, \\
    \eAXvfol &:= \{ \dif\in\eAXfol \mid (\vtang{\dif})^{-1}(\ASingMan)=\ASingMan \}, \\
    \dAXfol  &:= \eAXfol \cap \Diff(\AManif).
\end{align*}
Thus $\eAXfol$ is the monoid of $\Ci{\AManif}{\AManif}$ consisting of \FLPMap\ self-maps of $\AManif$, $\eAXvfol$ is a submonoid of \VERTINJECTIVE\ maps, and $\dAXfol$ is a subgroup of \FLPMap\ diffeomorphisms.
Then $\dAXfol \equiv \eAXvfol \cap \Diff(\AManif)$ due to Lemma~\ref{lm:F-tau-leaf-pres-maps}, and we have the following inclusions:
\begin{gather*}
    \dAXfol \subset \eAXvfol\subset  \eAXfol \subset \Ci{(\AManif,\ASingMan)}{(\AManif,\ASingMan)}.
\end{gather*}

Now Definitions~\ref{def:G-isol_nbg} and~\ref{def:folmap-scalable} can be adapted to the following:
\begin{subdefinition}\label{def:scalable_isolating_nbh}
Let $\ScalNbh \subset \ATotMan\cap\AManif$ be a neighborhood of $\ASingMan$ in $\AManif$.
Say that
\begin{itemize}[leftmargin=*]
\item
$\ScalNbh$ is \term{$\AFol$-isolating} for $\ASingMan$ if $\ScalNbh\cap(\overline{\omega}\setminus\omega) \subset \ASingMan$ for each leaf $\omega\in\AFol$;

\item
$\ScalNbh$ is \term{$\AFol$-scalable} (for $\ASingMan$ with respect to $\Apr\colon\ATotMan\to\ASingMan$) if $(\ScalNbh,\ScalNbh)$ is $\id_{\AFol}$-scalable in the sense of Definition~\ref{def:folmap-scalable}, i.e.\ satisfies the following condition:
\begin{itemize}
\item if $\px,\py\in\ScalNbh$ belong to the same leaf of $\AFol$ and $\tau\px,\tau\py\in\ScalNbh$ for some $\tau\geq1$, then these points $\tau\px$ and $\tau\py$ also  belong to the same leaf of $\AFol$.
\end{itemize}
\end{itemize}
\end{subdefinition}

Roughly speaking, \term{$\AFol$-scalability} means that the $\tau$-homothety $\ahomot{\tau}$ is ``$\AFol$-foliated on $\ScalNbh$''.
Furthermore, if the leaf $\omega\in\AFol$ is closed as a subset of $\AManif$, i.e.\ $\overline{\omega}=\omega$, then $\ScalNbh\cap(\overline{\omega}\setminus\omega) \subset \ASingMan$, so $\ScalNbh$ is automatically $\AFol$-isolating.

Also, see Lemma~\ref{lm:decr_scal_isol_nbhs}, if $\ScalNbh$ is $\AFol$-isolating (or $\AFol$-scalable) for $\ASingMan$, then any other neighborhood $\ScalNbh'\subset\ScalNbh$ of $\ASingMan$ is so.

As a consequence of Theorem~\ref{th:linearization_fol} we get the following statement which includes~\cite[Theorem~3.3.1]{Maksymenko:VBAuto:2024} stated without a proof:
\begin{subtheorem}
\label{th:foliated_linearization:formulation}
Let $a,b\geq0$, $\AFol$ be a partition of a connected manifold $\AManif$, $\ASingMan$ be a compact connected $\AFol$-saturated \GOOD\ submanifold, and $\Apr\colon\ATotMan\to\ASingMan$ be a tubular neighborhood of $\ASingMan$.
Suppose $\ASingMan$ admits a compact $\AFol$-isolating and $\AFol$-scalable neighborhood $\ScalNbh\subset\ATotMan\cap\AManif$.
Then the deformation $\Ghom$ from Theorem~\ref{th:linearization:diff_pres_submanif} of $\mABnt{a}{b}$ into $\mABntlin{a}{b}$ constructed for $\ANbhSingMan=\BNbhSingMan=\ScalNbh$ preserves the intersection $\eAXvfol \cap \mABnt{a}{b}$.

In particular, $\Ghom$ preserves the following family of subsets of $\eAXnt{a}{b}$:
\[
\alpha = \{\eAXvrt, \ \eAXfix, \ \eAXvid, \ \eAX\}
    \ \cap \ \eAXvfol \ \cap\ \eAXnt{a}{b},
\]
whence the inclusion of families $\alpha\cap\eAXL \,\fsubset\, \alpha$ is simultaneously a $\Wtr{\infty}$- and $\Str{\infty}$-homotopy equivalence of families.

Moreover, if $a=\axdim$ and $b=\ardim\geq1$, then one can choose $\Ghom$ so as to preserve diffeomorphisms, and so the inclusions of families
\begin{align}\label{equ:incl_DiffLP_Fol}
    \bigl\{\dAXnb, \ \dAXLfix, \ \dAXL, \ \dAXL\bigr\}
    \ \cap \ \dAXfol
    \ \ \fsubset \ \
    \bigl\{\dAXvid, \ \dAXfix, \ \dAXvrt, \ \dAX\bigr\}
    \ \cap \ \dAXfol
\end{align}
is simultaneously a $\Wtr{\infty}$- and $\Str{\infty}$-homotopy equivalence of families.
\end{subtheorem}

The following lemma (formulated without proof in~\cite[Lemma~2.3.2]{Maksymenko:VBAuto:2024}) presents an example of a partition $\AFol$ satisfying the assumptions of Theorem~\ref{th:foliated_linearization:formulation}.
\begin{sublemma}\label{lm:vbhomog}
Let $\Apr\colon\ATotMan\to\ASingMan$ be a vector bundle, and $\func\colon\ATotMan\to\bR$ be a continuous \term{fiberwise homogeneous} function of some (possibly fractional) degree $k>0$, that is $\func(\tau\pv) = \tau^{k}\func(\pv)$ for all $\tau>0$ and $\pv\in\ATotMan$.
Let also $\AFol$ be a partition of $\ATotMan$ whose elements are
\begin{itemize}[nosep]
\item connected components of $\ASingMan$;
\item connected components of $\func^{-1}(0)\setminus\ASingMan$;
\item connected components of the sets $\func^{-1}(c)$ for all $c\ne0$.
\end{itemize}

\begin{enumerate}[leftmargin=*, label={\rm(\arabic*)}]
\item\label{enum:vbhomog:cond:homog}
Then for every $\tau>0$, the $\tau$-homothety $\ahomot{\tau}$ is $\AFol$-foliated.
In particular, any neighborhood $\ScalNbh$ of $\ASingMan$ in $\ATotMan$ is $\AFol$-scalable.

\item\label{enum:vbhomog:cond:pathcomp}
Moreover, $\overline{\omega}\setminus\omega\subset\ASingMan$ for each leaf $\omega$ of $\AFol$.
In particular, any neighborhood $\ScalNbh$ of $\ASingMan$ in $\ATotMan$ is $\AFol$-isolating as well.
\end{enumerate}
\end{sublemma}
\begin{proof}
First note that since $\func$ is fiberwise homogeneous, for each $\tau>0$ and $c\in\bR$ we have that $\ahomot{\tau}\bigl(\func^{-1}(c)\bigr) = \func^{-1}(\tau^{k}c)$.
Indeed, $\pv\in\func^{-1}(c)$, i.e.\ $\func(\pv)=c$, then
\[ \func(\ahomot{\tau}(\pv)) = \func(\tau \pv) = \tau^{k}\func(\pv) = \tau^{k} c,\]
i.e.\ $\ahomot{\tau}(\pv)\in\func^{-1}(\tau^{k}c)$, and thus $\ahomot{\tau}\bigl(\func^{-1}(c)\bigr) \subset \func^{-1}(\tau^{k}c)$.
Applying the same arguments to $\ahomot{1/\tau} = (\ahomot{\tau})^{-1}$ we will get the inverse inclusion.

Also note that $\func = 0$ on $\ASingMan$, whence $\func^{-1}(c)\cap\ASingMan=\varnothing$ for $c\ne0$.

\ref{enum:vbhomog:cond:homog}
Now let $\tau>0$ and $\omega\in\AFol$.
We should check that $\ahomot{\tau}(\omega)\in\AFol$.
Consider two cases.

a) Suppose $\omega$ is a connected component of $\ASingMan$ (regarded as a zero-section of $\Apr$).
Then $\ahomot{\tau}$ is fixed on $\ASingMan$, and therefore $\ahomot{\tau}(\omega)=\omega\in\AFol$.

b) Suppose $\omega$ is a connected component of $\func^{-1}(c)\setminus\ASingMan$ for some $c\in\bR$.
Since $\ahomot{\tau}(\ASingMan)=\ASingMan$, we also have that $\ahomot{\tau}$ homeomorphically maps $\func^{-1}(c)\setminus\ASingMan$ onto $\func^{-1}(\tau{^k}c)\setminus\ASingMan$.
In particular, $\ahomot{\tau}(\omega)$ must be a connected component of $\func^{-1}(\tau^{k} c)\setminus\ASingMan$, i.e.\ a leaf of $\AFol$.

\ref{enum:vbhomog:cond:pathcomp}
Let $\omega\in\AFol$.
We should prove that $\overline{\omega}\setminus\omega\subset\ASingMan$.

a) Let $\Xman$ be either $\ASingMan$ or $\func^{-1}(c)$ for some $c\ne0$, and $\omega$ be a connected component of $\Xman$.
Then $\omega$ is closed in $\Xman$.
Since $\Xman$ is in turn closed in $\ATotMan$, we have that $\omega$ is closed in $\ATotMan$, i.e.\ $\overline{\omega}=\omega$, and therefore $\overline{\omega}\setminus\omega=\varnothing\subset\ASingMan$.

b) Denote $\Yman = \func^{-1}(0)$ and suppose that $\omega$ is a connected component of $\Yman\setminus\ASingMan$.
Then $\omega$ is closed in $\Yman\setminus\ASingMan$, whence $\omega = \overline{\omega} \cap (\Yman\setminus\ASingMan)$.
On the other hand, as $\Yman$ is closed in $\ATotMan$, we have that $\overline{\omega}\subset\Yman$.
Therefore
\[
    \overline{\omega}\setminus\omega
    =
    \overline{\omega}\setminus (\Yman\setminus\ASingMan)
    \subset
    \Yman\setminus (\Yman\setminus\ASingMan)= \ASingMan.
\]
Properties~\ref{enum:vbhomog:cond:homog} and~\ref{enum:vbhomog:cond:pathcomp} imply respectively that $\ATotMan$ is a $\AFol$-scalable and $\AFol$-isolating neighborhood for $\ASingMan$.
Therefore any neighborhood $\ScalNbh\subset\ATotMan$ is so.
\end{proof}

\newcommand\ImmNbF{\mathcal{I}^{nt}}
\newcommand\SubNbF{\mathcal{S}^{nt}}
\newcommand\LDiffNbF{\mathcal{D}^{loc}}
\newcommand\DiffNbF{\mathcal{D}}

\newcommand\mNFlinab[2]{\mNFnt{#1}{#2}_{lin}}

\newcommand\ImmNbFlin{\ImmNbF_{lin}}
\newcommand\SubNbFlin{\SubNbF_{lin}}
\newcommand\LDiffNbFlin{\LDiffNbF_{lin}}
\newcommand\DiffNbFlin{\mathcal{D}_{lin}}

\newcommand\STopMapNF[1]{\STopol{#1}{\mNF}} 
\newcommand\STopEMN[1]{\STopol{#1}{\mAB}}
\newcommand\WTopEMN[1]{\WTopol{#1}{\mAB}}

\section{Global linearization homotopy}
\label{sect:global_linearization}
In this section we will prove linearization Theorem~\ref{th:linearization_EMN_and_subspaces} (see \SECT{sect:proof_th:linearization_EMN_and_subspaces}) and its foliated variant Theorem~\ref{th:linearization_fol} (see \SECT{sect:proof:th:linearization_fol}), modulo one principal technical statement, Lemma~\ref{lm:estim_PQRS_var}, which will be established in \SECT{sect:proof:lm:estim_PQRS_var}.

Let $\AManif$ and $\BManif$ be two connected manifolds, $\ASingMan\subset\AManif$ and $\BSingMan\subset\BManif$ be connected \GOOD\ submanifolds, equipped with tubular neighborhoods $\Apr\colon\ATotMan\to\ASingMan$ and $\Bpr\colon\BTotMan\to\BSingMan$ of $\ASingMan$ and $\BSingMan$.
Let also $\ScalNbh\subset\ATotMan\cap\AManif$ be a smooth \STL\ neighborhood of $\ASingMan$.
Then we will consider the following three sets of maps of pairs:
\begin{align*}
\mAB &= \Ci{(\AManif,\ASingMan)}{(\BManif,\BSingMan)},  \\
\mNB &= \Ci{(\ScalNbh,\ASingMan)}{(\BManif,\BSingMan)}, \\
\mNF &= \Ci{(\ScalNbh,\ASingMan)}{(\BTotMan,\BSingMan)}.
\end{align*}
Just to memorize these letters: $\mAB$ is the space of certain maps from $\AManif$, $\mNB$ is the space of maps from a \term{neighborhood} $\ScalNbh$ of $\ASingMan$ but still with the values in all $\BManif$, while $\mNF$ is a space of ``\term{localized}'' maps from the neighborhood $\ScalNbh$ of $\ASingMan$ into the tubular neighborhood $\BTotMan$ of $\BSingMan$.
One can also regard $\mNF$ as a subset of $\mNB$, and we have a ``restriction to $\ScalNbh$ map'':
\[ \mAB \xrightarrow{~\dif \mapsto \restr{\dif}{\ScalNbh}~} \mNB  \ \supset \ \mNF. \]

Denote by $\mABlin$, $\mABvrt$, $\mABnt{a}{b}$, $\mABh{\hmap}$, $\mABv{\vmap}$ the subspaces of $\mAB$ consisting respectively of linearized, vertical, \MGENERIC{a}{b}\ maps, maps coinciding with a given $\Cinfty$ map $\hmap\colon\ASingMan\to\BSingMan$ on $\ASingMan$, and maps whose vertical derivative coincides with a given vector bundle morphism $\vmap\colon\ATotMan\to\BTotMan$ near $\ASingMan$.
Similar notations, e.g., $\mNBlin$, $\mNFnt{a}{b}$, etc.\ will be used for the corresponding subspaces of $\mNB$ and $\mNF$.

Let also $\LEF$ be the set of all vector bundle morphisms $\gdif\colon\ATotMan\to\BTotMan$.
Since every such $\gdif$ is determined by its values on any neighborhood of $\ASingMan$, the restriction map $\gdif\mapsto\restr{\gdif}{\ScalNbh}$ gives an inclusion $\LEF \subset \mNFlin$.

\subsection{Linearization with compact supports}
Our first aim is to construct a family of deformations of $\mNF$ into $\mNFlin$ which preserves the subspaces $\mNFvrt$, $\mNFnt{a}{b}$, see~\eqref{equ:Chom_formula}.
Then we will extend the arguments to construct a similar deformation of $\mABnt{a}{b}$ into $\mABntlin{a}{b}$.
The construction is based on a choice of the following two functions $\fnrm$ and $\mu$.
\begin{itemize}
\item Fix some $\Cinfty$ fiberwise positively definite quadratic form (scalar product) $Q\colon \ATotMan\oplus\ATotMan\to\bR$, and let
\[
    \fnrm\colon
        \ATotMan
            \xrightarrow{~\pv \ \mapsto \ (\pv,\pv)~}
        \ATotMan\oplus\ATotMan
            \xrightarrow{~Q~}
        \bR,
    \qquad
    \fnrm(\pv) = Q(\pv,\pv),
\]
be the corresponding \term{squared fiberwise norm} of $Q$.
Also for each $r\geq0$ define the following closed neighborhood of $\ASingMan$ in $\ScalNbh$:
\begin{align*}
    \genNbhx{r} &:= \{ \pw\in\ATotMan \mid \fnrm[\pw] \leq r^2\}.
\end{align*}
Thus $\genNbhx{r}$ is a disk bundle over $\ASingMan$ ``of radius $r$''.
We will especially consider the following three neighborhoods $\aNbh \subset \bNbh \subset \cNbh$.

\item
Also, let $\mu\colon\bR\to[0;1]$ be any $\Cinfty$ function satisfying
\begin{align*}
    \mu^{-1}(0)&=(-\infty;\aSqConst]=(-\infty;\aConst^2], \ & \
    \mu^{-1}(1)&=[\bSqConst;+\infty)=[\bConst^2;+\infty).
\end{align*}
\end{itemize}
Having them consider one more $\Cinfty$ function $\phifunc\colon(0;+\infty)\times[0;1]\times\ATotMan\to[0;1]$,
\begin{gather}\label{equ:defn:phi}
    \phifunc(\hprm,\pt,\pw) = \pt+(1-\pt)\mu\bigl(\tfrac{\fnrm(\pw)}{\hprm^2}\bigr),
\end{gather}
and for each $\dif\in\mNF$ define the following maps:
\begin{align*}
    &\glHom[\dif]\colon[0;1]\times\ScalNbh\to\BTotMan, &
    &\Chom[\dif]\colon(0;\infty)\times[0;1]\times\ScalNbh\to\BTotMan,
\end{align*}
\begin{align}
    \label{equ:homotopy_Ahom_}
    \glHom[\dif](\pt,\apnt) &=
    \begin{cases}
        \tfrac{1}{\pt}\dif(\pt\apnt), & \pt\in(0;1], \\
        \vtang{\dif}(\apnt),          & \pt=0;
    \end{cases}
    \\
    \label{equ:Chom_formula}
    \Chom[\dif](\hprm,\pt,\pw) =
    \glHom[\dif](\phifunc(\hprm,\pt,\pw), \pw) &=
        \begin{cases}
            \frac{\dif(\phifunc(\hprm,\pt,\pw)\pw)}{\phifunc(\hprm,\pt,\pw)}, &
            \phifunc(\hprm,\pt,\pw)>0, \\
            \vtang{\dif}(\pw), &
            \phifunc(\hprm,\pt,\pw)=0.
        \end{cases}
\end{align}

For $(\hprm,\pt)\in(0;+\infty)\times[0;1]$ it will be convenient to consider the following  maps $\glHom[\dif]_{\pt}, \Chom[\dif]_{\hprm,\pt}\colon\ScalNbh\to\BTotMan$,
\begin{align*}
    \glHom[\dif]_{\pt}(\pw)      &= \glHom[\dif](\pt,\pw), &
    \Chom[\dif]_{\hprm,\pt}(\pw) &= \Chom[\dif](\hprm,\pt,\pw).
\end{align*}

Note that if $\ScalNbh$ is compact, then there exist $\hprm_0 < \hprm_1 \in (0;+\infty)$ such that
\begin{align*}
&\cNbh \subset \ScalNbh\ \text{for}\ \hprm < \hprm_0 &
&\text{and}&
&\ScalNbh \subset \aNbh\ \text{for}\ \hprm_1 < \hprm.
\end{align*}

Our main statement, see \SECT{sect:defomr_mNF_mNFlin}, is that for a proper choice of a function $\hpFunc\colon\mNF\to(0;\infty)$, the correspondence $(\dif,\pt)\mapsto\Chom[\dif]_{\hpFunc(\dif),\pt}$, is a desired deformation of $\mNF$ into $\mNFlin$ which preserves the subspaces $\mNFvrt$, $\mNFh{\hmap}$, $\mNFv{\vmap}$, $\mNFnt{a}{b}$.
For its proof we will present local formulas for $\glHom[\dif]$ and estimates for Jacobi matrices of $\Chom[\dif]_{\hprm,\pt}$, and also formulate one technical Lemma~\ref{lm:estim_PQRS_var} describing the asymptotic behavior of Jacobi matrices of $\Chom[\dif]_{\hprm,\pt}$ for small $\hprm$.
The situation is local and therefore we will consider the situation when the vector bundles $\Apr$ and $\Bpr$ are trivial.

\subsection{Local estimates for $\glHom[\dif]$ and $\Chom[\dif]$}
\label{sect:local_setim_glHom_Chom}
First we need the following variant of the Hadamard lemma:
\begin{sublemma}[e.g.\ {\cite[Lemma~II.6.10]{GolubitskiGuillemin:StableMats:ENG:1973}}]
\label{lm:Hadamard_lemma_2}
Let
$\px=(\pxi{1},\ldots,\pxi{\axdim})$ be coordinates in $\bR^{\axdim}$,
$\py=(\pyi{1},\ldots,\pyi{\ardim})$ be coordinates in $\ALocFiber$, and
\[
    \bbfunc=(\bbfunc_1,\ldots,\bbfunc_{\brdim})\colon
    \ALocBaseFiber\to\BLocFiber
\]
be a $\Cr{2}$ map.
Then for all $\pv=(\pvi{1},\ldots,\pvi{\ardim})\in\ALocFiber$
\begin{align}\label{equ:g_decomp_1}
\hspace{-0.2cm}
\bbfunc(\px, \py+\pv)
    &= \bbfunc(\px,\py) \!+\!
    \sum_{i=1}^{\ardim}
        \pvi{i}
        \smallint_{0}^{1}
            \der{\bbfunc}{\pyi{i}}(\px, \py+s\pv) \,ds
\\
    \label{equ:g_decomp_2}
    &= \bbfunc(\px,\py) \!+\!\!
        \underbrace{\smashoperator{\sum\limits_{i=1}^{\ardim}} \pvi{i} \der{\bbfunc}{\pyi{i}}(\px,\py)}_{\JM[\px]{\bbfunc}\trans{\pv}}
        +
        \smashoperator{\sum\limits_{i,j=1}^{\ardim}}
            \pvi{i} \pvi{j}
            \smallint_{0}^{1}
            \smallint_{0}^{1}
                s \dertwo{\bbfunc}{\pyi{i}}{\pyi{j}}(\px,\py+ st\pv)\,ds\, dt.
\end{align}
\end{sublemma}
\begin{proof}
Let $\gamma\colon[0;1]\to\ALocFiber$ be a path given by $\gamma(s) = \py+s\pv$.
Then by the Newton--Leibniz formula applied to the composition $\bbfunc\circ\gamma$, we get~\eqref{equ:g_decomp_1}.
Applying the same method to partial derivatives $\der{\bbfunc}{\pyi{i}}$ we will obtain~\eqref{equ:g_decomp_2}.
Finally, one recognizes the first sum in~\eqref{equ:g_decomp_2} as the multiplication $\JM[\px]{\bbfunc}\trans{\pv}$ of the Jacobi matrix of $\bbfunc$ at $\px$ by the vector column $\pv$.
\end{proof}

Now let $\Apr\colon\ATotMan=\ASingMan\times\ALocFiber\to\ASingMan$ and $\Bpr\colon\BTotMan=\BSingMan\times\BLocFiber\to\BSingMan$ be trivial vector bundles over some open subsets $\ASingMan\subset\ALocBase$ and $\BSingMan\subset\BLocBase$, $\ScalNbh$ be a smooth \STL\ neighborhood of $\ASingMan\times 0$ in $\ASingMan\times\ALocFiber$, and $\dif=(\aafunc,\bbfunc)\colon\ScalNbh\to\BSingMan\times\BLocFiber$ be a $\Cinfty$ map such that $\dif(\ASingMan\times 0) \subset \BSingMan\times 0$.
Then under a proper choice of functions $\fnrm$ and $\mu$ one defines the maps $\glHom[\dif]$ and $\Chom[\dif]$ by~\eqref{equ:homotopy_Ahom_} and~\eqref{equ:Chom_formula} respectively.

\begin{sublemma}\label{lm:local_formulas_for_glHom}
The formula~\eqref{equ:homotopy_Ahom_} for $\glHom[\dif]$ reduces to the following one:
\begin{equation}\label{equ:homotopy_Ahom_local_formulas}
    \glHom[\dif]\bigl(\pt, (\px,\pv)\bigr) =
        \Bigl(\aafunc(\px,\pt\pv),
        \sum_{i=1}^{\ardim}
            \pvi{i}
            \smallint_{0}^{1}
                \der{\bbfunc}{\pvi{i}}(\px, s\pt\pv) \,ds
        \Bigr).
\end{equation}
In particular, $\glHom[\dif]$ and $\Chom[\dif]$ are $\Cinfty$.
Moreover, the Jacobi matrix of $\glHom_{\pt}$ at $(\px,0)$ consists of the following blocks:
\begin{equation}\label{equ:JM_Ahom_t}
    \JM[\px,0]{\glHom[\dif]_{\pt}} =
        \amatr{\der{\aafunc}{\px}(\px,0)}
              {\pt\der{\aafunc}{\pv}(\px,0)}
              {0}
              {\der{\bbfunc}{\pv}(\px,0)},
\end{equation}
while the vertical derivative $\vtang{\dif}$ of $\dif$ is given by
\begin{equation}\label{equ:tfib:local_formulas}
    \vtang{\dif}(\px,\pv)
    = \glHom[\dif](0, (\px,\pv)) =
    \bigl(\aafunc(\px,0), \bbfunc'_{\pv}(\px,0) \trans{\pv} \bigr).
\end{equation}
\end{sublemma}
\begin{proof}
The assumption $\dif(\ASingMan\times 0) \subset \BSingMan\times 0$ means that $\bbfunc(\px,0)=0\in\BLocFiber$ for all $\px\in\ASingMan$.
Then by~\eqref{equ:g_decomp_1}
\[
    \bbfunc(\px,\pv) =
     \sum_{i=1}^{\ardim}
        \pvi{i}
        \smallint_{0}^{1}
            \der{\bbfunc}{\pvi{i}}(\px, s\pv) \,ds, \quad (\px,\pv)\in\ScalNbh.
\]
As the multiplication in $\ATotMan=\ASingMan\times\ALocFiber$ by scalars is given by $\pt\cdot(\px,\pv) = (\px,\pt\pv)$, we see that for $\pt\ne0$
\begin{align*}
    \glHom[\dif](\pt, (\px,\pv)) &=
    \tfrac{1}{\pt}\dif(\pt\cdot(\px,\pv)) =
    \tfrac{1}{\pt}\dif(\px,\pt\pv)  =
    \bigl(\aafunc(\px,\pt\pv),\tfrac{1}{\pt}\bbfunc(\px,\pt\pv)\bigr)
    \\ &=
    \Bigl(\aafunc(\px,\pt\pv),
        \sum_{i=1}^{\ardim}
            \pvi{i}
            \smallint_{0}^{1}
                \der{\bbfunc}{\pvi{i}}(\px, st\pv) \,ds
    \Bigr).
\end{align*}
This proves~\eqref{equ:homotopy_Ahom_local_formulas}.
Formula~\eqref{equ:JM_Ahom_t} is straightforward.
\end{proof}

\subsection{Deformations of $\mNF$ into $\mNFlin$}
\label{sect:defomr_mNF_mNFlin}

As at the beginning of this section, let $\AManif$ and $\BManif$ be two connected manifolds, $\ASingMan\subset\AManif$ and $\BSingMan\subset\BManif$ be connected \GOOD\ submanifolds, equipped with tubular neighborhoods $\Apr\colon\ATotMan\to\ASingMan$ and $\Bpr\colon\BTotMan\to\BSingMan$ of $\ASingMan$ and $\BSingMan$, $\ScalNbh\subset\ATotMan\cap\AManif$ be a smooth \STL\ neighborhood of $\ASingMan$, and $\mNF = \Ci{(\ScalNbh,\ASingMan)}{(\BTotMan,\BSingMan)}$.
Then for $\dif\in\mNF$ we have defined maps $\glHom[\dif]$ and $\Chom[\dif]$ by~\eqref{equ:homotopy_Ahom_} and~\eqref{equ:Chom_formula}.

\begin{sublemma}\label{lm:glHom_Chom}
Let $(\dif,\hprm,\pt)\in\mNF\times(0;+\infty)\times[0;1]$.
Then
\begin{enumerate}[label={\rm(\alph*)}]
\item\label{enum:glHom_Chom:smooth}
the maps $\glHom[\dif]$, and therefore $\Chom[\dif]$, are $\Cinfty$;

\item\label{enum:glHom_Chom:global_map}
for each $r\geq0$ the map $\Chom\colon\mNF\times(0;+\infty)\times[0;1]\to\mNF$, $\Chom(\dif,\hprm,\pt)=\Chom[\dif](\hprm,\pt)$, is continuous as a map
\[
    \Chom\colon\WTopol{r+1}{\mNF}\times(0;+\infty)\times[0;1]\to\WTopol{r}{\mNF},
\]
i.e.\ when we endow $\mNF$ in the source with the topology $\Wtr{r+1}$ and in the destination with the topology $\Wtr{r}$;
in particular, it is continuous when $\mNF$ is endowed with the topology $\Wtr{\infty}$;

\item\label{enum:glHom_Chom:h}
$\Chom[\dif]_{\hprm,1}=\glHom[\dif]_{1} = \dif$ \, and \, $\Chom[\dif]_{\hprm,\pt}=\dif$ on $\ScalNbh\setminus\bNbh$;

\item\label{enum:glHom_Chom:h_gl}
$\Chom[\dif]_{\hprm,\pt}=\glHom[\dif]_{\pt}$ on $\aNbh\cap\ScalNbh$;

\item\label{enum:glHom_Chom:H_0}
$\Chom[\dif]_{\hprm,0}=\glHom[\dif]_{0} = \vtang{\dif}$ on $\aNbh\cap\ScalNbh$, so $\Chom[\dif]_{\hprm,0}\in\mNFlin$;

\item\label{enum:glHom_Chom:Vh}
$\vtang{\Chom[\dif]_{\hprm,\pt}} = \vtang{\glHom[\dif]_{\pt}} = \vtang{\dif}$;

\item\label{enum:glHom_Chom:hlin}
if $\dif=\vtang{\dif}$ on some \STL\ neighborhood $\Uman \subset\ScalNbh$ of $\ASingMan$, in particular, $\dif\in\mNFlin$, then $\Chom[\dif]_{\hprm,\pt}=\dif$ on $\Uman$ as well;
in particular $\Chom[\dif]_{\hprm,\pt}\in\mNFlin$;

\item\label{enum:glHom_Chom:LEF}
if $\dif\in\LEF \subset \mNFlin$, then $\Chom[\dif]_{\hprm,\pt}=\dif$;

\item\label{enum:glHom_Chom:vert}
if $\dif$ is vertical, then $\Chom[\dif]_{\hprm,\pt}$ is vertical as well;

\item\label{enum:glHom_Chom:neat}
if $\dif$ is \MGENERIC{a}{b}\ on some \STL\ neighborhood $\Uman\subset\ScalNbh$ of $\ASingMan$ for some $a,b\geq0$, then $\glHom[\dif]_{\pt}$ is also \MGENERIC{a}{b}\ on $\Uman$;
in particular, if $\dif\in\mNFnt{a}{b}$, then $\glHom[\dif]_{\pt}\in\mNFnt{a}{b}$ as well.
\end{enumerate}
\end{sublemma}
\begin{proof}
Statement \ref{enum:glHom_Chom:smooth} is a consequence of Lemma~\ref{lm:local_formulas_for_glHom}.

\ref{enum:glHom_Chom:global_map}
Note that formulas~\eqref{equ:homotopy_Ahom_local_formulas} and~\eqref{equ:Chom_formula}  express $\Chom[\dif]$ via coordinate functions of $\dif$ and their partial derivatives of order $1$ as a post- and pre-composition with some smooth maps.
This implies that $\Chom\colon\WTopol{r+1}{\mNF}\times(0;+\infty)\times[0;1]\to\WTopol{r}{\mNF}$ is continuous for each $r\geq0$.

\ref{enum:glHom_Chom:h}
We should prove that $\Chom[\dif]_{\hprm,1}=\dif$ and $\Chom[\dif]_{\hprm,\pt}=\dif$ on $\ScalNbh\setminus\bNbh$.
Note that if
\begin{itemize}[nosep]
\item either $\pt=1$;
\item or $\pw\in\ScalNbh\setminus\bNbh$, i.e.\ $\fnrm(\pw)>\bSqConst\hprm^2$, and so $\mu\bigl(\tfrac{\fnrm[\pw]}{\hprm^2}\bigr)=1$,
\end{itemize}
then $\phifunc(\hprm,\pt,\pw)=1$.
Hence $\Chom[\dif]_{\hprm,\pt}(\pw) = \glHom[\dif](1,\pw) = \dif(\pw)$.

\ref{enum:glHom_Chom:h_gl}
Let $\pw\in\aNbh$, i.e.\ $\fnrm[\pw] \leq \aSqConst\hprm^2$.
Then $\mu\bigl(\tfrac{\fnrm[\pw]}{\hprm^2}\bigr)=0$, and $\phifunc(\hprm,\pt,\pw) = \pt$.
Hence $\Chom[\dif]_{\hprm,\pt}(\pw)
        \stackrel{\eqref{equ:Chom_formula}}{=}
    \glHom[\dif](\phifunc(\hprm,\pt,\pw), \pw)
        =
    \glHom[\dif](\pt, \pw)$.

\ref{enum:glHom_Chom:H_0}
In particular, for $\pt=0$ and $\pw\in\aNbh$, we get $\Chom[\dif]_{\hprm,0} \stackrel{\ref{enum:glHom_Chom:h_gl}}{=} \glHom[\dif]_{0} = \vtang{\dif}$ on $\aNbh$, so $\Chom[\dif]_{\hprm,0}\in\mNFlin$;

\ref{enum:glHom_Chom:Vh}
Let us show that $\vtang{\Chom[\dif]_{\hprm,\pt}} = \vtang{\dif}$.
Due to~\ref{enum:glHom_Chom:H_0}, $\Chom[\dif]_{\hprm,0} = \vtang{\dif}$ on $\aNbh$.
Therefore $\vtang{\Chom[\dif]_{\hprm,0}} = \vtang{(\vtang{\dif})}\stackrel{\eqref{equ:VVh_Vh}}{=}\vtang{\dif}$.

Now let $\pt\in(0;1]$ and $\pw\in\ScalNbh$.
Then $\tau\pw\in\aNbh$ for all sufficiently small $\tau>0$, so $\mu(\fnrm(\tau\pw)/\hprm^2) = 0$, and therefore $\phifunc(\hprm,\pt,\tau\pw) = \pt$.
Hence
\begin{align*}
    \vtang{(\Chom[\dif]_{\hprm,\pt})}(\pw) &=
    \lim_{\tau\to0} \tfrac{1}{\tau}\,\Chom[\dif]_{\hprm,\pt}(\tau\pw)
    \stackrel{\eqref{equ:Chom_formula}}{=}
    \lim_{\tau\to0} \tfrac{1}{\tau}\tfrac{\dif(\phifunc(\hprm,\pt,\tau\pw)\tau\pw)}{\phifunc(\hprm,\pt,\tau\pw)}\, =
    \lim_{\tau\to0} \tfrac{1}{\tau}\tfrac{1}{\pt}\,\dif(\pt\tau\pw) \\ &=
    \lim_{\ps = \pt\tau\to0} \tfrac{1}{\ps}\,\dif(\ps\pw) = \vtang{\dif}(\pw).
\end{align*}

\ref{enum:glHom_Chom:hlin}
Suppose $\dif=\vtang{\dif}$ on some \STL\ neighborhood $\Uman \subset\ScalNbh$ of $\ASingMan$.
Let also $\pw\in\Uman$.
If $\phifunc(\hprm,\pt,\pw)=0$, then $\Chom[\dif]_{\hprm,\pt}(\pw) = \vtang{\dif}(\pw)$.
On the other hand suppose $\tau:=\phifunc(\hprm,\pt,\pw)>0$, then
\[
    \Chom[\dif]_{\hprm,\pt}(\pw) =
    \frac{1}{\tau} \dif(\tau\pw) =
    \frac{1}{\tau} \vtang{\dif}(\tau\pw) =
    \frac{\tau}{\tau} \vtang{\dif}(\pw) =
    \vtang{\dif}(\pw).
\]

\ref{enum:glHom_Chom:LEF}
This is a particular case of~\ref{enum:glHom_Chom:hlin} for $\Uman=\ScalNbh$.

\ref{enum:glHom_Chom:vert}
Suppose $\dif\in\mNFvrt$, i.e.\ it sends vertical vectors over $\ASingMan$ to vertical vectors over $\BSingMan$.
We need to show that $\Chom[\dif]_{\hprm,\pt}$ has the same property.
To check that it suffices to consider the restriction of $\Chom[\dif]_{\hprm,\pt}$ to an arbitrary small neighborhood of $\pw$.
But, by~\ref{enum:glHom_Chom:h_gl}, $\Chom[\dif]_{\hprm,\pt}$ coincides with $\glHom[\dif]_{\pt}$ on the neighborhood $\aNbh$, whence it is enough to check that $\glHom[\dif]_{\pt} \in \mNFvrt$.

Let $\ahomot{\tau}\colon\ATotMan\to\ATotMan$, $\ahomot{\tau}(\px)=\tau\px$, and $\bhomot{\tau}\colon\BTotMan\to\BTotMan$, $\bhomot{\tau}(\py)=\tau\py$, be multiplications by scalars in the vector bundles $\ATotMan$ and $\BTotMan$.
Then $\ahomot{\tau}$ and $\bhomot{\tau}$ leave each fiber of those vector bundles, and also preserve $\ASingMan$ and $\BSingMan$ respectively.
This implies that they send vertical vectors to vertical vectors.

Hence, for $\pt\in(0;1]$, the map $\glHom[\dif]_{\pt} = \bhomot{1/\pt} \circ \dif \circ \ahomot{\pt}$ should also send vertical vectors to vertical vectors.
Moreover, for $\pt=0$, the vector bundle morphism $\glHom[\dif]_{0} = \vtang{\dif}$ also sends vertical vectors to vertical vectors.
Thus $\Chom[\dif]_{\hprm,\pt}\in\mNFvrt$.

\ref{enum:glHom_Chom:neat}
Suppose $\dif$ is \MGENERIC{a}{b}\ on some \STL\ neighborhood $\Uman\subset\ScalNbh$ of $\ASingMan$ for some $a,b\geq0$.
This means that $\hrank[\pw]{\dif}\geq a$ and $\vrank[\pw]{\dif}\geq b$ for all $\pw\in\ASingMan$, and also $\trank[\pw]{\dif}\geq a+b$ for all $\pw\in\Uman$.

Since $\vtang{\glHom[\dif]_{\pt}}=\vtang{\dif}$ for all $\pt\in[0;1]$, we also have by Lemma~\ref{lm:hvranks} that
\begin{align*}
    &\hrank[\pw]{\glHom[\dif]_{\pt}}=\hrank[\pw]{\dif}\geq a, &
    &\vrank[\pw]{\glHom[\dif]_{\pt}}=\vrank[\pw]{\dif}\geq b
\end{align*}
for all $\pw\in\ASingMan$.
Therefore, it remains to check that $\trank[\pw]{\glHom[\dif]_{\pt}} \geq a+b$ for all $\pw\in\Uman$.
Consider two cases.

1) Suppose $\pt\in(0;1]$, so $\glHom[\dif]_{\pt} = \bhomot{1/\pt} \circ \dif \circ \ahomot{\pt}$.
As $\Uman$ is \STL, $\ahomot{\pt}(\Uman)\subset\Uman$.
Moreover, since $\bhomot{1/\pt}$ and $\ahomot{\pt}$ are diffeomorphisms, it follows that
\[\trank[\pw]{\glHom[\dif]_{\pt}} = \trank[\pw]{\dif} \geq a+b\]
for all $\pw\in\Uman$.

2) Now let $\pt=0$, so $\glHom[\dif]_{0}(\px,\pv) = \vtang{\dif}(\px,\pv) = \bigl(\aafunc(\px,0), \bbfunc'_{\pv}(\px,0) \trans{\pv} \bigr)$,  by~\eqref{equ:tfib:local_formulas}.
Then for $\pw = (\px,\pv)\in\Uman$
\[
    \JM[\px,\pv]{\glHom[\dif]_{0}} =
        \amatr{\der{\aafunc}{\px}(\px,0)}
              {0}
              {\mR(\px,\pv)}
              {\der{\bbfunc}{\pv}(\px,0)},
\]
for some matrix $\mR$.
But
\begin{align*}
    &\trank{\der{\aafunc}{\px}(\px,0)} = \hrank[(\px,0)]{\dif} \geq a, &
    &\trank{\der{\bbfunc}{\pv}(\px,0)} = \vrank[(\px,0)]{\dif} \geq b,
\end{align*}
whence
\[
    \trank{\JM[\px,\pv]{\glHom[\dif]_{0}}} \geq
    \trank{\der{\aafunc}{\px}(\px,0)} +
    \trank{\der{\bbfunc}{\pv}(\px,0)} \geq a+b.
\]
Hence $\glHom[\dif]_{0}$ is \MGENERIC{a}{b}\ as well.
This completes the proof of Lemma~\ref{lm:glHom_Chom}.
\end{proof}

The following statement gives an isotopy between any two tubular neighborhoods of $\ASingMan$ in $\AManif$, e.g.~\cite[Chapter~4, Theorem~5.3]{Hirsch:DiffTop}.
\begin{subcorollary}\label{cor:glHom_deformation}
Suppose $\ASingMan$ and $\ScalNbh$ are compact.
Endow $\mNF$ with the topology $\Wtr{\infty}$.
Then the map $\glHom\colon\mNF\times[0;1]\to\mNF$, $\glHom(\dif,\pt) = \glHom[\dif]_{\pt}$, is a strong deformation retraction of $\mNF$ onto $\LEF$ which leaves invariant all the spaces
\begin{align}\label{equ:glHom_pres:mNF-spaces}
    &\LEF \ \subset \ \mNFlin \ \subset \ \mNFvrt, &
    &\mNFh{\hmap}, &
    &\mNFv{\vmap}, &
    &\mNFnt{a}{b}.
\end{align}
In particular, if $\calX$ is any of those spaces, then the inclusions
\[
    \calX\cap\LEF
       \ \subset \
    \calX \cap \mNFlin
        \ \subset \
    \calX \cap \mNFvrt
        \ \subset \
    \calX
\]
are homotopy equivalences.
\end{subcorollary}
\begin{proof}
Notice that since $\ScalNbh$ is compact, $\ScalNbh \subset \aNbhx{\hprm}$ for sufficiently large $\hprm$.
Then due to Lemma~\ref{lm:glHom_Chom}\ref{enum:glHom_Chom:h_gl}, for such $\hprm$ we have that $\Chom[\dif]_{\hprm,\pt} = \glHom[\dif]_{\pt}$ on $\ScalNbh \cap \aNbhx{\hprm} = \ScalNbh$.
In particular, $\glHom = \restr{\Chom}{\mNF\times\{\hprm\}\times[0;1]}$, and therefore $\glHom$ is continuous if we endow $\mNF$ with the topology $\Wtr{\infty}$.

By formula~\eqref{equ:homotopy_Ahom_}, $\glHom[\dif]_{1} = \dif$ and $\glHom[\dif]_{0} = \vtang{\dif} \in \LEF$ for each $\dif\in\mNF$.
In other words, $\glHom_1=\id_{\mNF}$ and $\glHom_0(\mNF) \subset \LEF$.
Moreover, if $\dif\in\LEF$, then for $\pt\in(0;1]$ we have that
$\glHom[\dif]_{\pt}(\pw) = \tfrac{1}{\pt}\dif(\pt\pw) = \tfrac{\pt}{\pt}\dif(\pw) = \dif(\pw)$.
Hence each $\glHom_{\pt}$ is also fixed on $\LEF$.
Therefore $\glHom$ is a strong deformation retraction of $\mNF$ onto $\LEF$.

Also, by statements~\ref{enum:glHom_Chom:hlin}, \ref{enum:glHom_Chom:vert}, \ref{enum:glHom_Chom:neat} of that lemma, $\glHom$ preserves $\mNFlin$, $\mNFvrt$, and $\mNFnt{a}{b}$ respectively.
Moreover by~\ref{enum:glHom_Chom:Vh}, $\vtang{\glHom[\dif]_{\pt}} = \vtang{\dif}$, for all $\pt\in[0;1]$, which implies that the spaces $\mNFh{\hmap}$ and $\mNFv{\vmap}$ are also preserved under $\glHom$.
\end{proof}

Note that the homotopy $\glHom[\dif]$ linearizes (and thus changes) $\dif$ \term{on all of $\ScalNbh$}.
In turn, $\Chom[\dif]$ linearizes $\dif$ only near $\ASingMan$ and keeps it unchanged out of some larger neighborhood of $\ASingMan$.
The following statement gives many such deformations of $\mNF$ into $\mNFlin$; however, they do not necessarily preserve $\mNFnt{a}{b}$.
\begin{subcorollary}\label{cor:def_L_Llin_not_pres_Lntab}
Let $\hpFunc\colon\mNF\to(0;\infty)$ be a strictly positive continuous function from the topology $\Wtr{\infty}$ on $\mNF$.
Define the map
\[
    \Ghom\colon\mNF\times[0;1]\to\mNF,
    \qquad
    \Ghom(\dif,\pt)(\pw) = \Chom[\dif](\hpFunc(\dif),\pt,\pw).
\]
Then $\Ghom$ is a deformation of $\mNF$ into $\mNFlin$ which preserves the subspaces
\begin{align}\label{equ:Chom_pres:mNF-spaces}
    &\mNFlin \ \subset \ \mNFvrt,  &
    &\mNFh{\hmap}, &
    &\mNFv{\vmap},
\end{align}
but not necessarily $\mNFnt{a}{b}$ and $\LEF$.
\end{subcorollary}
\begin{proof}
First note that $\Ghom$ is continuous as a map
\[
    \mNF \times[0;1] \xrightarrow{~(\dif,\pt) \ \mapsto \ (\dif,\hpFunc(\dif),\pt)~}
    \mNF \times (0;+\infty) \times [0;1] \xrightarrow{~\Chom~}
    \mNF.
\]
Further, $\Ghom_{1} = \id_{\mNF}$ by statement~\ref{enum:glHom_Chom:h} of Lemma~\ref{lm:glHom_Chom}, $\Ghom_0(\mNF) \subset \mNFlin$ by~\ref{enum:glHom_Chom:H_0}, $\Ghom$ preserves $\mNFlin$ by~\ref{enum:glHom_Chom:hlin}.
Thus $\Ghom$ is a deformation of $\mNF$ into $\mNFlin$.
Moreover, by~\ref{enum:glHom_Chom:vert} it also preserves $\mNFvrt$.
Finally by~\ref{enum:glHom_Chom:Vh} it also does not change the vertical derivative which implies that $\mNFh{\hmap}$ and $\mNFv{\vmap}$ are also preserved.
\end{proof}

\subsection{Deformations of $\mNF$ into $\mNFlin$ preserving $\mNFnt{a}{b}$}
Let $\Mat{\bxdim}{\axdim}$ be the space of all real matrices of size $\bxdim\times\axdim$, ($\bxdim$ rows and $\axdim$ columns).
Then by the \term{norm} of $A = (a_{ij}) \in\Mat{\bxdim}{\axdim}$ we will mean the sum of absolute values of its coefficients:
\[
    \nrm{A} := \sum_{i,j} |a_{ij}|.
\]

In this Section we will formulate one local statement, see Lemma~\ref{lm:estim_PQRS_var}, implying that  if in Corollary~\ref{cor:def_L_Llin_not_pres_Lntab} the function $\hpFunc$ is sufficiently small, then the corresponding deformation $\Ghom$ also preserves $\mNFnt{a}{b}$ for all $a,b\geq0$.

\newcommand\hafunc[3][\empty]{\alpha^{#1}_{#2,#3}}
\newcommand\hbfunc[3][\empty]{\beta^{#1}_{#2,#3}}

For that reason, we return to the setting of \SECT{sect:local_setim_glHom_Chom} and consider a local situation, when $\Apr\colon\ATotMan=\ASingMan\times\ALocFiber\to\ASingMan$ and $\Bpr\colon\BTotMan=\BSingMan\times\BLocFiber\to\BSingMan$ are trivial vector bundles over some open subsets $\ASingMan\subset\ALocBase$ and $\BSingMan\subset\BLocBase$, and $\ScalNbh$ is a smooth \STL\ neighborhood of $\ASingMan\times 0$ in $\ASingMan\times\ALocFiber$, and $\dif=(\aafunc,\bbfunc)\colon\ScalNbh\to\BSingMan\times\BLocFiber$ is a $\Cinfty$ map such that $\dif(\ASingMan\times 0) \subset \BSingMan\times 0$.
Then under a choice of functions $\fnrm$ and $\mu$ one defines the maps $\glHom[\dif]$ and $\Chom[\dif]$ by~\eqref{equ:homotopy_Ahom_} and~\eqref{equ:Chom_formula} respectively.

For $(\hprm,\pt)\in(0;+\infty)\times[0;1]$ let $\Chom[\dif]_{\hprm,\pt}=(\hafunc[\dif]{\hprm}{\pt},\hbfunc[\dif]{\hprm}{\pt})\colon\ScalNbh \to \BSingMan\times\BLocFiber$ be the coordinate functions of $\Chom[\dif]_{\hprm,\pt}$.
Consider the Jacobi matrices of $\dif$ and $\Chom[\dif]_{\hprm,\pt}$:
\begin{align*}
    \JM{\dif}(\px,\pv) &=
    \amatr{\der{\aafunc}{\px}}{\der{\aafunc}{\pv}}
          {\der{\bbfunc}{\px}}{\der{\bbfunc}{\pv}} =
    \amatr{\mP[\dif]}{\mQ[\dif]}
          {\mR[\dif]}{\mS[\dif]}, \\[1mm]
    \JM{\Chom[\dif]_{\hprm,\pt}}(\px,\pv) &=
    \amatr{\der{(\hafunc[\dif]{\hprm}{\pt})}{\px}}{\der{(\hafunc[\dif]{\hprm}{\pt})}{\pv}}
          {\der{(\hbfunc[\dif]{\hprm}{\pt})}{\px}}{\der{(\hbfunc[\dif]{\hprm}{\pt})}{\pv}}
          =
    \amatr{\mP[{\Chombt[\dif]}]}
          {\mQ[{\Chombt[\dif]}]}
          {\mR[{\Chombt[\dif]}]}
          {\mS[{\Chombt[\dif]}]},
\end{align*}
where it is convenient to denote the corresponding blocks of these matrices with letters $\mP$, $\mQ$, $\mR$, $\mS$.
The following lemma is the principal technical statement estimating the asymptotics of these blocks as $\hprm$ decreases.
We will prove it in next section.

\begin{sublemma}\label{lm:estim_PQRS_var}
Let $\dif\in\mNF = \Ci{(\ScalNbh,\ASingMan)}{(\BTotMan,\BSingMan)}$ and $\BComp\subset\ScalNbh$ be a compact \STL\ subset.
Then for each $\hgeps>0$ there exist two constants $\boundQ>0$ and $\constC>0$ such that
\begin{itemize}
\item
for every $\Cinfty$ map $\gdif\colon\ScalNbh\to\BSingMan\times\BLocFiber$ satisfying $\xnrm{\gdif-\dif}{\TOP}{\BComp}<\hgeps$,
\item
and all $(\hprm,\pt)\in(0;+\infty)\times[0;1]$ and $(\px,\pv)\in\BComp\cap\cNbhx{\hprm}$,
\end{itemize}
the following conditions hold:
\begin{equation}\label{equ:estim_H_h}
\begin{aligned}
    &\nrm{\mP[{\Chombt[\gdif]}](\px,\pv) -
        \mP[\dif](\px,0)} < \constC \hprm + \hgeps,
    \\
    &\nrm{\mQ[{\Chombt[\gdif]}](\px,\pv)} < \boundQ,\\
    &\nrm{\mR[{\Chombt[\gdif]}](\px,\pv)} < \constC \hprm + \hgeps, \\
    &\nrm{\mS[{\Chombt[\gdif]}](\px,\pv) -
        \mS[\dif](\px,0)} < \constC \hprm + \hgeps.
\end{aligned}
\end{equation}
\end{sublemma}

As a consequence we get the following statement which estimates the ranks of Jacobi matrices of $\Chom[\dif]_{\hprm,\pt}$.

\newcommand\LMbl[1]{\amatr{\mP[#1]}{\mQ[#1]}{\mR[#1]}{\mS[#1]}}

\begin{sublemma}\label{lm:make_Chom_neat}
Let $\dif\in\mNF$.
Suppose there exists a compact \STL\ subset $\BComp\subset\ScalNbh$ and $a,b\geq0$ such that $\dif$ is \MGENERIC{a}{b}\ on $\BComp$.
Then for each $\hgeps>0$ there exists $\bhprm>0$ such that for all $\gdif\in\mNF$ satisfying $\xnrm{\gdif-\dif}{\TOP}{\BComp}<\hgeps$, the map $\Chom[\gdif]_{\hprm,\pt}$ is also \MGENERIC{a}{b}\ on $\BComp$ for all $(\hprm,\pt)\in(0;\bhprm]\times[0;1]$.
\end{sublemma}
\begin{proof}
\newcommand\aaeps{\eps_0}
\newcommand\bbeps{\eps_1}
\newcommand\aaboundQ{\boundQ_1}
\newcommand\aaconstC{\constC_1}
\newcommand\aaUnbh{\calU_0}
\newcommand\bbUnbh{\calU_1}
\newcommand\ccUnbh{\calU_2}

The assumption that $\dif$ is \MGENERIC{a}{b}\ on $\BComp$ means that $\hrank[\pw]{\dif} \geq a$ and $\vrank[\pw]{\dif} \geq b$ for all $\pw\in\BComp\cap\ASingMan$, and $\trank[\pw]{\dif} \geq a+b$ for all $\pw\in\BComp$.
The proof consists of four steps.

1) Since the conditions defining \MGENERIC{a}{b}\ maps are $\Wtr{1}$-open and $\BComp$ is compact, there exists $\aaeps>0$ such that for all $\gdif\in\mNF$ satisfying $\xnrm{\gdif-\dif}{1}{\BComp} < \aaeps$ we have that $\hrank[\pw]{\gdif} \geq a$ and $\vrank[\pw]{\gdif} \geq b$ for all $\pw\in\BComp\cap\ASingMan$, and $\trank[\pw]{\gdif} \geq a+b$ for all $\pw\in\BComp$.
Note that then
\[
    \aaUnbh := \{ \gdif\in\mNFnt{a}{b} \mid \xnrm{\gdif-\dif}{1}{\BComp} < \aaeps\}
\]
is a $\Wtr{1}$-open neighborhood of $\dif$ in $\mNF$.

Further, let $\gdif\in\aaUnbh$.
Since $\vtang{\gdif} = \vtang{\Chom[\gdif]_{\hprm,\pt}}$, it follows from Lemma~\ref{lm:hvranks} that
\begin{align*}
    &\hrank[\pw]{\Chom[\gdif]_{\hprm,\pt}} = \hrank[\pw]{\gdif} \geq a, &
    &\vrank[\pw]{\Chom[\gdif]_{\hprm,\pt}} = \vrank[\pw]{\gdif} \geq b
\end{align*}
for all $(\hprm,\pt)\in(0;+\infty)\times[0;1]$ and $\pw\in\BComp\cap\ASingMan$.

Moreover, since $\gdif = \Chom[\gdif]_{\hprm,\pt}$ on $\ScalNbh\setminus\bNbhx{\hprm}$, we have that
\[
    \trank[\pw]{\Chom[\gdif]_{\hprm,\pt}} = \trank[\pw]{\gdif} \geq a+b
\]
for all $\pw\in\BComp\setminus\bNbhx{\hprm}$.

Hence we should estimate $\trank[\pw]{\Chom[\gdif]_{\hprm,\pt}}$ for $\gdif$ sufficiently close to $\dif$, sufficiently small $\hprm>0$, and $\pw\in\BComp\cap\bNbhx{\hprm}$.

2) Fix any $\bbeps\in(0;\aaeps]$.
Then by Lemma~\ref{lm:estim_PQRS_var} there exist $\boundQ>0$ such that for all $\gdif\in\mNF$ satisfying $\xnrm{\gdif-\dif}{\TOP}{\BComp}<\bbeps$ and all $(\hprm,\pt)\in(0;+\infty)\times[0;1]$ and $\pw\in\BComp\cap\cNbhx{\hprm}$ we have that
\begin{equation}\label{equ:bound_Qmatrix}
    \nrm{\mQ[{\Chombt[\gdif]}](\pw)} < \boundQ.
\end{equation}
Then for all such $\gdif$ we have that $\xnrm{\gdif-\dif}{1}{\BComp} \leq \xnrm{\gdif-\dif}{\TOP}{\BComp} < \bbeps \leq \aaeps$.
Hence
\[
    \bbUnbh := \{ \gdif\in\mNF \mid \xnrm{\gdif-\dif}{\TOP}{\BComp} < \bbeps\}
\]
is a $\Wtr{\TOP}$-neighborhood of $\dif$ in $\mNF$ contained in $\aaUnbh$.

3) Further, note that the set $\Scomp = \BComp\cap\ASingMan$ is compact.
Consider the following continuous maps
\begin{align*}
    &\mP[\dif]\colon\Scomp\to\Mat{\bxdim}{\axdim}, &
    &\mS[\dif]\colon\Scomp\to\Mat{\brdim}{\ardim},
\end{align*}
associating to each $(\px,0)\in\Scomp$ the corresponding $\mP$- and $\mS$-blocks of $\JM{\dif}(\px,0)$.
Then by Lemma~\ref{clm:PRS_perturbations}, for the above $\boundQ$ there exists $\boundR>0$ such that for every point $(\px,0)\in\Scomp$ and every matrix $\uJmatr = \amatr{\uPmatr}{\uQmatr}{\uRmatr}{\uSmatr} \in \Mat{\bxdim+\brdim}{\axdim+\ardim}$ satisfying the inequalities
\begin{equation}\label{equ:tJ_Jh_estimates}
\begin{aligned}
&\nrm{\uPmatr-\mP[\dif](\px,0)} < \boundR, &\qquad
&\nrm{\uQmatr} < \boundQ, \\
&\nrm{\uRmatr} < \boundR, &
&\nrm{\uSmatr-\mS[\dif](\px,0)} < \boundR,
\end{aligned}
\end{equation}
we have that
\begin{equation}\label{equ:rank_any_J}
    \rank(\uJmatr) \geq a + b.
\end{equation}

4) Take any $\hgeps < \min\{\boundR,\bbeps\}$ and consider the following $\Wtr{\TOP}$-neighborhood of $\dif$ in $\mNF$:
\[
    \ccUnbh := \{ \gdif\in\mNF \mid \xnrm{\gdif-\dif}{\TOP}{\BComp} < \hgeps \leq \bbeps\} \subset \bbUnbh.
\]
Then again by Lemma~\ref{lm:estim_PQRS_var} there exist $\aaboundQ>0$ and $\aaconstC>0$ such that for all $\gdif\in\ccUnbh$ and all $(\hprm,\pt)\in(0;+\infty)\times[0;1]$ and $(\px,\pv)\in\BComp\cap\cNbhx{\hprm}$ we have that
\begin{equation}\label{equ:PQRS_ineq1}
\begin{aligned}
    &\nrm{\mP[{\Chombt[\gdif]}](\px,\pv) -
        \mP[\dif](\px,0)} < \aaconstC \hprm + \hgeps,
    \\
    &\nrm{\mQ[{\Chombt[\gdif]}](\px,\pv)} < \aaboundQ,\\
    &\nrm{\mR[{\Chombt[\gdif]}](\px,\pv)} < \aaconstC \hprm + \hgeps, \\
    &\nrm{\mS[{\Chombt[\gdif]}](\px,\pv) -
        \mS[\dif](\px,0)} < \aaconstC \hprm + \hgeps.
\end{aligned}
\end{equation}
Since $\ccUnbh\subset\bbUnbh$, we have by~\eqref{equ:bound_Qmatrix} that $\nrm{\mQ[{\Chombt[\gdif]}](\px,\pv)} < \boundQ$ as well.
Denote
\[
    \bhprm = (\boundR - \hgeps)/\aaconstC
    \qquad \text{so} \qquad
    \aaconstC \bhprm + \hgeps = \boundR.
\]
Then by~\eqref{equ:PQRS_ineq1} for all $(\hprm,\pt)\in(0;\bhprm]\times[0;1]$ each of the following three norms:
\begin{align*}
    &\nrm{\mP[{\Chombt[\gdif]}](\px,\pv) - \mP[\dif](\px,0)}, &
    &\nrm{\mR[{\Chombt[\gdif]}](\px,\pv)}, &
    &\nrm{\mS[{\Chombt[\gdif]}](\px,\pv) - \mS[\dif](\px,0)}
\end{align*}
is less than $\aaconstC \bhprm + \hgeps = \boundR$, i.e.\ they satisfy inequalities~\eqref{equ:tJ_Jh_estimates}.
Hence, by~\eqref{equ:rank_any_J}, $\rank(\JM{\Chom[\gdif]_{\hprm,\pt}}) \geq a+b$.

Thus $\Chom[\gdif]_{\hprm,\pt}$ is \MGENERIC{a}{b}\ on $\BComp$ for all $(\hprm,\pt)\in(0;\bhprm]\times[0;1]$ and $\gdif\in\mNF$ satisfying $\xnrm{\gdif-\dif}{\TOP}{\BComp} < \hgeps$.
\end{proof}

\begin{subcorollary}\label{cor:nbh_with_Lab}
Suppose $\ASingMan$ is compact.
Let also $a,b\geq0$.
Then for each $\dif\in\mNFnt{a}{b}$ there exists a $\Wtr{\TOP}$-open neighborhood $\calU_{\dif} \subset \mNFnt{a}{b}$ and a number $\hprm_{\dif}>0$ such that $\Chom[\gdif]_{\hprm,\pt}\in\mNFnt{a}{b}$ as well for all $(\gdif,\hprm,\pt)\in\calU_{\dif}\times(0;\hprm_{\dif}]\times[0;1]$.
\end{subcorollary}
\begin{proof}
\newcommand\AOP{\ASingMan'}
\newcommand\BOP{\BSingMan'}
\newcommand\Unb[1]{\Uman_{#1}}
\newcommand\Vnb[1]{\Vman_{#1}}
\newcommand\xhprm{\hprm_0}
Let $\dif\in\mNFnt{a}{b}$ and $\pw\in\ASingMan$.
Then $\dif(\pw)\in\BSingMan$.
Fix a trivialization $\psi\colon\Bpr^{-1}(\BOP) \to \BLocBase \times \BLocFiber$ (being an open embedding) of $\Bpr$ over some open neighborhood $\BOP \subset \BSingMan$ of $\dif(\pw)$.
Then there exists an open neighborhood $\Unb{}\subset\ATotMan$ of $\pw$ such that $\dif(\Unb{}) \subset \Bpr^{-1}(\BOP)$.
Moreover, decreasing $\Unb{}$ one can also assume that there is a trivialization $\phi\colon\Apr^{-1}(\AOP) \to \ALocBase \times \ALocFiber$ of $\Apr$ over some open neighborhood $\AOP \subset \ASingMan$ of $\pw$.
Thus finally we get the following commutative diagram:
\[
\xymatrix{
    \AOP \ar@{^(->}[d] &
    \ALocBase\times\ALocFiber \ar[l]_-{\pr_1}  &
    \phi(\Unb{}) \ar@{_(->}[l] \ar[r]^{\psi\circ\dif\circ\phi^{-1}}  &
    \psi(\Vnb{}) \ar@{^(->}[r] &
    \BLocBase\times\BLocFiber  \ar[r]^-{\pr_1} &
    \BOP \ar@{^(->}[d]
    \\
    \ASingMan & \Apr^{-1}(\AOP) \ar@{_(->}[l]  \ar[u]^{\phi} &
    \Unb{} \ar@{_(->}[l] \ar[r]^-{\dif} \ar[u]^{\phi} &
    \Vnb{} \ar@{^(->}[r]  \ar[u]^{\psi} &
    \Bpr^{-1}(\BOP)  \ar@{^(->}[r] \ar[u]^{\psi} &
    \BTotMan
}
\]
The map $\psi\circ\dif\circ\phi^{-1}$ will be called a \term{$(\Apr,\Bpr)$-coordinate representation} of $\dif$, so it is a coordinate representation of $\dif$ with respect to trivialized charts of vector bundles $\Apr$ and $\Bpr$.
Note that it is determined by the triple $(\Unb{}, \phi, \psi)$.

Since $\ASingMan$ is compact and $\dif\in\mNFnt{a}{b}$, it follows from Lemma~\ref{lm:make_Chom_neat} that there exist finitely many $(\Apr,\Bpr)$-coordinate representations
\[
     \bigl(
        \Unb{i}, \quad
        \phi_i\colon\Apr^{-1}(\AOP_i)\to\ALocBase\times\ALocFiber,  \quad
        \psi_i\colon\Bpr^{-1}(\BOP_i)\to\BLocBase\times\BLocFiber
        \bigr), \
        i=1,\ldots,s,
\]
of $\dif$ and compact \STL\ subsets $\BComp_i\subset\Unb{i}$ such that if we denote $\BComp:=\mathop{\cup}\limits_{i=1}^{s} \BComp_{i}$ and $\Uman:=\mathop{\cup}\limits_{i=1}^{s} \Unb{i}$, then
\begin{enumerate}
\item
there exists $\xhprm>0$ such that $\ASingMan \subset \cNbhx{\xhprm} \subset \BComp \subset \Uman$;
\item
for each $i=1,\ldots,s$, there exist $\hgeps_i>0$ and $\hprm_i>0$ such that for every $\Cinfty$ map $\gdif\colon \phi_i(\Unb{i}) \to \BLocBase\times\BLocFiber$ such that
\begin{align*}
   &\gdif\bigl(\phi_i(\Unb{i})\cap \ALocBase\times 0\bigr) \subset \BLocBase\times 0, &
   &\xnrm{\gdif - \psi_i \circ \dif \circ \phi_i^{-1}}{\TOP}{\phi_i(\BComp_i)} < \hgeps_i,
\end{align*}
the map $\Chom[\gdif]_{\hprm,\pt}$ is also \MGENERIC{a}{b}\ on $\phi_i(\BComp_i)$ for all $(\hprm,\pt)\in(0;\hprm_i]\times[0;1]$.
\end{enumerate}
Then
\[
    \calU_{\dif} := \{
    \dif'\in\mNFnt{a}{b} \mid
        \dif'(\Unb{i}) \subset \Bpr^{-1}(\BOP_i),
        \
        \xnrm{\psi_i \circ\dif'\circ\phi_i^{-1} - \psi_i \circ \dif \circ \phi_i^{-1}}{\TOP}{\phi_i(\BComp_i)} < \hgeps_i
    \}
\]
is a $\Wtr{\TOP}$-open neighborhood of $\dif$ in $\mNFnt{a}{b}$.

Denote $\hprm_{\dif} := \min \{\xhprm, \hprm_1,\ldots,\hprm_{s}\}$.
Then for all $(\gdif,\hprm,\pt)\in\calU_{\dif}\times(0;\hprm_{\dif}]\times[0;1]$ we have that $\Chom[\gdif]_{\hprm,\pt}$ is \MGENERIC{a}{b}\ on all of $\ScalNbh$, i.e.\ belongs to $\mNFnt{a}{b}$.
\end{proof}

Let us also recall one simple lemma about continuous functions on paracompact spaces:
\begin{sublemma}\label{lm:func_on_paracompact}
Let $\alpha=\{\Uman_i\}_{i\in I}$ be an open cover of a paracompact space $\Xman$.
Suppose also for each $i\in I$ we are given a positive number $\hprm_i>0$.
Then there exists a continuous function $\hpFunc\colon\Xman\to(0;+\infty)$ with the following property: for every $x\in\Xman$ there exist $i\in I$ such that $x\in\Uman_i$ and $\hpFunc(x) \leq \hprm_i$.
\end{sublemma}
\begin{proof}
Take any open locally finite cover $\beta = \{\Vman_j\}_{j \in J}$ of $\Xman$ subordinated to $\alpha$, so for each $j\in J$ there exists $i_j\in I$ such that $\Vman_j\subset \Uman_{i_j}$.
Let also $\{\phi_{j}\colon\Xman\to[0;1] \mid j\in J\}$ be a partition of unity subordinated to the cover $\beta$, i.e.\ it is a family of continuous functions such that $\supp\phi_{j} \subset \Vman_j$ for every $j\in J$ and $\sum\limits_{j\in J}\phi_{j} \equiv 1$.
Then the following function $\hpFunc = \sum\limits_{j\in J}\hprm_{i_j}\phi_{j} \colon \Xman\to(0;\infty)$ is well defined, continuous and satisfies the required properties.

Indeed, let $x\in \Xman$.
Since $\beta$ is locally finite, there are only finitely many indexes $j_1,\ldots,j_k \in J$ such that $\phi_{j_s}(x)\ne0$.
This implies that
\[
    x \in
    \mathop{\cap}\limits_{s=1}^{k} \Vman_{j_s} \subset
    \mathop{\cap}\limits_{s=1}^{k} \Uman_{i_{j_s}}
\]
Moreover, let $i\in\{i_{j_1},\ldots,i_{j_k}\}$ be any index for which
$\hprm_{i} = \max\limits_{s=1,\ldots,k}\{\hprm_{i_{j_s}}\}$.
Then
\[
    \hpFunc(x)
        = \sum_{s=1}^{k} \hprm_{i_{j_s}} \phi_{j_s}(x)
        \leq \sum_{s=1}^{k} \hprm_{i} \phi_{j_s}(x)
        =  \hprm_{i} \sum_{s=1}^{k} \phi_{j_s}(x)
        =  \hprm_{i}.
    \qedhere
\]
\end{proof}

\begin{subcorollary}\label{cor:def_L_Llin_pres_Lntab}
There exists a strictly positive function $\hpFunc\colon\mNFnt{a}{b}\to(0;\infty)$ continuous from the topology $\Wtr{\TOP}$ on $\mNFnt{a}{b}$ such that the image of the map
\[
    \Ghom\colon\mNFnt{a}{b}\times[0;1]\to\mNF,
    \qquad
    \Ghom(\dif,\pt)(\pw) = \Chom[\dif](\hpFunc(\dif),\pt,\pw).
\]
is contained in $\mNFnt{a}{b}$.
Hence for each $r\geq\TOP$ we have the induced continuous map $\Ghom\colon\WTopol{r}{\mNFnt{a}{b}}\times[0;1]\to\WTopol{r-1}{\mNFnt{a}{b}}$.

In particular, if we endow $\mNFnt{a}{b}$ with the topology $\Wtr{\infty}$, then $\Ghom$ is a deformation of $\mNFnt{a}{b}$ into $\mNFntlin{a}{b}$, and so the inclusion $\mNFntlin{a}{b} \subset \mNFnt{a}{b}$ is a homotopy equivalence.
\end{subcorollary}
\begin{proof}
\newcommand\jidx{j}
\newcommand\Jidx{J}
\newcommand\cV[1]{\calV_{#1}}
\newcommand\cVj[1]{\cV{\jidx_{#1}}}
\newcommand\cU[1]{\calU_{#1}}
\newcommand\cUh[1]{\cU{\dif_{#1}}}
\newcommand\cUhj[1]{\cUh{\jidx_{#1}}}

By Corollary~\ref{cor:nbh_with_Lab}, for each $\dif\in\mNFnt{a}{b}$ there exists a $\Wtr{\TOP}$-open neighborhood $\calU_{\dif}\subset\mNFnt{a}{b}$ and a number $\hprm_{\dif}>0$ such that $\Chom[\gdif]_{\hprm,\pt}\in\mNFnt{a}{b}$ for all $(\gdif,\hprm,\pt)\in\calU_{\dif}\times(0;\hprm_{\dif}]\times[0;1]$.

Since $\ScalNbh$ is compact, the space $\WTopol{\TOP}{\mNFnt{a}{b}}$ is metrizable, and therefore paracompact by Stone's theorem.
Hence, by Lemma~\ref{lm:func_on_paracompact} there exists a continuous function $\hpFunc\colon\WTopol{\TOP}{\mNFnt{a}{b}}\to(0;+\infty)$ having the following property: for every $\gdif\in\mNFnt{a}{b}$ there is $\dif\in\mNFnt{a}{b}$ such that $\gdif\in\calU_{\dif}$ and $\hpFunc(\gdif) \leq \hprm_{\dif}$.
Therefore $\Ghom(\gdif,\pt) = \Chom[\gdif]_{\hpFunc(\gdif),\pt}\in\mNFnt{a}{b}$ for all $\pt\in[0;1]$.
\end{proof}

\newcommand\hfam{\xi}
\newcommand\HF[1][\empty]{G\ifx#1\empty\relax\else(#1)\fi}
\newcommand\SCover[2]{(#1,#2)}
\newcommand\NWCover[1]{\SCover{\BNbhSingMan}{#1}\text{-cover}}
\newcommand\CH{\cNbhx{\hprm_{\dif}}}

\begingroup
\subsection{Proof of Theorem~\ref{th:linearization_EMN_and_subspaces}}
\label{sect:proof_th:linearization_EMN_and_subspaces}
Let $\AManif$ and $\BManif$ be two connected manifolds, $\ASingMan\subset\AManif$ and $\BSingMan\subset\BManif$ be connected \GOOD\ submanifolds, equipped with tubular neighborhoods $\Apr\colon\ATotMan\to\ASingMan$ and $\Bpr\colon\BTotMan\to\BSingMan$ of $\ASingMan$ and $\BSingMan$.
Assume that $\ASingMan$ is compact.
Let also $a,b\geq0$, $\ANbhSingMan \subset \ATotMan\cap\AManif$ be a compact neighborhood of $\ASingMan$, $\BNbhSingMan\subset\BTotMan\cap\BManif$ be a neighborhood of $\BSingMan$, $\mAB = \Ci{(\AManif,\ASingMan)}{(\BManif,\BSingMan)}$, and $\mABnt{a}{b} \subset \mAB$ be the subspace consisting of \MGENERIC{a}{b} maps.
We need to construct a deformation of $\mABnt{a}{b}$ into its subspace $\mABntlin{a}{b}$ of linearized maps.
The proof is similar to Corollaries~\ref{cor:def_L_Llin_not_pres_Lntab} and~\ref{cor:def_L_Llin_pres_Lntab}.
However, the main problem is that given a compact neighborhood $\ScalNbh\subset\ANbhSingMan$ of $\ASingMan$, there might exist maps $\dif\in\mABnt{a}{b}$ such that $\dif(\ScalNbh) \not\subset \BTotMan$.
Therefore we need to adapt the arguments of those corollaries.

1) First let us fix some smooth compact \STL\ neighborhood $\ScalNbh\subset\ANbhSingMan$ of $\ASingMan$.
Denote $\mNB = \Ci{(\ScalNbh,\ASingMan)}{(\BManif,\BSingMan)}$ and let $\mNBnt{a}{b}$ be the subset consisting of \MGENERIC{a}{b} maps.
Endow $\mNB$ with the topology $\Wtr{\infty}$.
Since $\ScalNbh$ is compact, $\mNB$ and therefore $\mNBnt{a}{b}$ are metrizable.
Therefore they are paracompact.

\newcommand\aahprm{\hprm_0}
\newcommand\bbhprm{\hprm_1}
\newcommand\cchprm{\hprm_{\dif}}

\newcommand\aaUnb{\calU_0}
\newcommand\bbUnb{\calU_1}
\newcommand\ccUnb{\calU_{\dif}}

\newcommand\aamNF{\mNF^{\dif}}
\newcommand\aamNFnt[2]{\mNF^{\dif}_{(#1,#2)}}
\newcommand\rmap[1][\empty]{\rho_{#1}}

2) Let $\dif\in\mNBnt{a}{b}$.
Since $\dif(\ASingMan) \subset \BSingMan$, there exists $\aahprm>0$ such that
\begin{align}\label{equ:h_Thprm0__BNbhSingMan}
    &\cNbhx{\aahprm} \subset \ScalNbh, &
    &\glHom[\dif]([0;1]\times\cNbhx{\aahprm}) \subset \BNbhSingMan.
\end{align}
In particular,
\[
    \aaUnb = \{ \gdif \in \mNBnt{a}{b} \mid \glHom[\gdif]([0;1]\times\cNbhx{\aahprm}) \subset \BNbhSingMan\}
\]
is a $\Wtr{1}$-neighborhood of $\dif$ in $\mNBnt{a}{b}$.

Let $\aamNF := \Ci{(\cNbhx{\aahprm},\ASingMan)}{(\BTotMan,\BSingMan)}$ and $\aamNFnt{a}{b}$ be the subset consisting of \MGENERIC{a}{b} maps.
Then by Corollary~\ref{cor:nbh_with_Lab}, applied to $\aamNFnt{a}{b}$, for the restriction $\restr{\dif}{\cNbhx{\aahprm}} \in \aamNFnt{a}{b}$ there exists a $\Wtr{\TOP}$-open neighborhood $\bbUnb \subset \aamNFnt{a}{b}$ and  $\bbhprm>0$ such that $\Chom[\gdif]_{\hprm,\pt}$ is also \MGENERIC{a}{b}\ on $\cNbhx{\aahprm}$ for all $(\gdif,\hprm,\pt)\in\bbUnb\times(0;\bbhprm]\times[0;1]$.

Consider the restriction map $\rmap\colon\mNBnt{a}{b} \to \aamNFnt{a}{b}$, $\rmap(\dif) = \restr{\dif}{\cNbhx{\aahprm}}$.
Then $\rmap$ is $\Wtr{r,r}$-continuous for all $r\geq0$.
In particular,
\[ \ccUnb := \aaUnb \cap \rmap^{-1}(\bbUnb) \]
is a $\Wtr{\TOP}$-open neighborhood of $\restr{\dif}{\ScalNbh}$ in $\mNBnt{a}{b}$.

Denote $\hprm_{\dif} := \min\{\aahprm,\bbhprm\}$.
Then by Lemma~\ref{lm:func_on_paracompact} there is a continuous function $\hpFunc\colon\WTopol{\TOP}{\mNBnt{a}{b}}\to(0;+\infty)$ having the following property: for every $\gdif\in\mNBnt{a}{b}$ there is $\dif\in\mNBnt{a}{b}$ such that $\gdif\in\ccUnb$ and $\hpFunc(\gdif) \leq \hprm_{\dif}$.
Therefore $\Chom[\gdif]_{\hpFunc(\gdif),\pt}\in\mNBnt{a}{b}$ for all $\pt\in[0;1]$.

3) In particular, we have a well defined mapping $\Ghom\colon\mABnt{a}{b}\times[0;1]\to\mABnt{a}{b}$
\begin{equation}\label{equ:Ghom_formula}
    \Ghom(\dif,\pt)(\pw) =
    \begin{cases}
        \Chom[\dif](\hpFunc(\restr{\dif}{\ScalNbh}),\pt, \pw), & \pw\in\cNbhx{\hpFunc(\restr{\dif}{\ScalNbh})}, \\
        \dif(\pw), & \pw\in\AManif\setminus\cNbhx{\hpFunc(\restr{\dif}{\ScalNbh})},
    \end{cases}
\end{equation}
and for each $\pt\in[0;1]$ define another ``slice'' map
\[
    \Ghom_{\pt}\colon\mABnt{a}{b}\to\mABnt{a}{b},
    \qquad
    \Ghom_{\pt}(\dif)=\Ghom(\dif,\pt).
\]
Then, as in Corollary~\ref{cor:def_L_Llin_not_pres_Lntab}, $\Ghom$ is a deformation of $\mABnt{a}{b}$ into $\mABntlin{a}{b}$ which preserves the space of vertical maps $\mABvrt \cap \mABnt{a}{b}$.

Let us verify properties~\ref{enum:th:lin_main:V}--\ref{enum:th:lin_main:alpha} of $\Ghom$.

\ref{enum:th:lin_main:V}
The identity $\vtang{\Ghom(\dif,\pt)} = \vtang{\dif}$ for all $\dif\in\mAB$ follows from Lemma~\ref{lm:glHom_Chom}\ref{enum:glHom_Chom:H_0}.

\ref{enum:th:lin_main:support}
By the construction, for each $\dif\in\mABnt{a}{b}$ there exists a neighborhood $\cNbhx{\aahprm} \subset \ANbhSingMan$ of $\ASingMan$ such that $\dif(\cNbhx{\aahprm})\subset\BNbhSingMan$ and $\Ghom(\dif,\pt)=\dif$ on $\AManif\setminus\cNbhx{\aahprm}$.
Thus one should put $\ANbhSingMan[\dif] = \cNbhx{\aahprm}$.

\ref{enum:th:lin_main:strong_top}
Due~\ref{enum:th:lin_main:support}, $\Ghom$ changes each $\dif$ only on \emph{the same compact subset $\ANbhSingMan$}.
This implies that $\Ghom$ is continuous if $\mABnt{a}{b}$ is endowed with the topology $\Str{\infty}$.

\ref{enum:th:lin_main:alpha}
Due to~\ref{enum:th:lin_main:V}, $\Ghom$ also preserves the intersection of $\mABnt{a}{b}$ with the spaces $\mABh{\hmap}$ and $\mABv{\vmap}$ for all $\Cinfty$ maps $\hmap\colon\ASingMan\to\BSingMan$ and vector bundle morphisms $\vmap\colon\ATotMan\to\BTotMan$.
Hence $\Ghom$ deforms the family $\alpha$ defined by~\eqref{equ:th:lin_main:alpha_family} into the family $\alpha \cap \mABlin$.

This completes the proof of Theorem~\ref{th:linearization_EMN_and_subspaces} modulo Lemma~\ref{lm:estim_PQRS_var}.
\qed
\endgroup

\subsection{Proof of Theorem~\ref{th:linearization_fol}}
\label{sect:proof:th:linearization_fol}
Suppose that $\AManif$ and $\BManif$ are equipped with some partitions $\AFol$ and $\BFol$ respectively, such that $\ASingMan$ is $\AFol$-saturated and $\BSingMan$ is $\BFol$-saturated.
Let further $\ASFol = \AFolPr(\ASingMan) \subset \AFol$ and $\BSFol = \BFolPr(\BSingMan) \subset \BFol$ be the corresponding sets of leaves contained in $\ASingMan$ and $\BSingMan$ respectively, $\AScalNbh \subset \ATotMan\cap\AManif$ and $\BScalNbh\subset\BTotMan\cap\BManif$ be two neighborhoods of $\ASingMan$ and $\BSingMan$ respectively, and $\folmap\colon(\AFol,\ASFol)\to(\BFol,\BSFol)$ be a map such that
\begin{itemize}
\item
$\folmap$ is \FOLNEAT, i.e.\ $\ASFol=\folmap^{-1}(\BSFol)$;
\item
$(\AScalNbh,\BScalNbh)$ is $\folmap$-scalable, i.e.\ if $\tau\geq1$ and $\px\in\AScalNbh$ and $\py\in\BScalNbh$ are such that $\tau\px\in\AScalNbh$, $\tau\py\in\BScalNbh$, and $\BFolPr(\py) = \folmap(\AFolPr(\px))$, then $\BFolPr(\tau\py) = \folmap(\AFolPr(\tau\px))$;
\item
$\BScalNbh$ is $\BFol$-isolating for $\BSingMan$, i.e.\ $(\overline{\bleaf}\setminus\bleaf) \cap \BScalNbh \subset \BSingMan$ for every leaf $\bleaf\in\BFol$.
\end{itemize}

Further, let $\ANbhSingMan\subset\ATotMan\cap\AManif$ and $\BNbhSingMan\subset\BTotMan\cap\BManif$ be some neighborhoods of $\ASingMan$ and $\BSingMan$ respectively, $\ScalNbh\subset\ANbhSingMan$ be any smooth compact neighborhood of $\ASingMan$, $\mNB = \Ci{(\ScalNbh,\ASingMan)}{(\BManif,\BSingMan)}$, $\mNBnt{a}{b}$ be the subspace consisting of \MGENERIC{a}{b} maps, and $\Ghom\colon\mABnt{a}{b}\times[0;1]\to\mABnt{a}{b}$ be the homotopy constructed in the proof of Theorem~\ref{th:linearization_EMN_and_subspaces} and given by~\eqref{equ:Ghom_formula}.

By Lemma~\ref{lm:decr_scal_isol_nbhs} one can decrease neighborhoods $\AScalNbh$ and $\BScalNbh$ without losing the above properties.
Similarly, we can decrease the neighborhoods $\ANbhSingMan$ and $\BNbhSingMan$.
Therefore one can assume that $\BScalNbh=\BNbhSingMan$ and $\ScalNbh \subset \ANbhSingMan = \AScalNbh$.

We claim that \emph{$\Ghom$ preserves the subspace $\mABfol{\folmap} \cap \mABnt{a}{b}$ of \ftfol\ maps being also \VERTINJECTIVE}.

Indeed, let $\dif\in\mABfol{\folmap}\cap\mABnt{a}{b}$, i.e.\ $\BFolPr\circ\dif = \folmap\circ\AFolPr$ and $(\vtang{\dif})^{-1}(\BSingMan)=\ASingMan$.
Let also $\pv\in\AManif$, $\aleaf = \AFolPr(\pv)\in\AFol$ be the leaf containing $\pv$, and $\bleaf = \folmap(\aleaf)$.
Then the identity $\BFolPr\circ\dif = \folmap\circ\AFolPr$ means that $\dif(\aleaf)\subset\bleaf$, in particular, $\dif(\pv)\in\bleaf$.
Now let
\begin{gather*}
    d = \hpFunc(\restr{\dif}{\ScalNbh}),  \qquad
    \pt\in[0;1], \qquad
    \ps = \phifunc(d,\pt,\pv),
    \\
    \pw =
    \Ghom(\dif,\pt)(\pv)
        \stackrel{\eqref{equ:Ghom_formula}}{=}
    \Chom[\dif](d, \pt)(\pv)
        \stackrel{\eqref{equ:Chom_formula}}{=}
    \glHom[\dif](\ps,\pv).
\end{gather*}
We should verify that $\pw \in \bleaf$ as well.
Consider three cases.

1) Suppose $\pv\in\AManif\setminus\bNbhx{d}$.
Then, by~\eqref{equ:Ghom_formula}, $\pw=\Ghom(\dif,\pt)(\pv) = \dif(\pv)\in\bleaf$.

2) Suppose $\pv\in\bNbhx{d} \subset \AScalNbh$ and $\ps > 0$.
Then
\begin{align*}
    \pw = \glHom[\dif](\ps,\pv) = \tfrac{1}{\ps}\dif(\ps\pv).
\end{align*}
We claim that $\tau:=\tfrac{1}{\ps} \geq1$ and the points $\px := \ps\pv$ and $\py:=\dif(\px)=\ps\pw$ satisfy assumptions of Definition~\ref{def:folmap-scalable}, see Figure~\ref{fig:scalable_nbh}, i.e.
\begin{align}\label{equ:check_scallability_xyw}
&\BFolPr(\py) = \folmap\circ\AFolPr(\px), &
&\px,\tfrac{1}{\ps}\px = \pv\in\AScalNbh, &
&\py,\tfrac{1}{\ps}\py = \pw\in\BScalNbh.
\end{align}
Then $\folmap$-scalability of the pair $(\AScalNbh,\BScalNbh)$ will imply that:
\[
    \BFolPr(\pw) = \BFolPr(\tau\py)
                 = \folmap\circ\AFolPr(\tau\px)
                 = \folmap\circ\AFolPr(\pv)
                 = \folmap(\aleaf)
                 = \bleaf.
\]
In other words, we will get that $\pw\in\bleaf$.

\begin{figure}
\centering
\includegraphics[height=2.5cm]{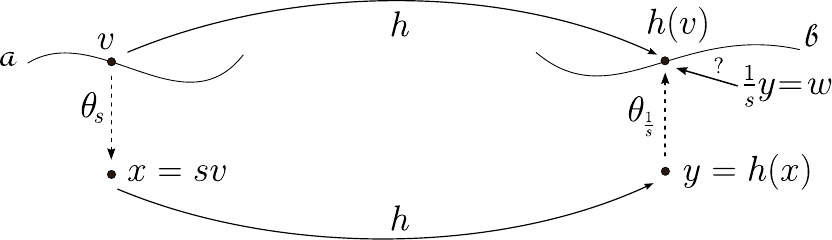}
\caption{Proof of case 2)}
\label{fig:scalable_nbh}
\end{figure}

\begin{itemize}[leftmargin=*]
\item
Indeed, since $\dif$ is \ftfol, $\BFolPr(\py)=\BFolPr\circ\dif(\px)=\folmap\circ\AFolPr(\px)$.

\item
Also, by assumption, $\tfrac{1}{\ps}\px = \pv\in \bNbhx{d}\subset\AScalNbh$.
Moreover, since $\bNbhx{d}$ is \STL\ and $\ps\leq 1$, we obtain that $\px = \ps\pv\in\ps\bNbhx{d}\subset\bNbhx{d}\subset\AScalNbh$ as well.

\item
Due to~\eqref{equ:h_Thprm0__BNbhSingMan}, we have that $\glHom[\dif]([0;1]\times\cNbhx{d}) \subset \BScalNbh$.
In particular,
\begin{align*}
    &\py=\dif(\px)=\glHom[\dif](1,\px)\in \BScalNbh, &
    &\tfrac{1}{\ps}\py=\pw =\glHom[\dif](s,\pv)  \in \BScalNbh.
\end{align*}
\end{itemize}
This proves~\eqref{equ:check_scallability_xyw}, which, as noted above, implies that $\pw\in\bleaf$.

3) Suppose $\pv\in\bNbhx{d} \subset \AScalNbh$ but
\[
    \phifunc(d,\pt,\pv)
        =
    \pt+(1-\pt)\mu\bigl(\tfrac{\fnrm(\pv)}{d^2}\bigr)
        \stackrel{\eqref{equ:defn:phi}}{=}
    0.
\]
This is possible only when $\pt=0$ and $\mu\bigl(\tfrac{\fnrm(\pv)}{d^2}\bigr)=0$, i.e.\ $\pv\in\aNbhx{d}$.
Hence
\begin{equation}\label{equ:w_Gd_ht__Vh}
    \pw:=\Ghom(\dif,0)(\pv) \stackrel{\eqref{equ:homotopy_Ahom_}}{=} \vtang{\dif}(\pv).
\end{equation}

Consider the path $\gamma\colon[0;1]\to\BScalNbh\subset\BManif$, $\gamma(\pt) = \Ghom(\dif,\pt)(\pv)$.
Then, by 2), $\gamma\bigl((0;1] \bigr) \subset \bleaf$, and so by continuity of $\gamma$ we have that $\pw=\gamma(0)\in\overline{\bleaf}$.
If $\pw\in\bleaf$, then there is nothing to prove.
Otherwise $\gamma(0)\in(\overline{\bleaf}\setminus\bleaf)\cap\BScalNbh$, and since $\BScalNbh$ is $\BFol$-isolating, this intersection is contained in $\BSingMan$, whence $\pw=\gamma(0)\in\BSingMan$.

We claim that then $\pv\in\ASingMan$.
This will imply that $\pw \stackrel{\eqref{equ:w_Gd_ht__Vh}}{=} \vtang{\dif}(\pv) \stackrel{\eqref{equ:Vh_on_as__h_on_as}}{=} \dif(\pv) \in\bleaf$.
Indeed, since $\dif$ is \VERTINJECTIVE, $\ASingMan = (\vtang{\dif})^{-1}(\BSingMan)$, and therefore
\[
    \pv \in (\Ghom(\dif,0))^{-1}(\pw)  \cap \aNbhx{d}
    \subset
    (\vtang{\dif})^{-1}(\pw)
    \subset \ASingMan.
\]
This proves that $\Ghom(\dif,\pt)$ is \ftfol.

Finally note that, by Theorem~\ref{th:linearization_EMN_and_subspaces}, $\vtang{\Ghom(\dif,\pt)} = \vtang{\dif}$, and therefore
\[ (\vtang{\Ghom(\dif,\pt)})^{-1}(\BSingMan) = (\vtang{\dif})^{-1}(\BSingMan) = \ASingMan, \]
i.e.\ $\Ghom(\dif,\pt)$ is \VERTINJECTIVE\ as well.
Hence $\mABfol{\folmap}\cap\mABnt{a}{b}$ is $\Ghom$-invariant, and therefore the inclusion $\mABfol{\folmap}\cap\mABntlin{a}{b} \subset \mABfol{\folmap}\cap\mABnt{a}{b}$ is a homotopy equivalence as well.

\newcommand\mAhom{\glHom}
\newcommand\mChom{\Chom}
\newcommand\mDhom{\Dhom}
\newcommand\mAhomt{\mAhom_{\pt}}
\newcommand\mChombt{\Chombt}
\newcommand\mDhombt{\Dhombt}

\section{Proof of Lemma~\ref{lm:estim_PQRS_var}}
\label{sect:proof:lm:estim_PQRS_var}

Before proving this lemma we obtain certain estimates on ranks of Jacobi matrices of the maps $\Chom[\dif]_{\hprm,\pt}$.

\subsection{Two lemmas on matrices ranks}
\label{sect:lemmas_on_ranks}
Let $\axdim,\bxdim,\ardim,\brdim\geq0$ be non-negative integers and $J = \amatr{\Pmatr}{\Qmatr}{0}{\Smatr} \in \Mat{\bxdim+\brdim}{\axdim+\ardim}$ be a splitting of some matrix $J$ into blocks such that $\Pmatr\in\Mat{\bxdim}{\axdim}$ and $\Smatr\in\Mat{\brdim}{\ardim}$.
If either of those numbers is zero, then the corresponding parts of $J$ vanish.
Note that we have the following inequalities:
\[
\newcommand\rlq{\text{\rotatebox[origin=c]{270}{$\leq$}}}
\begin{matrix}
\rank(\Pmatr)         & + & \rank(\Smatr)         & \leq & \rank(J)             \\
  \rlq &   & \rlq  &      & \rlq \\
\min\{\axdim,\bxdim\} & + & \min\{\ardim,\brdim\} & \leq & \min\{\axdim+\ardim,\bxdim+\brdim\}
\end{matrix}
\]
The following lemma is worth mentioning, though it is not used in the proofs.
\begin{sublemma}\label{lm:rankJ}
Suppose
\begin{enumerate}[topsep=1ex, label={\rm(\arabic*)}]
\item\label{enum:lm:rankJ:A} either $\axdim \geq \bxdim = \rank(\Pmatr)$, so this rank equals the number of rows of $\Pmatr$;
\item\label{enum:lm:rankJ:B} or $\rank(\Smatr) = \ardim \leq \brdim$, so this rank equals the number of columns of $\Smatr$.
\end{enumerate}
Then $\rank(\Pmatr) + \rank(\Smatr) = \rank(J)$.
\end{sublemma}
\begin{proof}
\ref{enum:lm:rankJ:A}
Suppose $\axdim \geq \bxdim = \rank(\Pmatr)$.
Then by adding linear combinations of columns of $\Pmatr$ we can eliminate matrix $\Qmatr$.
This will not change the rank of $J$ and will reduce it to the matrix $J' = \amatr{\Pmatr}{0}{0}{\Smatr}$ having the rank $\rank(\Pmatr) + \rank(\Smatr)$.

Similarly, in the case~\ref{enum:lm:rankJ:B} the matrix $J$ reduces to $J'$ by adding linear combinations of rows of $\Smatr$ to rows of the matrix $\Qmatr$.
\end{proof}

\begingroup
\newcommand\ball[2]{\mathcal{B}_{#2}(#1)}

\begin{sublemma}\label{clm:PRS_perturbations}
Let $a,b\geq0$, $\Scomp$ be a compact space, and
\begin{align*}
    &\Pmatr\colon\Scomp\to\Mat{\bxdim}{\axdim},&
    &\Smatr\colon\Scomp\to\Mat{\brdim}{\ardim}
\end{align*}
be two continuous maps such that $\trank{\Pmatr(\px)}\geq a$ and $\trank{\Smatr(\px)}\geq b$ for all $\px\in\Scomp$.
Then for any (arbitrary large) number $\boundQ>0$ there exists a (sufficiently small) $\boundR>0$ with the following property:
\begin{itemize}[leftmargin=*, nosep]
\item For each $\px\in\Scomp$ and every matrix $\uJmatr = \amatr{\uPmatr}{\uQmatr}{\uRmatr}{\uSmatr} \in \Mat{\bxdim+\brdim}{\axdim+\ardim}$ satisfying the inequalities
\begin{equation}\label{equ:uJ_J_dqdd}
\begin{aligned}
&\nrm{\uPmatr-\Pmatr(\px)} < \boundR, &\qquad
&\nrm{\uQmatr} < \boundQ, \\
&\nrm{\uRmatr} < \boundR, &
&\nrm{\uSmatr-\Smatr(\px)} < \boundR,
\end{aligned}
\end{equation}
we have that $\trank{\uJmatr} \geq a+b$.
\end{itemize}
\end{sublemma}
\begin{proof}
1) First consider the case when $\Scomp$ is a one-point set $\{*\}$.
Then every map of $\Scomp$ into $\Mat{p}{q}$ is just a single matrix, so we can omit the point $*$ from the notations below.

1.1) Assume also that the matrices $\Pmatr := \Pmatr(*)$ and $\Smatr := \Smatr(*)$ are square and non-degenerate, so $\axdim=\bxdim$ and $\ardim=\brdim$.
Let $\uQmatr\in\Mat{\axdim}{\ardim}$ be any matrix and $\Jmatr = \amatr{\Pmatr}{\uQmatr}{0}{\Smatr}$.
Then
\begin{equation}\label{equ:det:P_QB_0_S}
    \det\Jmatr = \det(\Pmatr)\det(\Smatr) \ne 0.
\end{equation}
In particular, $\rank(\Pmatr)+\rank(\Smatr) = \axdim+\ardim = \rank(\Jmatr)$.
It easily follows that if $\nrm{\uQmatr}\leq \boundQ$, then one can find $\boundR>0$ such that $\det\amatr{\uPmatr}{\uQmatr}{\uRmatr}{\uSmatr} \ne 0$ whenever conditions~\eqref{equ:uJ_J_dqdd} holds.
In other words, $\rank(\uJmatr) = \rank(\Jmatr) = \rank(\Pmatr)+\rank(\Smatr)\geq a+b$, which proves the lemma in this case.

1.2) In the general case, $\Pmatr$ and $\Smatr$ contain some square non-degenerate submatrices $\Pmatr'$ and $\Smatr'$ of sizes $\rank(\Pmatr)$ and $\rank(\Smatr)$ respectively.
Choose $\boundR$ as in 1.1) applied to matrices $\Pmatr'$ and $\Smatr'$.
Let also
\begin{itemize}
\item $\uJmatr = \amatr{\uPmatr}{\uQmatr}{\uRmatr}{\uSmatr}$ be any real matrix satisfying~\eqref{equ:uJ_J_dqdd};
\item $\uPmatr'$ be the submatrix of $\uPmatr$ consisting of the same entries as $\Pmatr'$ in $\Pmatr$;
\item $\uSmatr'$ be the submatrix of $\uSmatr$ consisting of the same entries as $\Smatr'$ in $\Smatr$.
\end{itemize}
Then $\uJmatr$ contains a square submatrix $\uJmatr' = \amatr{\uPmatr'}{\uQmatr'}{\uRmatr'}{\uSmatr'}$ for some submatrix $\uQmatr'$ of $\uQmatr$ and submatrix $\uRmatr'$ of $\uRmatr$.
Moreover,
\begin{equation}\label{equ:uJ_J_dqdd1}
\begin{aligned}
&\nrm{\uPmatr'-\Pmatr'} \leq \nrm{\uPmatr-\Pmatr} < \boundR, &\qquad
&\nrm{\uQmatr'} \leq \nrm{\uQmatr} < \boundQ, \\
&\nrm{\uRmatr'} \leq\nrm{\uRmatr} < \boundR, &
&\nrm{\uSmatr'-\Smatr'} \leq \nrm{\uSmatr-\Smatr} < \boundR.
\end{aligned}
\end{equation}
Therefore, by 1), the matrix $\uJmatr'$ is non-degenerate and has rank $\rank(\Pmatr) +\rank(\Smatr)$.
Hence $\rank(\uJmatr) \geq \rank(\uJmatr') = \rank(\Pmatr) +\rank(\Smatr) \geq a+b$.

2) Suppose now that $\Scomp$ is an arbitrary compact space.
Then by 1.2) for the given $\boundQ$ for every $\px\in\Scomp$ there exists $\boundR_{\px}$ such that for every matrix $\uJmatr = \amatr{\uPmatr}{\uQmatr}{\uRmatr}{\uSmatr}$ satisfying the inequalities
\begin{equation}\label{equ:uJ_J_dqdd_xi}
\begin{aligned}
&\nrm{\uPmatr-\Pmatr(\px)} < \boundR_{\px}, &\qquad
&\nrm{\uQmatr} < \boundQ, \\
&\nrm{\uRmatr} < \boundR_{\px}, &
&\nrm{\uSmatr-\Smatr(\px)} < \boundR_{\px},
\end{aligned}
\end{equation}
we have that $\rank(\uJmatr)\geq a+b$.

Let $\Uman_{\px} \subset \Scomp$ be a neighborhood of $\px$ such that for all $\py\in\Uman_{\px}$
\begin{align}\label{equ:PxPy_SxSy_estim}
    \nrm{\Pmatr(\py)-\Pmatr(\px)} &< \tfrac{\boundR_{\px}}{2}, &
    \nrm{\Smatr(\py)-\Smatr(\px)} &< \tfrac{\boundR_{\px}}{2}.
\end{align}
Since $\Scomp$ is compact, one can find finitely many points $\px_1\ldots,\px_k\in\Scomp$ such that $\Scomp = \Uman_{\px_1} \cup\cdots\cup\Uman_{\px_k}$.
We claim that then $\boundR = \tfrac{1}{2}\min\{\boundR_{\px_1},\ldots,\boundR_{\px_k}\}$ satisfies the statement of our lemma.

Indeed, let $\uJmatr = \amatr{\uPmatr}{\uQmatr}{\uRmatr}{\uSmatr} \in \Mat{\bxdim+\brdim}{\axdim+\ardim}$ and $\px\in\Scomp$ a point for which the inequalities~\eqref{equ:uJ_J_dqdd} hold, that is
\begin{align*}
    &\nrm{\uPmatr-\Pmatr(\px)} < \boundR, &
    &\nrm{\uQmatr} < \boundQ, \\
    &\nrm{\uRmatr} < \boundR, &
    &\nrm{\uSmatr-\Smatr(\px)} < \boundR.
\end{align*}
Then by~\eqref{equ:PxPy_SxSy_estim} there exists $i\in\{1,\ldots,k\}$ such that
\begin{align*}
    \nrm{\Pmatr(\px)-\Pmatr(\px_i)} &< \tfrac{\boundR_{\px_i}}{2}, &
    \nrm{\Smatr(\px)-\Smatr(\px_i)} &< \tfrac{\boundR_{\px_i}}{2}.
\end{align*}
Hence
\[
    \nrm{\uPmatr-\Pmatr(\px_i)} \leq
    \nrm{\uPmatr-\Pmatr(\px)} + \nrm{\Pmatr(\px)-\Pmatr(\px_i)}  <
    \boundR + \tfrac{\boundR_{\px_i}}{2} \leq
    \tfrac{\boundR_{\px_i}}{2} + \tfrac{\boundR_{\px_i}}{2}  =
    \boundR_{\px_i},
\]
and similarly, $\nrm{\uSmatr-\Smatr(\px_i)} < \boundR_{\px_i}$.
Moreover, $\nrm{\uQmatr} < \boundQ$ and $\nrm{\uRmatr} < \boundR \leq \boundR_{\px_i}$.
Thus the inequalities~\eqref{equ:uJ_J_dqdd_xi} are satisfied, and therefore $\rank(\uJmatr) \geq a+b$.
\end{proof}

\subsection{Estimates for ranks of $\Chom[\dif]_{\hprm,\pt}$}
In order to obtain these estimates we will pass to local trivialization of our vector bundles, and therefore we assume that
\begin{align*}
    &\Apr\colon\ATotMan=\ASingMan\times\ALocFiber\to\ASingMan, &
    &\Bpr\colon\BTotMan=\BSingMan\times\BLocFiber\to\BSingMan
\end{align*}
are trivial vector bundles over some open subsets $\ASingMan \subset \ALocBase$ and $\BSingMan \subset \BLocBase$, and $\ScalNbh\subset\ATotMan$ is a \STL\ neighborhood of $\ASingMan$.
Denote $\mNF := \Ci{(\ScalNbh,\ASingMan)}{(\BTotMan,\BSingMan)}$.
Then for all $(\dif,\hprm,\pt)\in\mNF\times(0;+\infty)\times[0;1]$ we have a well defined map $\glHom[\dif]$ and $\Chom[\dif]_{\hprm,\pt}\in\mNF$ given by~\eqref{equ:homotopy_Ahom_} and~\eqref{equ:Chom_formula} respectively.
Our aim is to estimate the ranks of Jacobi matrices of $\Chom[\gdif]_{\hprm,\pt}$ for $\gdif$ sufficiently close to $\dif$ in the topology $\Wtr{\TOP}$ and sufficiently small $\hprm$.
To do that we need to specify precise formulas for $\Chom[\dif]$ in the context of trivial vector bundles.

\newcommand\constTheta{C_0}
\newcommand\constVSigma{C_1}
\newcommand\constH{C_2}
\newcommand\constLHDiff{C_3}
\newcommand\constHLDiff{C_4}
\newcommand\constDC{C_5}
\newcommand\constMB{C_6}

\paragraph{\bf Auxiliary functions}
First recall that
\begin{itemize}[nosep]
\item we fixed a $\Cinfty$ function $\mu\colon\bR\to[0;1]$ such that $\mu^{-1}(0)=(-\infty;\aSqConst]$ and $\mu^{-1}(1)=[\bSqConst;+\infty)$.
\item a squared fiberwise norm
\[
    \fnrm\colon\ATotMan\to[0;+\infty),
    \qquad
    \fnrm(\px,\pv)= \pv\Amap(\px)\trans{\pv},
\]
where $\Amap\colon\ASingMan\to\GLRn{\ardim}$ is some $\Cinfty$ map whose image consists of positively definite symmetric matrices; 
\item
and defined the following function $\phifunc\colon(0;+\infty)\times[0;1]\times\ScalNbh\to[0;1]$
\[
    \phifunc(\hprm,\pt,\px,\pv)
    = \pt + (1-\pt)\mu\Bigl(\frac{\fnrm(\px,\pv)}{\hprm^2}\Bigr)
    = \pt + (1-\pt)\mu\Bigl(\frac{\pv\Amap(\px)\trans{\pv}}{\hprm^2}\Bigr).
\]
\end{itemize}

In what follows we will assume that $\BComp\subset\ScalNbh$ is a compact \STL\ subset.
Also for each $\hprm>0$ put
\[
    \BComp_{\hprm}
        = \BComp \cap \cNbhx{\hprm}
        = \{ \pw\in\BComp \mid \fnrm[\pw] \leq \hprm^2\}.
\]

\begin{sublemma}\label{lm:estim_phifunc_der}
There exist $\constTheta>0$ depending on $\BComp$ such that for all $(\hprm,\pt)\in(0;\infty)\times[0;1]$ and $(\px,\pv)\in\BComp$ the following estimates hold:
\begin{align}
    \label{equ:phifunc_x}
    |\der{\phifunc}{\px}(\hprm,\pt,\px,\pv)|  &\leq \constTheta \cdot \tfrac{\nrm{\pv}^2}{\hprm^2}, \\
    \label{equ:phifunc_v}
    |\der{\phifunc}{\pv}(\hprm,\pt,\px,\pv)|  &\leq \constTheta \cdot \tfrac{\nrm{\pv}}{\hprm^2},
\end{align}
Moreover, there exists another a constant $\constVSigma>0$ depending on $\BComp$ such that
\begin{equation}\label{equ:v_div_bdelta}
    \nrm{\pv} \leq \constVSigma\,\hprm,
    \quad
    (\px,\pv)\in\BComp_{\hprm}.
\end{equation}
\end{sublemma}
\begin{proof}
By formulas for $\phifunc$ for each $i=1,\ldots,\axdim$ and $j=1,\ldots,\ardim$ we have that
\begin{align*}
    \der{\phifunc}{\pxi{i}} &=
        (1-\pt)\mu'\left(\tfrac{\pv\Amap(\px)\trans{\pv}}{\hprm^2}\right)
        \tfrac{\pv\der{\Amap}{\pxi{i}}\trans{\pv}}{\hprm^2}, & \
    \der{\phifunc}{\pvi{j}} &=
        (1-\pt)\mu'\left(\tfrac{\pv\Amap(\px)\trans{\pv}}{\hprm^2}\right)
        \tfrac{2\Amap^{j}(\px)\trans{\pv}}{\hprm^2},
\end{align*}
where $\Amap^{j}(\px)$ denotes the $j$-th row of the matrix $\Amap(\px)$, which implies~\eqref{equ:phifunc_x} and~\eqref{equ:phifunc_v}.

Moreover, since $\BComp$ is compact and $\Amap(\px)$ is a non-degenerate matrix, there exists $\constVSigma>0$ such that
\[
    \frac{\nrm{\pv}^2}{\constVSigma^2}  \leq \fnrm(\px,\pv) = \pv\Amap(\px)\trans{\pv},
    \qquad
    (\px,\pv)\in\BComp.
\]
Therefore, if $(\px,\pv)\in\BComp_{\hprm}$, that is $\fnrm(\px,\pv) \leq \hprm^2$, then $\nrm{\pv}^2 \leq \constVSigma^2 \hprm^2$, which proves~\eqref{equ:v_div_bdelta}.
\end{proof}

\paragraph{\bf Estimates}
Now let $\dif=(\aafunc, \bbfunc)\colon \ScalNbh \to \BSingMan\times\BLocFiber$ be a map from $\mNF$.
Then we have a $\Cinfty$ homotopy $\mAhom = (\haafunc,\hbbfunc)\colon[0;1]\times\ScalNbh\to\BTotMan$ between $\dif = \mAhom_1$ and its vertical derivative $\restrb{\vtang{\dif}}{\ScalNbh} = \mAhom_0$ given by~\eqref{equ:homotopy_Ahom_local_formulas}:
\begin{align}\label{equ:glHom_h__local_formulas}
    \haafunc(\pt, \px, \pv) &= \aafunc(\px, \pt\pv), &
    \hbbfunc(\pt, \px, \pv) &= \sum_{k=1}^{\ardim} \pvi{k} \smallint\limits_{0}^{1} \der{\bbfunc}{\pvi{k}}(\px,s\pt\pv) ds.
\end{align}
Note that then
\begin{align}
\label{equ:deriv:a_t}
\der{\haafunc}{\pt} &= \sum_{k=1}^{\ardim}\pvi{k} \der{\aafunc}{\pvi{k}}(\px, \pt\pv), \\
\label{equ:deriv:a_x}
\der{\haafunc}{\px} &= \der{\aafunc}{\px}(\px,\pt\pv), \\
\label{equ:deriv:a_vi}
\der{\haafunc}{\pvi{i}} &= \pt\der{\aafunc}{\pvi{i}}(\px,\pt\pv), \\
\label{equ:deriv:b_t}
\der{\hbbfunc}{\pt} &= \sum_{k=1}^{\ardim}\sum_{l=1}^{\ardim} \pvi{k}\pvi{l}
    \smallint\limits_{0}^{1} s\dertwo{\bbfunc}{\pvi{k}}{\pvi{l}}(\px,s\pt\pv) ds,
\\
\label{equ:deriv:b_x}
\der{\hbbfunc}{\px} &=
    \sum_{k=1}^{\ardim} \pvi{k}
    \smallint\limits_{0}^{1} \dertwo{\bbfunc}{\pvi{k}}{\px}(\px,s\pt\pv) ds,
\\
\label{equ:deriv:b_vi}
\der{\hbbfunc}{\pvi{i}} &=
    \smallint\limits_{0}^{1} \der{\bbfunc}{\pvi{i}}(\px,s\pt\pv) ds
    +
    \sum_{k=1}^{\ardim} \pvi{k}
    \smallint\limits_{0}^{1} s\pt \dertwo{\bbfunc}{\pvi{k}}{\pvi{i}}(\px,s\pt\pv) ds.
\end{align}
It will also be convenient to define the map $\Phimap\colon(0;+\infty)\times[0;1]\times\ATotMan\to[0;1]\times\ATotMan$
\[
    \Phimap(\hprm,\pt,\pw) = \bigl( \phifunc(\hprm,\pt,\pw), \pw \bigr).
\]
Then the corresponding two-parametric family of maps from $\mNF$ is given by
\begin{gather*}
    \mChom\colon(0;+\infty)\times[0;1]\times\ScalNbh \to \BTotMan,
    \\
    \mChom(\hprm,\pt,\pw) =
    \mAhom\bigl(\phifunc(\hprm,\pt,\pw), \pw\bigr) =
    \mAhom \circ \Phimap(\hprm,\pt,\pw),
\end{gather*}
see~\eqref{equ:Chom_formula}, where we omitted the upper index $\dif$ in $\mAhom$ and $\mChom$ to simplify notations.

Denote by
\begin{align*}
    \JM{\dif} &= \LMbl{\dif}, &
    \JM{\glHomt} &=
        \amatr{\mP[\glHomt]}
              {\mQ[\glHomt]}
              {\mR[\glHomt]}
              {\mS[\glHomt]},
    &
    \JM{\mChombt}
        &=
        \amatr{\mP[\mChombt]}
              {\mQ[\mChombt]}
              {\mR[\mChombt]}
              {\mS[\mChombt]}.
\end{align*}
the splittings of the Jacobi matrices of $\dif$, $\mAhomt$, $\mChombt$ into $(\bxdim,\brdim)\times(\axdim,\ardim)$-blocks.

\begin{sublemma}
There exists $\constLHDiff>0$ such that
\begin{equation}\label{equ:Ahom_t__xv_estim}
\begin{aligned}
\nrm{\mP[\glHomt](\px,\pv)-\mP[\dif](\px,0)}    & \leq \constLHDiff \cdot \nrm{\pv} \cdot \xnrm{\dif}{2}{\BComp},\\
\nrm{\mQ[\glHomt](\px,\pv)-\pt\mQ[\dif](\px,0)} & \leq \constLHDiff \cdot \nrm{\pv} \cdot \xnrm{\dif}{2}{\BComp},\\
\nrm{\mR[\glHomt](\px,\pv)}                     & \leq \constLHDiff \cdot \nrm{\pv} \cdot \xnrm{\dif}{3}{\BComp},\\
\nrm{\mS[\glHomt](\px,\pv)-\mS[\dif](\px,0)}    & \leq \constLHDiff \cdot \nrm{\pv} \cdot \xnrm{\dif}{3}{\BComp}.
\end{aligned}
\end{equation}
\end{sublemma}
\begin{proof}
\newcommand\lafunc{(\haafunc_{\pt})}
\newcommand\lbfunc{(\hbbfunc_{\pt})}
Define the coordinate functions of $\glHom_{\pt}$, $\pt\in[0;1]$, by
\begin{align*}
    \lafunc(\px,\pv) &= \haafunc(\pt,\px,\pv), &
    \lbfunc(\px,\pv) &= \hbbfunc(\pt,\px,\pv).
\end{align*}
Then there exists a constant $\constLHDiff>0$ such that
\newcommand\MASPACE{\hspace{2cm}}
\begin{align*}
&\nrm{\mP[\glHomt](\px,\pv)-\mP[\dif](\px,0)} \equiv
 \nrm{\mP[\glHomt](\px,\pv)-\mP[\glHomt](\px,0)} := \\[-1mm]
&\MASPACE =
    \nrm{\der{\lafunc}{\px}(\px,\pv)-\der{\lafunc}{\px}(\px,0)}
\stackrel{\eqref{equ:g_decomp_1}}{\leq}
\ardim\cdot \nrm{\pv} \cdot \xnrm{\der{\lafunc}{\px}}{1}{\BComp}
\stackrel{\eqref{equ:deriv:a_x}}{\leq} \\
&\MASPACE\leq \ardim\cdot \nrm{\pv} \cdot \xnrm{\der{\aafunc}{\px}}{1}{\BComp}
\leq \constLHDiff\cdot \nrm{\pv} \cdot \xnrm{\aafunc}{2}{\BComp},
\\[2mm]
&\nrm{\mQ[\glHomt](\px,\pv)-\pt\mQ[\dif](\px,0)} \equiv
\nrm{\mQ[\glHomt](\px,\pv)-\mQ[\glHomt](\px,0)} := \\[-2mm]
&\MASPACE = \nrm{\der{\lafunc}{\pv}(\px,\pv)-\der{\lafunc}{\pv}(\px,0)}
\leq \ardim\cdot \nrm{\pv} \cdot \xnrm{\der{\lafunc}{\pv}}{1}{\BComp}
\stackrel{\eqref{equ:deriv:a_vi}}{\leq} \\
&\MASPACE \leq \ardim^2\cdot \nrm{\pv} \cdot \xnrm{\der{\aafunc}{\pv}}{1}{\BComp}
\leq \constLHDiff\cdot \nrm{\pv} \cdot \xnrm{\aafunc}{2}{\BComp},
\\[2mm]
&\nrm{\mR[\glHomt](\px,\pv)} \equiv
\nrm{\mR[\glHomt](\px,\pv)-\mR[\glHomt](\px,0)} := \\[-2mm]
&\MASPACE = \nrm{\der{\lbfunc}{\px}(\px,\pv)-\der{\lbfunc}{\px}(\px,0)}
\leq \ardim\cdot \nrm{\pv} \cdot \xnrm{\der{\lbfunc}{\px}}{1}{\BComp}
\stackrel{\eqref{equ:deriv:b_x}}{\leq} \\
&\MASPACE\leq \constLHDiff\cdot \nrm{\pv} \cdot \xnrm{\bbfunc}{3}{\BComp},
\\[2mm]
&\nrm{\mS[\glHomt](\px,\pv)-\mS[\dif](\px,0)} \equiv
\nrm{\mS[\glHomt](\px,\pv)-\mS[\glHomt](\px,0)} := \\[-2mm]
&\MASPACE = \nrm{\der{\lbfunc}{\pv}(\px,\pv)-\der{\lbfunc}{\pv}(\px,0)}
\leq 2 \ardim\cdot \nrm{\pv} \cdot \xnrm{\der{\lbfunc}{\pv}}{1}{\BComp}
\stackrel{\eqref{equ:deriv:b_vi}}{\leq} \\
&\MASPACE \leq \constLHDiff\cdot \nrm{\pv} \cdot \xnrm{\bbfunc}{3}{\BComp}.
\end{align*}
These inequalities imply~\eqref{equ:Ahom_t__xv_estim}.
\end{proof}

\begin{sublemma}\label{lm:estim_Delta_C_min_A}
There exists $\constHLDiff>0$ depending only on $\BComp$ such that for all $(\hprm,\pt)\in(0;+\infty)\times[0;1]$ the following estimates for the blocks of the difference $\JM{\mChombt}-\JM{\mAhomt}$ hold:
\begin{equation}\label{equ:estim_PQRS}
\begin{aligned}
    \xnrm{\mP[{\Chom[\dif]_{\hprm,\pt}}](\px,\pv) - \mP[\glHomt](\px,\pv)}{0}{\BComp_{\hprm}} &\leq \constHLDiff \xnrm{\aafunc}{2}{\BComp_{\hprm}} \cdot \hprm,   \\
    \xnrm{\mQ[{\Chom[\dif]_{\hprm,\pt}}](\px,\pv) - \mQ[\glHomt](\px,\pv)}{0}{\BComp_{\hprm}} &\leq \constHLDiff \xnrm{\aafunc}{1}{\BComp_{\hprm}}, \\
    \xnrm{\mR[{\Chom[\dif]_{\hprm,\pt}}](\px,\pv) - \mR[\glHomt](\px,\pv)}{0}{\BComp_{\hprm}} &\leq \constHLDiff \xnrm{\bbfunc}{3}{\BComp_{\hprm}} \cdot \hprm^2, \\
    \xnrm{\mS[{\Chom[\dif]_{\hprm,\pt}}](\px,\pv) - \mS[\glHomt](\px,\pv)}{0}{\BComp_{\hprm}} &\leq \constHLDiff \xnrm{\bbfunc}{3}{\BComp_{\hprm}} \cdot \hprm.
\end{aligned}
\end{equation}
\end{sublemma}
\begin{proof}
Note that the Jacobi matrices of $\mAhom$ and $\Phimap$ split into the following $(\bxdim,\brdim)\times(1,\axdim,\ardim)$- and $(1,\axdim,\ardim)\times(\axdim,\ardim)$-blocks respectively:
\begin{align*}
    \JM{\glHom} &=
    \matrTwoThree
        {\der{\haafunc}{\pt}}{\der{\haafunc}{\px}}{\der{\haafunc}{\pv}}
        {\der{\hbbfunc}{\pt}}{\der{\hbbfunc}{\px}}{\der{\hbbfunc}{\pv}}
    =
    \matrTwoThree
        {\der{\haafunc}{\pt}}{\mP[\mAhomt]}{\mQ[\mAhomt]}
        {\der{\hbbfunc}{\pt}}{\mR[\mAhomt]}{\mS[\mAhomt]},
    &
    \JM{\Phimap[\hprm,\pt]} =
    \matrThreeTwo
        {\der{\phifunc}{\px}}{\der{\phifunc}{\pv}}
        {I_{\axdim}}{0}
        {0}{I_{\ardim}}.
\end{align*}
Hence the Jacobi matrix of $\mChombt=\mAhom\circ\Phimap[\hprm,\pt]$ splits into the $(\bxdim,\brdim)\times(\axdim,\ardim)$-blocks and is equal to
\newcommand\MMA{A}
\newcommand\MMB{B}
\begin{align*}
    \JM{\mChombt}(\px,\pv)&=
    \JM[{\Phimap[\hprm,\pt](\px,\pv)}]{\mAhom}
    \cdot
    \JM{\Phimap[\hprm,\pt]}(\px,\pv) = \\
    &=
    \AMATR
    {(\der{\haafunc}{\pt}\circ\Phimap[\hprm,\pt]) \cdot \der{\phifunc}{\px} \ + \ \der{\haafunc}{\px}\circ\Phimap[\hprm,\pt]}
    {(\der{\haafunc}{\pt}\circ\Phimap[\hprm,\pt]) \cdot \der{\phifunc}{\pv} \ + \ \der{\haafunc}{\pv}\circ\Phimap[\hprm,\pt]}
    {(\der{\hbbfunc}{\pt}\circ\Phimap[\hprm,\pt]) \cdot \der{\phifunc}{\px} \ + \ \der{\hbbfunc}{\px}\circ\Phimap[\hprm,\pt]}
    {(\der{\hbbfunc}{\pt}\circ\Phimap[\hprm,\pt]) \cdot \der{\phifunc}{\pv} \ + \ \der{\hbbfunc}{\pv}\circ\Phimap[\hprm,\pt]}.
\end{align*}
Emphasize that $\der{\haafunc}{\pt}$ is a \term{column} of size $\bxdim\times 1$, while $\der{\phifunc}{\px}$ is a \term{row} of size $1\times\axdim$, whence the product $(\der{\haafunc}{\pt}\circ\Phimap[\hprm,\pt]) \cdot \der{\phifunc}{\px}$ is a well-defined $\bxdim\times\axdim$-matrix.
Similar observations hold for other parts of these matrices.

Now, let $\pt\in[0;1]$ and $(\px,\pv)\in\BComp_{\hprm}$.
We need to estimate the corresponding blocks of the following difference:
\begin{align*}
&\JM{\mChombt}(\px,\pv) - \JM{\mAhomt}(\px,\pv) = \\ &=
    \underbrace{
        \AMATR
        {(\der{\haafunc}{\pt}\circ\Phimap[\hprm,\pt]) \cdot \der{\phifunc}{\px}}
        {(\der{\haafunc}{\pt}\circ\Phimap[\hprm,\pt]) \cdot \der{\phifunc}{\pv}}
        {(\der{\hbbfunc}{\pt}\circ\Phimap[\hprm,\pt]) \cdot \der{\phifunc}{\px}}
        {(\der{\hbbfunc}{\pt}\circ\Phimap[\hprm,\pt]) \cdot \der{\phifunc}{\pv}}
    }_{\MMA}
    +
    \underbrace{
        \AMATR
        {\der{\haafunc}{\px}\circ\Phimap[\hprm,\pt] - \der{\haafunc}{\px}}
        {\der{\haafunc}{\pv}\circ\Phimap[\hprm,\pt] - \der{\haafunc}{\pv}}
        {\der{\hbbfunc}{\px}\circ\Phimap[\hprm,\pt] - \der{\hbbfunc}{\px}}
        {\der{\hbbfunc}{\pv}\circ\Phimap[\hprm,\pt] - \der{\hbbfunc}{\pv}}
    }_{\MMB}.
\end{align*}
It will be convenient to consider the matrices $\MMA$ and $\MMB$ separately.

\emph{Matrix $\MMA$.}
Note that there exist a constant $\constDC>0$ depending only on $\BComp$ such that the blocks of the matrix $\MMA$ satisfy
\begin{align*}
    &|(\der{\haafunc}{\pt}\circ\Phimap[\hprm,\pt]) \cdot\der{\phifunc}{\px}|
        \stackrel{\eqref{equ:deriv:a_t}, \eqref{equ:phifunc_x}}{\leq}
    \constH \, \nrm{\pv}\, \xnrm{\aafunc}{1}{\BComp_{\hprm}} \cdot \constTheta \tfrac{\nrm{\pv}^2}{\hprm^2}
        \stackrel{\eqref{equ:v_div_bdelta}}{\leq}
    \constDC \xnrm{\aafunc}{1}{\BComp_{\hprm}} \cdot \hprm,
    \\
    &|(\der{\haafunc}{\pt}\circ\Phimap[\hprm,\pt]) \cdot \der{\phifunc}{\pv}|
        \stackrel{\eqref{equ:deriv:a_t}, \eqref{equ:phifunc_v}}{\leq}
    \constH \, \nrm{\pv}\, \xnrm{\aafunc}{1}{\BComp_{\hprm}} \cdot \constTheta \tfrac{\nrm{\pv}}{\hprm^2}
        \stackrel{\eqref{equ:v_div_bdelta}}{\leq}
    \constDC \xnrm{\aafunc}{1}{\BComp_{\hprm}},
    \\
    &|(\der{\hbbfunc}{\pt}\circ\Phimap[\hprm,\pt]) \cdot\der{\phifunc}{\px}|
        \stackrel{\eqref{equ:deriv:b_t}, \eqref{equ:phifunc_x}}{\leq}
    \constH \, \nrm{\pv}^2 \,\xnrm{\bbfunc}{2}{\BComp_{\hprm}} \cdot \constTheta \tfrac{\nrm{\pv}^2}{\hprm^2}
        \stackrel{\eqref{equ:v_div_bdelta}}{\leq}
    \constDC \cdot \xnrm{\bbfunc}{2}{\BComp_{\hprm}} \cdot \hprm^2,
    \\
    &|(\der{\hbbfunc}{\pt}\circ\Phimap[\hprm,\pt]) \cdot\der{\phifunc}{\pv}|
        \stackrel{\eqref{equ:deriv:b_t}, \eqref{equ:phifunc_v}}{\leq}
    \constH \, \nrm{\pv}^2 \,\xnrm{\bbfunc}{2}{\BComp_{\hprm}} \cdot \constTheta \tfrac{\nrm{\pv}}{\hprm^2}
        \stackrel{\eqref{equ:v_div_bdelta}}{\leq}
    \constDC \cdot \xnrm{\bbfunc}{2}{\BComp_{\hprm}} \cdot \hprm,
\end{align*}
where $\constTheta,\constVSigma>0$ are constants from Lemma~\ref{lm:estim_phifunc_der}.

\emph{Matrix $\MMB$.}
There exists $\constMB', \constMB>0$ such that
\begin{align*}
\nrm{\der{\haafunc}{\px}\circ\Phimap[\hprm,\pt] - \der{\haafunc}{\px}}
    &\stackrel{\eqref{equ:deriv:a_x}}{=}
\nrm{\der{\aafunc}{\px}(\px,\phifunc\pv) - \der{\aafunc}{\px}(\px,\pt\pv)}
    \stackrel{\eqref{equ:g_decomp_1}}{\leq} \\
    &\qquad \leq
\ardim\,\nrm{\pv}\,\xnrm{\aafunc}{2}{\BComp_{\hprm}}
    \leq
\constMB\xnrm{\aafunc}{2}{\BComp_{\hprm}}\cdot\hprm,
\\
\nrm{\der{\haafunc}{\pv}\circ\Phimap[\hprm,\pt] - \der{\haafunc}{\pv}}
    &\stackrel{\eqref{equ:deriv:a_vi}}{=}
\sum_{i=1}^{\ardim}\nrm{\phifunc\der{\aafunc}{\pvi{i}}(\px,\phifunc\pv) - \pt\der{\aafunc}{\pvi{i}}(\px,\pt\pv)}
    \leq
\constMB \xnrm{\aafunc}{1}{\BComp_{\hprm}},
\\
\nrm{\der{\hbbfunc}{\px}\circ\Phimap[\hprm,\pt] - \der{\hbbfunc}{\px}}
    &\stackrel{\eqref{equ:deriv:b_x}}{=}
\nrm{\sum_{k=1}^{\ardim} \pvi{k}
    \smallint\limits_{0}^{1} \dertwo{\bbfunc}{\pvi{k}}{\px}(\px,s\phifunc\pv) ds - \sum_{k=1}^{\ardim} \pvi{k}
    \smallint\limits_{0}^{1} \dertwo{\bbfunc}{\pvi{k}}{\px}(\px,s\pt\pv) ds}
    \stackrel{\eqref{equ:g_decomp_1}}{\leq} \\
    &\qquad \leq
\ardim\,\nrm{\pv}^2\,\xnrm{\bbfunc}{3}{\BComp_{\hprm}}
    \leq
\constMB\,\xnrm{\bbfunc}{3}{\BComp_{\hprm}}\cdot\hprm^2,
\\
\nrm{\der{\hbbfunc}{\pv}\circ\Phimap[\hprm,\pt] - \der{\hbbfunc}{\pv}}
    &\stackrel{\eqref{equ:deriv:b_vi}}{\leq}
    \sum_{i=1}^{\ardim}\Bigl(\nrm{
        \smallint\limits_{0}^{1} (
            \der{\bbfunc}{\pvi{i}}(\px,s\phifunc\pv) -
            \der{\bbfunc}{\pvi{i}}(\px,s\pt\pv)
        ) ds}
    + \\ & \qquad\qquad +
    \sum_{k=1}^{\ardim} \nrm{
    \pvi{k}
    \smallint\limits_{0}^{1} s
     \bigl(\phifunc \dertwo{\bbfunc}{\pvi{k}}{\pvi{i}}(\px,s\phifunc\pv)
     -
     \pt \dertwo{\bbfunc}{\pvi{k}}{\pvi{i}}(\px,s\pt\pv)\bigr) ds}
     \Bigr)
     \stackrel{\eqref{equ:g_decomp_1}}{\leq}\\
     &\qquad \leq
      \constMB'\,\nrm{\pv}\,\xnrm{\bbfunc}{3}{\BComp_{\hprm}}
     \leq \constMB\,\xnrm{\bbfunc}{3}{\BComp_{\hprm}}\cdot\hprm.
\end{align*}

Adding the above bounds for the corresponding blocks of $\MMA$ and $\MMB$ we see that there exists $\constHLDiff>0$ such that the estimates~\eqref{equ:estim_PQRS} holds:
\begin{align*}
    \xnrm{\mP[{\Chom[\dif]_{\hprm,\pt}}](\px,\pv) - \mP[\glHomt](\px,\pv)}{0}{\BComp_{\hprm}}
        &\leq
        \constDC \xnrm{\aafunc}{1}{\BComp_{\hprm}}\hprm
        +
        \constMB\xnrm{\aafunc}{2}{\BComp_{\hprm}}\hprm
        \leq
        \constHLDiff \xnrm{\aafunc}{2}{\BComp_{\hprm}}\hprm,
    \\
    \xnrm{\mQ[{\Chom[\dif]_{\hprm,\pt}}](\px,\pv) - \mQ[\glHomt](\px,\pv)}{0}{\BComp_{\hprm}}
    &\leq
    \constDC \xnrm{\aafunc}{1}{\BComp_{\hprm}} + \constMB \xnrm{\aafunc}{1}{\BComp_{\hprm}} \leq
    \constHLDiff \xnrm{\aafunc}{1}{\BComp_{\hprm}},
    \\
    \xnrm{\mR[{\Chom[\dif]_{\hprm,\pt}}](\px,\pv) - \mR[\glHomt](\px,\pv)}{0}{\BComp_{\hprm}}
    &\leq
    \constDC \xnrm{\bbfunc}{2}{\BComp_{\hprm}}  \hprm^2
    +
    \constMB\xnrm{\bbfunc}{3}{\BComp_{\hprm}}\hprm^2
    \leq
    \constHLDiff \xnrm{\bbfunc}{3}{\BComp_{\hprm}}  \hprm^2,
    \\
    \xnrm{\mS[{\Chom[\dif]_{\hprm,\pt}}](\px,\pv) - \mS[\glHomt](\px,\pv)}{0}{\BComp_{\hprm}}
    &\leq
    \constDC \xnrm{\bbfunc}{2}{\BComp_{\hprm}}  \hprm
    +
    \constMB\xnrm{\bbfunc}{3}{\BComp_{\hprm}}\hprm
    \leq
    \constHLDiff \xnrm{\bbfunc}{3}{\BComp_{\hprm}}  \hprm.
\end{align*}
Lemma is proved.
\end{proof}

\subsection{Proof of Lemma~\ref{lm:estim_PQRS_var}}
\label{sect:real_proof:lm:estim_PQRS_var}
Let $\dif\in\mNF$, $\BComp\subset\ScalNbh$ be a compact \STL\ subset, and $\hgeps>0$.
We should find two constants $\boundQ>0$ and $\constC>0$ such that for all $\gdif\in\mNF$ satisfying $\xnrm{\gdif-\dif}{\TOP}{\BComp}<\hgeps$ and all $(\hprm,\pt)\in(0;+\infty)\times[0;1]$ and $(\px,\pv)\in\BComp_{\hprm}$ the following conditions hold:
\begin{align*}
    &\nrm{\mP[{\Chombt[\gdif]}](\px,\pv) -
        \mP[\dif](\px,0)} < \constC \hprm + \hgeps,
    \quad &
    &\nrm{\mQ[{\Chombt[\gdif]}](\px,\pv)} < \boundQ,\\
    &\nrm{\mR[{\Chombt[\gdif]}](\px,\pv)} < \constC \hprm + \hgeps, &
    &\nrm{\mS[{\Chombt[\gdif]}](\px,\pv) -
        \mS[\dif](\px,0)} < \constC \hprm + \hgeps.
\end{align*}

First let us mention that there exists $C'>0$ depending only on the norming function $\fnrm$ such that for all $(\px,\pv)\in\BComp_{\hprm} = \BComp\cap \cNbhx{\hprm}$, we have that $\nrm{\pv} \leq C' \hprm$.

Now let $\gdif\in\mNF$, $(\hprm,\pt)\in(0;+\infty)\times[0;1]$, and $(\px,\pv)\in\ScalNbh$.
Then
\begin{align*}
    &\nrm{
        \mP[{\Chombt[\gdif]}](\px,\pv) -
        \mP[\dif](\px,0)
    } \leq
    \nrm{
        \mP[{\Chombt[\gdif]}](\px,\pv) -
        \mP[{\Ahom[\gdif]_{\pt}}](\px,\pv)
    } + \\
    &\hspace{2.5cm} +
    \nrm{
        \mP[{\Ahom[\gdif]_{\pt}}](\px,\pv) -
        \mP[\gdif](\px,0)
    } +
    \nrm{
        \mP[\gdif](\px,0)-
        \mP[\dif](\px,0)
    }.
\end{align*}

Now assume that $(\px,\pv)\in\BComp_{\hprm}$ and $\xnrm{\gdif-\dif}{\TOP}{\BComp} < \hgeps$.
Then $\xnrm{\gdif}{s}{\BComp} \leq \hgeps + \xnrm{\dif}{\TOP}{\BComp}$ for $0\leq s\leq \TOP$.
Hence we get the following estimates for each of the three summands of the right hand side of the above inequality:
\begin{align*}
&\nrm{\mP[{\Chombt[\gdif]}](\px,\pv) - \mP[{\Ahom[\gdif]_{\pt}}](\px,\pv) }
     \stackrel{\eqref{equ:estim_PQRS}}{\leq}
        C'' \xnrm{\gdif}{2}{\BComp} \cdot \hprm
        \leq
        C'' (\hgeps + \xnrm{\dif}{\TOP}{\BComp}) \cdot \hprm,
\\
&\nrm{\mP[{\Ahom[\gdif]_{\pt}}](\px,\pv) - \mP[\gdif](\px,0) }
     \stackrel{\eqref{equ:Ahom_t__xv_estim}}{\leq}
        C'' \nrm{\pv} \cdot \xnrm{\gdif}{2}{\BComp}
        \leq \\ & \hspace{5cm} \leq
        C'' C' \xnrm{\gdif}{2}{\BComp} \cdot \hprm
        \leq
        C'' C'(\hgeps+\xnrm{\dif}{\TOP}{\BComp}) \cdot \hprm,
\\
&\nrm{\mP[\gdif](\px,0)-\mP[\dif](\px,0)} \leq \xnrm{\gdif-\dif}{\TOP}{\BComp} < \hgeps,
\end{align*}
for some constant $C''>0$.
Therefore one can find $\constC>0$ such that
\[
\nrm{
        \mP[{\Chombt[\gdif]}](\px,\pv) -
        \mP[\dif](\px,0)
    }
    < \constC \hprm + \hgeps.
\]
The proof of other inequalities in~\eqref{equ:estim_H_h} is similar and is also based on~\eqref{equ:estim_PQRS} and~\eqref{equ:Ahom_t__xv_estim}.
We leave it to the reader.
\qed

This completes the proof of Lemma~\ref{lm:estim_PQRS_var} and thus the proof of Theorem~\ref{th:linearization_EMN_and_subspaces}.

\bibliographystyle{plainurl}
\bibliography{biblio}
\end{document}